\documentclass[preprint,11pt]{elsarticle} 
\makeatletter
\def\ps@pprintTitle{%
	\let\@oddhead\@empty
	\let\@evenhead\@empty
	\def\@oddfoot{}%
	\let\@evenfoot\@oddfoot}
\makeatother
\usepackage[left=15mm,right = 15mm,top = 15mm,bottom = 15mm]{geometry}
\usepackage[titletoc]{appendix}
\usepackage{CJKutf8}

\usepackage[table]{xcolor}
\usepackage{float}
\usepackage{subcaption}
\usepackage{amssymb,amsmath,amsthm,graphics,psfrag,graphicx,color,fancyhdr}
\usepackage{algorithmic}
\usepackage[color=yellow]{todonotes}

\usepackage{appendix}
\usepackage{verbatim}
\usepackage{blindtext}
\usepackage{hyperref}
\usepackage{cleveref}
\usepackage{stmaryrd}
\usepackage{soul}
\usepackage{sidecap}
\usepackage{tikz}
\usepackage{rotating}
\usepackage{booktabs}
\usepackage{comment}
\usepackage{multirow}
\usepackage{changepage}
\usepackage{float}
\usepackage{wrapfig}
\usepackage{stmaryrd}
\usepackage{pgfplots}
\usepackage{framed}
\usepackage{float}

\usepgfplotslibrary{fillbetween}

\crefformat{appendix}{#2#1#3}

\definecolor{lightblue}{rgb}{0.32,0.45,0.90}
\definecolor{lightgreen}{rgb}{0.42,0.7,0.40}
\numberwithin{equation}{section}
\numberwithin{figure}{section}
\hypersetup{colorlinks=true,linkcolor = blue,citecolor=blue}
\usetikzlibrary{shapes,arrows}
\usetikzlibrary{decorations.markings}
\numberwithin{figure}{section}
\renewcommand{\div}{\operatorname*{div}}

\def\b{\boldsymbol}

\makeatletter
\newcommand{\vast}{\bBigg@{4}}
\newcommand{\Vast}{\bBigg@{5}}

\makeatother\def\b{\boldsymbol}

\colorlet{cgray}{gray!20!white}

\theoremstyle{definition}
\newtheorem{definition}{Definition}[section]

\newtheorem{method}{Method}[section]

\newtheorem{assumption}{Assumption}[section]

\newtheorem{theorem}{Theorem}[section]
\theoremstyle{remark}
\newtheorem*{remark}{Remark}

\newtheorem{example}{Example}[section]
\newtheorem{lemma}{Lemma}[section]

\tikzset{every label/.style={font=\footnotesize,inner sep=1pt}}

\allowdisplaybreaks

\pgfplotsset{compat=1.18} 

\begin{document}
\begin{frontmatter}
\title{Strong Stability Preservation for Stochastic Partial Differential Equations}
\author[inst4]{J. Woodfield}

\affiliation[inst4]{organization={Department of Mathematics, Imperial College London},
addressline={South Kensington Campus}, 
city={London},
postcode={SW7 2AZ}, 
country={United Kingdom}}

\begin{abstract}
This paper extends deterministic notions of Strong Stability Preservation (SSP) to the stochastic setting, enabling nonlinearly stable numerical solutions to stochastic differential equations (SDEs) and stochastic partial differential equations (SPDEs) with pathwise solutions that remain unconditionally bounded. This approach may offer modelling advantages in data assimilation, particularly when the signal or data is a realization of an SPDE or PDE with a monotonicity property. 
\end{abstract}
\end{frontmatter}
%\tableofcontents

% \begin{abstract}
% In this paper, deterministic notions of Strong Stability Preservation are translated into the Stochastic setting, allowing stable solutions to SPDEs whose pathwise solutions are unconditionally bounded.

% This will be shown useful in the context of numerical methods employing Riemann solvers. This will shown to be potentially advantageous in the context of data assimilation in the instance when the data arises from a physical process with monotone properties.
% %signal is a realisation of an SPDE or PDE with a monotone property. 
% %when using a particle filter whose forward model can be described by an SPDE with a monotone property. 
% %When the data arises from a physical process with monotone properties (such as stochastic coarse grain/model reduction by transport noise) or a realisation of the stochastic forward model with nonlinearity preservation. 
% \end{abstract}

%% main text
\tableofcontents

\section{Introduction}

\subsection{Motivation}
Many important partial differential equations (PDEs) have solutions with monotonic and contractivity properties, typically defined with respect to some convex functional, norm or semi-norm.  For example, weak entropy solutions of the one-dimensional Burgers equation are contractive with respect to the total variation semi-norm, $|| u(t^{n+1})||_{BV(\mathbb{T}^1)}\leq || u(t^{n})||_{BV(\mathbb{T}^1)}$.
Solutions to 2D Euler's equation have bounded vorticity, $||\omega(t^{n+1})||_{L^{\infty}(\mathbb{T}^2)} \leq ||\omega(t^{n})||_{L^{\infty}(\mathbb{T}^2)}$. The preservation of positivity is a crucially important monotonic property of the advection equation. Furthermore, many different types of monotonicity exist and are of theoretical and practical interest, including local maximum principles, entropy conditions, and contractive behaviour between two solutions $||q_1(t^{n+1})-q_2(t^{n+1})||\leq ||q_1(t^n)-q_2(t^n)||$. Deterministic Strong Stability Preserving
(SSP) methodology was developed to numerically preserve discrete versions of these nonlinear properties.

The solutions to stochastic partial differential equations SPDEs can also have nonlinear properties that are desirable to preserve at a discrete level. SPDEs are an increasingly popular tool used to model physical phenomena that require additional parametrisation. In this paper we investigate strong stability preserving Runge-Kutta methods for stochastic equations, with a focus on constructing unconditionally contractive (monotone) numerical schemes for solving SPDEs. To the best of our knowledge, the stochastic extension/application of SSP time-stepping has not yet been explored. The theory presented here consists of three key components: a spatial scheme that preserves the desired monotone property, a deterministic SSP time-stepping method adapted for the stochastic setting, and appropriately bounded driving increments.
\subsection{Literature}
% Runge-Kutta methods remain an important solving strategy for the numerical solution of ODEs due to their computational effectiveness, in addition to geometric, algebraic and group characterisations \cite{butcher1987numerical,butcher1963coefficients,butcher1972algebraic,butcher1996runge,hairer2006geometric}. Runge-Kutta methods have a history of different notions of stability,  \cite{butcher1996history} highlights some historical development of A-stability, L-stability, B-stability. More recently, the notion of Strong Stability Preservation and contractivity has received much attention due to the ability to preserve nonlinear properties. 

Strong stability preserving Runge Kutta (SSPRK) methods are Runge-Kutta methods capable of being written as a convex combination of forward Euler schemes, allowing nonlinear stability properties to be inherited from the forward Euler flow map. The optimal convex combination is referred to as the Shu-Osher \cite{shu1988efficient} representation. It is characterised by a stepsize condition referred to as either the radius of monotonicity or the SSP coefficient. The radius of monotonicity is the largest timestep for which Kraaijevanger's conditions hold \cite{kraaijevanger1991contractivity} and characterises the effectiveness of an RK-scheme to retain monotonic properties. 

For a general review of deterministic SSP methods, we refer to the review paper \cite{gottlieb2005high} and the references therein. For a selective review of some of the important theoretical developments in the deterministic SSP literature, we refer to 
\cite{ferracina2005stepsize,ferracina2004stepsize,higueras2005representations,higueras2004strong,ketcheson2004algebraic,kraaijevanger1991contractivity,shu1988efficient}, where amongst other important developments, the notion of contractivity and the SSP property are unified under different convex functionals. The extension to Additive Runge-Kutta schemes is made in \cite{higueras2006strong} following \cite{kraaijevanger1991contractivity}, and then similarly extended to include Generalised Additive Runge-Kutta schemes in \cite{sandu2015generalized}. Implicit systems required an extension of the original Shu-Osher representation and is discussed in \cite{bolley1978conservation,spijker1983contractivity,ferracina2004stepsize,higueras2005representations,higueras2004strong}. To circumnavigate specific order-barriers one can use perturbed RK methods viewed as Additive Runge-Kutta methods in \cite{higueras2018optimal,kraaijevanger1991contractivity}. For large systems, low storage high effective radius of monotonicity methods have been proven effective see \cite{williamson1980low,ketcheson2008highly,higueras2019new}. Internal and external monotonicity is discussed in \cite{hundsdorfer2011boundedness}. For an adaptation of the SSP theory to preserve nonlinear properties in the setting of linear operators see \cite{conde2017implicit}.

 For Stochastic Runge-Kutta methods, convergence to Itô and Stratonovich systems is well established \cite{rossler2010runge,burrage2004numerical,Kloeden_Platen_1992,milstein2013numerical} and concepts such as A-Stability have been well developed for both Itô and Stratonovich cases \cite{Kloeden_Platen_1992}.  Whilst SSP theory has not been discussed in the stochastic setting. The related notion of positivity preservation has been studied in the context of SDEs, and we highlight some recent contributions. In \cite{scalone2022positivity} positivity preserving theta methods are proposed, using bounded increments is not discussed. In \cite{davila2005numerical} schemes are constructed using the exact solutions of a linearised form of the SDE on short time windows, as to construct positivity-preserving numerical schemes. In \cite{kelly2018adaptive} a time-adaptive approach is taken, and positivity preservation is discussed in a probabilistic sense. In \cite{kiouvrekis2025domain} positivity preservation of a specific SDE is attained through the use of an exponential type integrator. In \cite{higham2012convergence} non-negativity of a Milstein-type scheme is shown in a probabilistic sense. In \cite{beyn2017stochastic} Stochastic C-stability and B-consistency are established for Milstein-type schemes, what is referred to as satisfying a global monotonicity condition refers to the L2 stability bounds on the drift and diffusion, and not the notion of monotonicity discussed in this paper. More recently in 2023 Fang Zhao and Zhao \cite{fang2023strong} investigated the SSP property for forward-backwards SDEs for multistep schemes, but defined a notion of strong stability preservation in expectation. Closer to the results and aims as in this paper, \cite{lei2023strong} show that that a logarithmic transformed truncated Euler–Maruyama method can attain positive solutions of SDEs with strong order 1/2 convergence. This paper considers strong stability preservation pathwise, and considers schemes with weak convergence order 1 and mean square (strong) order convergence order 1/2.

\subsection{Background: Introduction to deterministic SSP integration.}

 PDEs are often discretised as a system of ordinary differential equations (ODEs), this is common after discretisation in space using the method of lines, where the resulting ODEs take the form 
\begin{align}
    \frac{d \b q}{dt} = \b f(\b q), \quad \b q(0) = \b q_0, \quad \b q \in \mathbb{R}^n. \label{eq: ode system}
\end{align}
Strong Stability Preserving (SSP) integrators emerged out of a desire to ensure that numerical methods discretising the system of ODEs satisfy a discrete contractive property analogue to that of the continuum model. One approach commonly taken in the Runge-Kutta framework is to assume that the forward Euler flow map can be proven monotone under an arbitrary convex functional, norm or semi-norm $||\cdot ||$ for all timesteps less than a critical value $\tau_{0}$, i.e. it is assumed 
\begin{align}
||\b q^{n+1}|| = ||\operatorname{FE}(\b q^n,\Delta t, \b f) || = ||\b q^n+ \Delta t \b f(\b q^n) || \leq ||\b q^{n}||, \quad \forall \Delta t \leq \tau_{0}.
\end{align}

It is then common to write a Runge-Kutta method, as a convex combination of forward Euler steps. For example, the Heun scheme admits the following representation
\begin{align}
\text{Heun}(\b q^n,\Delta t, \b f) = 1/2 \b q^n + 1/2 \operatorname{FE}\left(\operatorname{FE}(\b q^n,\Delta t, \b f),\Delta t, \b f \right).
\end{align}

Since it is assumed proven for forward Euler, one can show that using the convex combination representation, the same monotone property holds for the HEUN scheme as follows
\begin{align}
||\b q^{n+1}|| = ||\text{Heun}(\b q^n,\Delta t, \b f)|| \leq 1/2 ||\b q^n || + 1/2 || \operatorname{FE}(\operatorname{FE}(\b q^n,\Delta t, \b f),\Delta t, \b f)|| \leq 1 ||\b q^n||, \quad \forall \Delta t \leq \tau_0.
\end{align}

In this paper, we note that the notion of strong stability can be extended into the stochastic setting for some widely used stochastic integrators. More specifically we formalise and generalise an argument of the following form. Should the Euler-Maruyama map be SSP in the following sense $||\operatorname{EM}(\b q^n,\Delta t, \b f,G,\b \Delta W)||\leq ||\b q^n||$ (requiring bounded increments). 
Then a stochastic Heun scheme with a particular convex representation can be bounded as follows
\begin{align}
||\b q^{n+1}|| = ||\text{Stochastic-Heun}(\b q^n,\Delta t, \b f,G,\Delta W)|| \leq 1/2 ||\b q^n || + 1/2 || \operatorname{EM}^2(\b q^n,\Delta t, \b f,G, \Delta \b W)|| \leq ||\b q^n||,
\end{align}
and could inherit nonlinear monotonic properties from the Euler Maruyama scheme.

\subsection{Main contributions and outline of the paper.}
In this work, we establish that three conditions are sufficient for a stochastic Runge-Kutta scheme to be SSP. 
\begin{enumerate}
    \item A provably monotonic numerical method for the Euler Maruyama scheme.\label{item 1}
    \item Bounded increments.
    \item The application of an SSP-RK scheme with a nonzero radius of monotonicity.
\end{enumerate} 
Where condition 2 is required for condition 1.
The extension to Additive Runge-Kutta schemes and Generalised Additive Runge-Kutta schemes is also considered, allowing condition 1 to be relaxed to the following strictly weaker condition. One requires a provably monotonic numerical method for the Forward Euler scheme and a diffusion-only Euler Maruyama scheme. 

In \cref{Sec:Strong stability preserving stochastic integrators} we define what it means for some stochastic RK, ARK and GARK schemes to be SSP. Methods of proofs follow the deterministic literature \cite{kraaijevanger1991contractivity,higueras2005representations}, but are applied in the stochastic setting. In \cref{sec:convergence} we discuss the convergence implications of using bounded increments. More specifically we prove a variant of a theorem stated by Milstein and Tretyakov \cite{milstein2002numerical,milstein2004stochastic}, where normal increments can be truncated and one can still attain strong order of convergence. The outline of proof is sketched in \cite{milstein2002numerical} and an in-depth derivation is provided for the backward Euler Maruyama scheme is found in \cite{milstein2002numerical}. Whilst this result is known, a detailed proof for the Additive Runge-Kutta method involving regularity conditions has been omitted in the literature. 

% Precise literature review:
% \begin{enumerate}
% \item The strong convergence of an Implicit Euler Maruyama scheme to the Itô system is established in \cite{milstein2002numerical}.
% \item In \cite{milstein2004stochastic}, a scheme similar to the SARK \cref{method: Stochastic Additive Runge-Kutta} is proposed, it differs by having additional terms to correct the scheme to converge to the Itô form. \Cref{method: Stochastic Additive Runge-Kutta} differs by being entirely derivative-free and converging to a Rumelin-SDE, a solution dependent on the coefficients in the SARK method. 
% \item 
% In \cite{milstein2004stochastic} it is claimed that using the bounded increments defined in \cref{eq:bounded increments} one can prove the mean square convergence of a SARK scheme (similar to that of \cref{method: Stochastic Additive Runge-Kutta}), the details of the proof are omitted. We provide some of the details in this work.
% \item The convergence of the SARK is established towards an SDE we call the Ruemelin SDE using the same techniques outlined in \cite{milstein2002numerical,milstein2004stochastic}, we elaborate more regarding specific details in the proof. We then include an outline of the (extension \cite{sandu2015generalized}) to include a Stochastic Generalised Additive Runge-Kutta \cref{method: Stochastic Generalised Additive Runge-Kutta}.
% \end{enumerate}

In \cref{sec:Practical Methods} we introduce some practical methods. In \cref{sec:example 1a burgers}, we demonstrate all three conditions are of practical value to creating a monotonic solution. In \cref{sec:Example 1b}, we numerically demonstrate that these conditions may not be strictly necessary. In \cref{sec:Example 1c: 2D Advection} we show that when using multidimensional slope limiters with SSP integration with bounded increments it is possible to produce range-bounded solutions to a stochastic advection equation. In \cref{Example 1d: Incompressible Euler} we show one-dimensional slope limiters can be used with an SSP method with bounded increments to produce range-bounded solutions to Euler's equation with transport-type noise. In \cref{sec:example:1e operator splitting GARK}, we illustrate that when an EM scheme may not be proven monotone one can employ SARK and SGARK methods to construct bounded solutions to a Burgers equation with deformational spatially varying compressible transport noise.

\section{Strong stability preserving stochastic integrators}\label{Sec:Strong stability preserving stochastic integrators}
\subsection{Equations}
We consider the $d$-dimensional stochastic process $\b q(t)$, defined as the solution
of the Stratonovich SDE
\begin{align}
\b q(t) &= \b q(0)  + \int_{0}^{t}\b f (\b q) ds+\sum_{p=1}^{P}\int_{0}^{t} \b g_p(\b q)\circ d  W_s^p; \quad \b q (0)= \b q_0.\label{stratonovich}
\end{align}
over $[0,t]$ for some $t\in \mathbb{R}^{>0}$.
Here the initial state is denoted $\b q_0 \in \mathbb{R}^d$ at $t_0 = 0$, the vector valued drift function is denoted $\b f(\b q):\mathbb{R}^d \rightarrow \mathbb{R}^d$, the matrix valued diffusion is denoted $G(\b q):\mathbb{R}^d \rightarrow \mathbb{R}^{d\times P}$ whose $p$th column is denoted $\b g_p(\b q):\mathbb{R}^d \rightarrow \mathbb{R}^{d}$ for $ p=\lbrace 1,...,P\rbrace$. Finally, we denote a $P$-dimensional $\mathcal{F}_t$-adapted real-valued vector Wiener process as $\left(\b W_t\right)_{t \geqslant 0}=\left( ( W_t^1, \ldots, W_t^P)^{T}\right)_{t \geqslant 0}$, such that the components are scalar adapted independent Wiener processes taking values on a filtered probability space $(\Omega, \mathcal{F}, \mathbb{P};\left(\mathcal{F}_t\right)_{t \geqslant 0})$ with the usual conditions\begin{footnote}{The usual conditions for a filtered probability space. The filtration (a collection of increasing sub $\sigma$-algebras of $\mathcal{F}$) is assumed right continuous (in the sense that $\mathcal{F}_t= \lbrace A \in \mathcal{F}_0: A\in \mathcal{F}_s, \forall s > t \rbrace $) the measure space $(\Omega, \mathcal{F}, \mathbb{P})$ is assumed complete, and it is also assumed that if $\mathbb{P}(A)=0$, then $A\in \mathcal{F}_0$ (see \cite{lord2014introduction,oksendal2013stochastic}).}.
\end{footnote}The $\circ$ denotes Stratonovich integration. These $d$-dimensional systems often arise after the spatial discretisation of stochastic partial differential equations in the method of lines framework. System \cref{stratonovich} can be denoted with the following ``differential" shorthand 
\begin{align}
d\b q = \b f (\b q) dt+  G(\b q) \circ d\b W(t).
\end{align}
% or component-wise with the following expression
% \begin{align}
% q_i(t) = q_i(0) + \int_{0}^{t} f_i(\b q)ds + \sum_{p=1}^{P}\int_{0}^{t} {g}_{i,p}(\b q)\circ dW^{p}_{s};\quad \forall i\in \lbrace 1,...,d \rbrace.
% \end{align}

We shall also refer to the following Itô system
\begin{align}
\b q^{I}(t) &= \b q^{I}(0)  + \int_{0}^{t}\b f (\b q^{I}) ds+\sum_{p=1}^{P}\int_{0}^{t} \b g_p(\b q^{I}) d  W_s^p; \quad \b q^{I} (0)= \b q_0.\label{eq:Ito}
\end{align}
An SDE, whose stochastic integration is understood in the Itô sense \cite{oksendal2013stochastic,karatzas2014brownian}, whose shorthand is \begin{align}
d \b q^I = \b f(\b q^I)dt + G(\b q^I)d\b W_t.
\end{align}

The Stratonovich SDE \cref{stratonovich} can be written as an Itô SDE with a modified drift
\begin{align}
\b q(t) &= \b q(0)  + \int_{0}^{t} \underline{\b f} (\b q) ds+\sum_{p=1}^{P}\int_{0}^{t} \b g_p(\b q) d  W_s^p; \quad \b q (0)= \b q_0,\label{ito form}
\end{align}
whose drift differs by the Itô Stratonovich correction.

% , defined through the quadratic covariation and component-wise taking the following form
% \begin{align}
% \underline{f}^i = f^i + \frac{1}{2}\sum_{p=1}^{P}\sum_{k=1}^{d} g_{j,p}\frac{\partial g_{i,p}}{\partial q_{k}}.   
% \end{align}

\subsection{Well posedness}\label{sec:well posedness estimates}
The existence and uniqueness of the Stratonovich ODE system is well known and requires Lipschitz continuity and linear growth bounds on the Itô Stratonovich corrected drift and diffusion.  These are expressed below, where we have used $L$ to define an arbitrary constant
\begin{align}
 \left\|\underline{\b f} (\b q ) \right\|_2^2  \leq L\left(1+\|\b q \|_2^2\right), \quad
\|G(\b q )\|_{\mathrm{F}}^2 \leq L\left(1+\|\b q \|_2^2\right), \quad \forall \b q \in \mathbb{R}^d,\\
 \left\| \underline{\b f} \left(\b q_1\right)-\underline{\b f}\left(\b q_2\right)\right\|_2 \leq L\left\|\b q_1-\b q_2\right\|_2, \quad
 \left\|G\left(\b q_1\right)-G\left(\b q_2\right)\right\|_{\mathrm{F}} \leq L\left\|\b q_1- \b q_2\right\|_2, \quad \forall \b q_1, \b q_2 \in \mathbb{R}^d .
\end{align}
Whereas the Itô system requires the same conditions but for the uncorrected drift $\b f$ (permitting lower regularity on $G$). 
In \cite{lord2014introduction}, the existence and uniqueness of the Itô system is established using the contraction mapping theorem in the Banach space $(\mathcal{H}_{2, T},\|\cdot\|_{\mathcal{H}_{2, T}})$, the set of $\mathbb{R}^d$-valued predictable processes $\{\b q(t): t \in[0, T]\}$ such that
$
\|\b q \|_{\mathcal{H}_{2, T}}:=\sup _{t \in[0, T]} \mathbb{E}\left[\|\b q(t)\|_2^2\right]^{1 / 2}<\infty.$ Existence and pathwise uniqueness of strong solutions is also established in \cite{kloeden1992stochastic} under the assumptions of joint measurably of both diffusion and drift, initially bounded moments and the Lipschitz and linear growth bounds (page 128 \cite{kloeden1992stochastic}). However, rather than bounding the Frobenius norm they (equivalently) require that 
\begin{align}
||\b g_p||_{2}^2 \leq L\left(1+||\b q||_2^2\right), \quad ||\b g_p(\b q_1)-\b g_p(\b q_2)||_{2}^2 \leq L||\b q_1 - \b q_2||_{2}^2, \quad \forall p\in \lbrace 1,...,P\rbrace.
\end{align}
Some of these inequalities are used later in this paper when establishing estimates required for the Fundamental theorem of mean square convergence \cite{milstein2002numerical}.

\subsection{Methods}\label{sec:methods}
Similar in analogy to how the deterministic SSP literature builds schemes in terms of the forward Euler flow map, we will discuss particular stochastic Runge-Kutta methods built upon convex combinations of the Euler Maruyama scheme defined below. 
\begin{method}[Euler-Maruyama]\label{method:EM}
The Euler-Maruyama scheme is a numerical flow map \newline $EM(\b q^n,\Delta t,\b f, G, \Delta \b S): \b q^{n} \in \mathbb{R}^{d} \rightarrow  \b q^{n+1}\in \mathbb{R}^{d}$, advancing the current state variable $\b q^n$ by a increment of time $\Delta t \in \mathbb{R}^{>0}$, as follows
\begin{align}
\b q^{n+1} &= \b q^n + \b f(\b q^{n}) \Delta t + G(\b q^n) \Delta\b S, \label{eq:EM}
\end{align}
where $\Delta \b S$ is a realisation of a $P$ dimensional signal (typically drawn from a distribution) and $G(\b q) \Delta\b S$, is understood as a matrix vector product $(\mathbb{R}^{d\times P},\mathbb{R}^{P})\rightarrow \mathbb{R}^{d}$, defined by $G(\b q) \Delta\b S = \sum_{p=1}^{p=P}\b g_p(\b q^n) \Delta S^{p}.$ 
\end{method}
Typically, in the context of Brownian motion $\Delta W^p = \int_{t^{n}}^{t^{n+1}}dW^p(s) = W^p(t^{n+1}) - W^p(t^{n}) \sim \sqrt{\Delta  t} N(0,1)$, one samples from a scaled normally distribution, and one can establish convergence to the Itô equation \cref{eq:Ito}. However, this is not a necessary condition for either weak or strong convergence of the numerical scheme. Sampling from random variables with bounded increments will be essential for the theoretical developments in this paper and allows us to define (in \cref{def:Strong Stability Preservation of Euler-Maruyama scheme}) an EM scheme with an SSP property. We postpone how using bounded increments affects the convergence until later in \cref{sec:convergence}.

\begin{definition}[Strong Stability Preservation of Euler-Maruyama scheme.]\label{def:Strong Stability Preservation of Euler-Maruyama scheme} The Euler-Maruyama scheme, will be defined to be SSP  with respect to the arbitrary convex semi-norm $||\cdot ||$, if there exists a critical timestep $\tau_{0} \geq 0$, such that for all positive timesteps smaller than this $\forall \Delta t \in(0,\tau_{0})$ and for all possible sampled increments $\Delta S^{p}$. The following property holds
\begin{align}
||\b q^{n+1}|| &= ||\b q^n + \b f(\b q^{n}) \Delta t + \sum_{p=1}^{P} \b g_p(\b q^{n}) \Delta S^{p} || \leq ||\b q^{n}||. \label{eq: EMSSP condition} 
\end{align}

This particular definition requires  $\Delta S^p$, to be bounded.

\end{definition}

We now consider a particular Stochastic Runge-Kutta \cref{method: Stochastic Runge-Kutta} and subsequently define an SSP property (\cref{def:SSP-SRK}).

\begin{method}[Stochastic Runge-Kutta]\label{method: Stochastic Runge-Kutta}
An $s$-stage  Stochastic RK method is defined by the Butcher Tableau $A\in \mathbb{R}^{s\times s}$, $b\in \mathbb{R}^s$. Such that the old value $\b q^n$ is updated by the use of $s$ internal substages $k^i$ as follows
\begin{align}
\b q^{n+1}= \b q^n+\Delta t \sum_{i=1}^s b_i \b f(\b k^i) + \sum_{i=1}^s b_i  G(\b k^i)\Delta \b S, \quad \text{where} \quad \b k^i= \b q^n + \Delta t \sum_{j=1}^s a_{i j} \b f(\b k^j) + \sum_{j=1}^s a_{i j}G(\b k^j)\Delta \b S. \label{eq:SSSPRK method}
\end{align}
The increments $\Delta\b S$ are sampled once at $t^n$, and contracted against the diffusion matrix $G$, such that the same increments are reused at different sub-stages of the RK method. 
\end{method}

\begin{definition}[SSP-SRK]\label{def:SSP-SRK}
The Stochastic Runge-Kutta \cref{method: Stochastic Runge-Kutta} is said to be strong stability preserving with a radius of monotonicity $C$ if the numerical solution $\b q^{n+1}$ generated by the numerical method in \cref{eq:SSSPRK method} satisfies the following property 
\begin{align}
    ||\b q^{n+1}|| \leq ||\b q^n||, \quad \forall \Delta t \leq C \tau_0, \quad  \forall \Delta S_i^{p}.
\end{align}
Where it is assumed that there exists a critical timestep $\tau_0$ such that the Euler Maruyama scheme is SSP as specified in \cref{def:Strong Stability Preservation of Euler-Maruyama scheme}. \Cref{method: Stochastic Runge-Kutta} is said to be internally strong stability preserving with radius of monotonicity $C$ when the substages satisfy $||\b q^{n+1}||\leq ||\b k^{i}|| \leq ||\b k^{i-1} ||  \leq ||\b q^n||$, for all $\Delta t \leq C\tau_0$, for all sampled increments $\Delta \b S$, for all $i\in \lbrace 1,...,s\rbrace$.
\end{definition}

\begin{remark}
The convex combination (Shu-Osher) representation of a Runge-Kutta method enables the preservation of monotonic properties that are strictly stronger than contractivity for any convex functional. For example, the Shu-Osher representation allows the enforcement of local maximum principles at each substage in a Runge Kutta method. 
\end{remark}

Throughout this work we will adopt the compact notation from \cite{higueras2005representations}, where a butcher tableau $(A,\b b)$ is turned into the following matrix
\begin{align}
\mathbb{A}=
\begin{pmatrix}
A & 0 \\
\b b^T & 0
\end{pmatrix} \in \mathbb{R}^{(s+1)\times (s+1)}. 
\end{align}

\begin{theorem}[Stochastic Runge-Kutta \cref{method: Stochastic Runge-Kutta} is SSP with radius of monotonicity $R(\mathbb{A})$]
Assuming there exists a critical timestep $\tau_0$ such that EM is assumed contractive with respect to a convex functional, then \cref{method: Stochastic Runge-Kutta} is also contractive with respect to the same convex functional with timestep condition $\Delta t \leq R(\mathbb{A})\tau_0$, where 
\begin{align}
R(\mathbb{A})=\max \{r \mid r \geq 0;\forall s\in[0,r], (I+s\mathbb{A})\in \operatorname{GL}(s+1, \mathbb{R}),(I+s\mathbb{A})^{-1}s\mathbb{A}\geq 0,(I+s\mathbb{A})^{-1}\b e_s\geq 0  \} . \label{eq: Kraiijevanger conditions}
\end{align}
\end{theorem}
The conditions in \cref{eq: Kraiijevanger conditions} are algebraically equivalent to the Kraaijevanger \cite{kraaijevanger1991contractivity} conditions \cite{higueras2005representations,higueras2004strong} on the butcher tableau $(A,b)$, namely $I+sA$ is nonsingular, and $1+s b^T(I+s A)^{-1} e_s \geq 0$, $A(I+sA)^{-1} \geq 0$, $b^T(I+s A)^{-1} \geq 0$, and $(I+sA)^{-1} e_s \geq 0$.

\begin{proof}
By identifying the vector-valued function $\mathfrak{F}(\b q): \mathbb{R}^n \mapsto \mathbb{R}^n$ associated with
\begin{align}
\mathfrak{F}(\b q) := \b f(\b q^n) + G(\b q^n)\Delta \b S(\Delta t)^{-1},
\end{align}
one can make an equivalence between the Euler Maruyama Scheme, and the forward Euler scheme as follows 
\begin{align}
 FE(\b q^n,\Delta t,  \mathfrak{F} ) = \b q^n + \Delta t \mathfrak{F} (\b q^n) = \operatorname{EM}(\b q^n,\Delta t, \b f, G, \Delta S).
\end{align} Furthermore, this same identification allows the  stochastic Runge-Kutta \cref{method: Stochastic Runge-Kutta}, to be defined by a single $s$-stage RK method defined by a real matrix $A\in \mathbb{R}^{s\times s}$, and a real vector $b\in \mathbb{R}^s$. Such that the old value $\b q^n$ is updated by the use of $s$ internal substages $k^i$ as follows
\begin{align}
\b q^{n+1}= \b q^n+\Delta t \sum_{i=1}^s b_i \mathfrak{F}(\b k^i) \quad \text{where} \quad \b k^i=\b q^n+ \Delta t \sum_{j=1}^s a_{i j} \mathfrak{F}(\b k^j) \quad \forall i\in \lbrace 1,...,s\rbrace. \label{eq:SSSPRK is SSPRK}
\end{align}

These identifications, allow the entire deterministic SSPRK theory to be directly translated into the setting of stochastic Runge-Kutta methods. The method of proof can then follow one established in the deterministic literature (see \cite{kraaijevanger1991contractivity,higueras2005representations,gottlieb2005high}), repeated here for clarity and self-containedness. Add $r\sum_{j=1}^{s}a_{ij}\b k^{j}$, to both the LHS and RHS of the substage equation for $\b k^i$, and add $r\sum_{j=1}^{s}b_{j}\b k^{j}$ to both sides of the equation for the final stage $\b q^{n+1}$ in \cref{eq:SSSPRK is SSPRK} to give
\begin{align}
\b k^{i} + r\sum_{j=1}^{s}a_{ij}\b k^{j} &= \b q^{n} + r\sum_{j=1}^{s}a_{ij}\b k^{j} + \Delta t\sum_{j=1}^{s}a_{ij}\mathfrak{F}(\b k^{j}), \quad  \forall i \in \lbrace 1,...,s\rbrace,\\
\b q^{n+1} + r\sum_{j=1}^{s}b_{j}\b k^{j} &= \b q^n + r\sum_{j=1}^{s}b_{j}\b k^{j} + \Delta t \sum_{j=1}^{s}b_{j}\mathfrak{F}(\b k^{j}).
\end{align}

This allows the scheme to be written in terms of Euler Maruyama flow maps as follows
\begin{align}
\b k^{i} + r\sum_{j=1}^{s}a_{ij}\b k^{j} &= \b q^{n} + r\sum_{j=1}^{s}a_{ij}\operatorname{EM}(\b k^j,\Delta t/r, \b f, G, \Delta S/r^{1/2}),\quad  \forall i \in \lbrace 1,...,s\rbrace,\label{eq:substages}\\
\b q^{n+1} + r\sum_{j=1}^{s}b_{j}\b k^{j} &= \b q^n + r\sum_{j=1}^{s}b_{j}\operatorname{EM}(\b k^j,\Delta t/r, \b f, G, \Delta S/r^{1/2}).\label{eq:finalstage}
\end{align}
In the instance of a one dimensional SDE when $d=1$, one can write \cref{eq:finalstage,eq:substages} as
\begin{align}
    (I+r \mathbb{A})\b S = q^n \b e_{s+1} +r \mathbb{A} \operatorname{EM}(\b S).
\end{align}
This is a convex combination under the conditions, $(I+r \mathbb{A})^{-1} r\mathbb{A}>0$, $(I+r \mathbb{A})^{-1} \b e_{s+1}>0$, since these matrices add to give the identity matrix. For $d>1$, one can employ the notation in \cite{higueras2005representations}, which conveniently extends the theory to vector-valued ODEs (and SDEs).  We let $S := ((\b k^1)^T,...,(\b k^s)^T, \b (\b q^{n+1})^T)^T\in \mathbb{R}^{(s+1)d}$ be a vector containing the sub-stages and the final stage. We let \begin{align}
\operatorname{EM}(\b S) := (\operatorname{EM}(\b k^1)^T,\operatorname{EM}(\b k^1)^T,...,\operatorname{EM}(\b k^n)^T,0\b e_d^T )^T.    
\end{align} Then \cref{eq:finalstage,eq:substages} is written as
\begin{align}
\b S + (\mathbb{A}\otimes_{kron} I )\b S =  \b e_{s+1}\otimes_{kron}\b q^n  + r(\mathbb{A}\otimes_{kron}  I )\operatorname{EM}(\b S),
\end{align}
where $
     \mathbb{A}\otimes_{kron}I \in \mathbb{R}^{(s+1)d\times (s+1)d}$.
So as before, a convex combination can be achieved below
\begin{align}
\b S = (I_{(s+1)d}+ (\mathbb{A}\otimes_{kron} I ))^{-1} (\b e \otimes_{kron} \b q^n)+ (I+ (\mathbb{A}\otimes_{kron} I ))^{-1}r(\mathbb{A}\otimes_{kron}  I )\operatorname{EM}(\b S)
\end{align}
under the manipulation of the Kronecker product, the same Kraiijevanger conditions are observed.
\end{proof}

We will discuss the strong stability preservation of a particular Stochastic Additive Runge-Kutta (SARK) method built upon convex combinations of a diffusion-only Euler Maruyama scheme and a deterministic forward Euler scheme. 

\begin{method}[Stochastic Additive Runge-Kutta]\label{method: Stochastic Additive Runge-Kutta} We define the  Stochastic Additive Runge-Kutta flow map associated with treating the drift $\b f$ and diffusion $G$ with the Butcher tableaus $(A,b)$, $(\tilde{A},\tilde{b})$ respectively by
\begin{align}
    \b q^{n+1} &= \b q^{n} + \Delta t \sum_{i=1}^{s}b_i \b f(\b k^i) + \sum_{i=1}^{s}\tilde{b}_{i} G(\b k^i)\Delta\b S, \quad \text{where}\quad \b k^{i} = \b q^{n} + \Delta t \sum_{j=1}^{s}a_{i,j}\b f(\b k^j) + \sum_{j=1}^{s}\tilde{a}_{i,j}G(\b k^j)\Delta \b S.
\end{align}
Where $i\in \lbrace 1,..., s\rbrace$, and as before $G\Delta S$ denotes a matrix vector product.
\end{method}
The SARK method is a strict generalisation of the previous SRK method, allowing the drift to be treated with a different Butcher-Tableau than the diffusion. The  Stochastic Additive Runge-Kutta method inherits monotonic properties from both the FE flow map and the EM without drift flow map, a specific consequence of choosing this particular additive structure. This is particularly useful if the Euler Maruyama scheme cannot be proven monotone in the presence of both drift and diffusion.

\begin{theorem}[Stochastic Additive Runge-Kutta \cref{method: Stochastic Additive Runge-Kutta} is SSP under the usual ARK extension of the Kraiijevanger conditions.]
It is assumed that one can establish the following properties,
\begin{align}
    ||\b q^n+\Delta t\b f(\b q^n)||&\leq||\b q^n||, \quad \forall \Delta t\leq \tau_f\label{eq:ARK1},\\
    ||\b q^n+  G(\b q^n) \Delta \b S||&\leq||\b q^n||,\quad \forall \Delta t \leq \tau_g, \quad \forall \Delta
    \b S.\label{eq:ARK2}
\end{align}
Then
if $\Delta t\leq\tau_f r$ and $\Delta t \leq \tau_g \tilde{r}$, $\forall (r,\tilde{r})\in \mathcal{R}(\mathbb{A},\tilde{\mathbb{A}})$ then \cref{method: Stochastic Additive Runge-Kutta} preserves the desired notion of nonlinear stability. 
The region of absolute monotonicity, denoted by $\mathcal{R}(\mathbb{A}, \tilde{\mathbb{A}})$, is defined by the set of positive values $(r,\tilde{r})$ in which the natural extension of Kraaijevangers conditions \cite{higueras2005representations,kraaijevanger1991contractivity} hold
\begin{align}
& I+r \mathbb{A}+\tilde{r} \tilde{\mathbb{A}}\in GL(s+1,\mathbb{R}),\\
& (I+r \mathbb{A}+\tilde{r} \tilde{\mathbb{A}})^{-1} \mathbb{A} \geq 0 , \\
& (I+r \mathbb{A}+\tilde{r} \tilde{\mathbb{A}})^{-1} \tilde{\mathbb{A}} \geq 0, \\
& (I+r \mathbb{A}+\tilde{r}\tilde{\mathbb{A}})^{-1} e \geq 0.
\end{align}
Where the inequalities are understood point-wise, and if the set $\mathcal{R}(\mathbb{A}, \tilde{\mathbb{A}})$ has measure 0, the ARK method is not SSP and has no region of absolute monotonicity. 
\end{theorem}
\begin{proof}
We identify the following FE, and diffusion EM flow maps,
\begin{align}
\operatorname{FE}(\b q,\Delta t, \b f) = \b q^n + \Delta t \b f(\b q^n), \quad 
\operatorname{EM}(\b q,\Delta S, G) = \b q^n + \Delta t \sum_{p=1}^{P}(\Delta t)^{-1} \b g_p(\b q^n)\Delta S^{p},
\end{align}
each assumed to be SSP under the critical timestep condition $\tau_{f},\tau_{G}$ respectively \cref{eq:ARK1,eq:ARK2}. Then the SARK \cref{method: Stochastic Additive Runge-Kutta} is decomposable in the way deterministic Additive RK methods are (\cite{kraaijevanger1991contractivity,higueras2005representations}), using the notation $\mathbb{A},\tilde{\mathbb{A}}$ 
allows the following compact representation,
\begin{align}
\b S= \b e_{s+1} \otimes_{kron} \b q^n+\Delta t(\mathbb{A} \otimes_{kron} I) \b F(\b S) + \Delta t(\tilde{\mathbb{A}} \otimes_{kron} I) \mathcal{G}(\b S).
\end{align}
Where 
\begin{align}
\b e=(1, \ldots, 1)^T \in \mathbb{R}^{s+1},\quad \b S=\left((\b k^1)^T, \ldots, (\b k_s)^T, \b q_{n+1}^T\right)^T \in \mathbb{R}^{(s+1) d},\\
\b F(\b S)=\left(\b f\left(\b k^1\right)^T, \ldots, \b f\left(\b k^s\right)^T, 0\right)^T \in \mathbb{R}^{(s+1) d},\\
\mathcal{G}(S)=\Delta t^{-1} \left( \left(\sum_{p=1}^{P}\b g_p\left(\b k^1\right) \Delta S^{p} \right)^T , \ldots, \left(\sum_{p=1}^{P}\b g_p\left(\b k^{s}\right) \Delta S^{p} \right)^T, 0\right)^T \in \mathbb{R}^{(s+1) d}.
\end{align} The symbol $\otimes_{kron}$ denotes the Kronecker product. Let $d = 1$ such that we can consider 
\begin{align}
\b S &=  q^n \b e_{s+1} + \Delta t \mathbb{A}\b F(\b S) + \Delta t \tilde{\mathbb{A}} \mathcal{G}(\b S)\label{eq:additive one d}
\end{align}
We add $r \mathbb{A} \b S $ and $\tilde{r} \tilde{\mathbb{A}} \b S $ to both sides of \cref{eq:additive one d}, to give 
\begin{align}
\b S + r \mathbb{A} \b S  + \tilde{r} \tilde{\mathbb{A}} \tilde{\b S} &=  q^n \b e_{s+1} + r \mathbb{A} \left(\b S + \frac{\Delta t}{r} \b F(\b S)\right) + \tilde{r} \tilde{\mathbb{A}} \left(\b S + \frac{\Delta t}{\tilde{r}} \mathcal{G}(\b S) \right).
\end{align}
We now define, $M = I+ r \mathbb{A} \b S  + \tilde{r} \tilde{\mathbb{A}} \tilde{\b S}$, such that we have 
\begin{align}
\b S  &= M^{-1} q^n  \b e_{s+1} + r M^{-1}\mathbb{A} \left(\b S + \frac{\Delta t}{r} \b F(\b S)\right) +   \tilde{r} M^{-1} \tilde{\mathbb{A}} \left(\b S + \frac{\Delta t}{\tilde{r}} \mathcal{G}(\b S) \right)\label{eq:SSPARK 1d}.
\end{align}

Suppose that the following conditions hold (natural extension of Kraiijevanger conditions)
\begin{align}
M := I+ r \mathbb{A} \b S  + \tilde{r} \tilde{\mathbb{A}} \tilde{\b S} \in \operatorname{GL}(\mathbb{R}^{s+1}),\quad 
M^{-1}\b e \geq 0, \quad M^{-1}\mathbb{A}\geq 0, \quad M^{-1}\tilde{\mathbb{A}}\geq 0.
\end{align}
Since $M^{-1} + M^{-1}r\mathbb{A} + M^{-1}\tilde{r}\tilde{\mathbb{A}} = M^{-1}(I + r\mathbb{A} + \tilde{r}\tilde{\mathbb{A}}) = MM^{-1} = I$,  then \cref{eq:SSPARK 1d} is a convex combination. 
% Convex combinations have the following property
% \begin{align}
% ||\b S|| &\leq M^{-1} ||q^n  \b e_{s+1}|| + r M^{-1}\mathbb{A} ||\b S + \frac{\Delta t}{r} \b F(\b S) || +  \tilde{r} M^{-1} \tilde{\mathbb{A}} || \b S + \frac{\Delta t}{\tilde{r}} \mathcal{G}(\b S) ||\\
% &\leq M^{-1} ||q^n  \b e_{s+1}|| + r M^{-1}\mathbb{A} ||\b S || +  \tilde{r} M^{-1} \tilde{\mathbb{A}} ||\b S || = ||\b S||
% \end{align}
With $\Delta t \leq \min(r \tau_f,\tilde{r}\tau_g)$ one attains monotonicity of the stochastic additive Runge-Kutta scheme. The higher dimensional case when $d\geq 1$, follows analogously under standard manipulations of the Kronecker product $\otimes_{kron}$, see \cite{higueras2005representations}. 
\end{proof}

% One can consider different splittings, such as $||q^n+\Delta t f(q^n) + g_1\Delta W^1||$ + $||q^n+ \sum_{i>1}g_i\Delta W^{i}||< || q^n ||$. Interestingly these are all guaranteed by the previous definition of SSP of EM but with $\tau_0/2$. 

We define a particular Generalised Additive Runge-Kutta (GARK) method for stochastic schemes as follows. 
\begin{method}[Stochastic Generalised Additive Runge-Kutta]\label{method: Stochastic Generalised Additive Runge-Kutta} We define the  Stochastic Generalised Additive Runge-Kutta flow map as follows
\begin{align}
     \b k_{i}^f &= \b q^{n} + \Delta t \sum_{j=1}^{s^f}a^{f,f}_{i,j}\b f(\b k^f_j) + \sum_{j=1}^{s^g} a^{f,g}_{i,j}G(\b k^g_j)\Delta \b S,\quad i \in \lbrace 1,...,s^{f}\rbrace, \label{eq:SGARK_kf}\\
\b k_{i}^g &= \b q^{n} + \Delta t \sum_{j=1}^{s^f}a^{g,f}_{i,j}\b f(\b k^f_j) + \sum_{j=1}^{s^g}{a}^{g,g}_{i,j}G(\b k^g_j)\Delta \b S,\quad i \in \lbrace 1,...,s^{g}\rbrace, \label{eq:SGARK_kg}\\
    \b q^{n+1} &= \b q^{n} + \Delta t \sum_{i=1}^{s^{f}}b^f_i \b f(\b k^f_i) + \sum_{i=1}^{s^g}{b}^g_{i} G(\b k^g_i)\Delta\b S. \label{eq:SGARK_final}
\end{align}
GARK schemes allow different stage values for different components. 
\end{method}
The extension of the SSP theory to stochastic systems can be made using the theory developed in \cite{sandu2015generalized}, whilst making associations to appropriate FE maps similar to the previous examples in this paper. Rather than elaborate on such a construction, we shall instead introduce some practical examples (\cref{method:SSP- Additive Operator Splitting}, \cref{method:SSP- Sequential Operator Splitting}) based on operator splitting later in this paper.

\subsection{Convergence}\label{sec:convergence}
One requires having bounded increments for the stochastic generalisation of the SSP property. One approach to having bounded increments is to sample from bounded distributions agreeing in moments with the normal distribution. Two classical methods are as follows.

\begin{example}
The P-dimensional random variable whose components are two-point random variables 
\begin{align}
    \mathbb{P} \left(\Delta \widetilde{W}^p  =\pm \sqrt{\Delta t} \right) = 1/2,\label{eq:2point_random_variable}
\end{align}
has the same first 3 moments as that of the $N(0,\Delta t)$ distribution and satisfies $|\mathbb{E}[\widetilde{W}]| + |\mathbb{E}[\widetilde{W}^3]|+|\mathbb{E}[\widetilde{W}^2]-\Delta t|\leq K\Delta t^2$ sufficient for weak order $1$ of convergence \cite{Kloeden_Platen_1992}. 
\end{example}

\begin{example}
Also found in \cite{Kloeden_Platen_1992} the $P$ dimensional random variable defined component-wise as a three-point random variable \begin{align}
    \mathbb{P} \left(\Delta \widetilde{W}^p  =\pm \sqrt{3 \Delta t}  \right) = 1/6,\quad \mathbb{P}\left(\Delta \widetilde{W}^p=0\right)=2/3,
\end{align}\label{eq:3point_rv}
is a distribution with the same first 5 moments as that of the normal $N(0,\Delta t)$ distribution and satisfies a similar moment estimate  $|\mathbb{E}[\widetilde{W}]| + |\mathbb{E}[\widetilde{W}^2]-\Delta t| + |\mathbb{E}[\widetilde{W}^3]|+ |\mathbb{E}[\widetilde{W}^4]-3\Delta t|+|\mathbb{E}[\widetilde{W}^5]|\leq K\Delta t^3$. 
\end{example}

Whilst one can retain the weak order of convergence (sufficient for most applications) provided high enough moments are captured and the drift and diffusion are sufficiently differentiable \cite{Kloeden_Platen_1992}, 
one can lose the strong order of convergence or the mean square order of convergence by sampling from such distributions.
The remainder of this section discusses strong and mean square convergence when cutting the tails of the normal distribution and can be avoided for those who wish to have weak convergence. The mean square order of convergence is defined below in \cref{def:mean square convergence}. 

\begin{definition}[Mean square convergence]\label{def:mean square convergence}
The output of a numerical method $\b q^n$ converges in the mean-square sense with order $p>0$ to the exact solution $\b q(t^n)$ of an SDE at time $t^n$ if there exists a constant $L>0$, and a critical timestep $\tau_c>0$ such that for each $\left.\Delta t \in( 0, \tau_c\right]$, on has 
\begin{align}
\left(\mathbb{E}\left[\left\|\b q^n -\b q(t_n)\right\|_2^2\right]\right)^{1 / 2} \leq L \Delta t^p.
\end{align}

\end{definition}
Strong convergence is defined below.
\begin{definition}[Strong global convergence]\label{def:strong convergence} Let $\b q^n$ denote the numerical approximation of the stochastic process $\b q$ at time $t^n$, after $n$ steps of equal step-size $\Delta t$. We say that the numerical method is converging with strong global order $p$ if $\exists \tau_c>0, L>0$, s.t $
\mathbb{E}[ ||\b q^{n}-\b q(t^n)||_2]\leq L\Delta t^{p}, \quad \forall \Delta t \in (0,\tau_c).$
\end{definition}

% \begin{definition}[Weak global convergence order $p$]\label{def:weak convergence}
% Let $\b q_n$ denote the numerical approximation of the stochastic process $\b q$ at time $t_n$, after $n$ steps of equal step-size $\Delta t$. We say that the numerical method is converging with weak global order $p>0$, if $\forall \Phi \in C^{2(p+1)}_P(\mathbb{R}^d,\mathbb{R})$, $\exists \delta>0, K>0$, s.t 
% \begin{align}
% |\mathbb{E}[\Phi(q_n)]-\mathbb{E}[\Phi(q(t_n)]|\leq K\Delta t^{p}, \quad \forall \Delta t \in (0,\delta).
% \end{align}
% Where $C^{l}_P(\mathbb{R}^d,\mathbb{R})$ denotes the set of $l$ times differentiable functions $\Phi:\mathbb{R}^d\rightarrow\mathbb{R}$, whose $l$th order partial derivatives and below have polynomial growth.
% \end{definition}
Mean square convergence implies strong convergence under Lyapunov's inequality. To retain the mean square (strong) order of the scheme whilst using bounded increments, one can truncate the tails of the normal distribution. More specifically in \cite{milstein2002numerical} the increments in \cref{eq:bounded increments} were proven to converge to the Itô system when using a Backward Euler discretisation. As is common, the increments $\Delta W\sim N(0,\Delta t)$ can be written in terms of the standard normal distribution $\Delta Z \sim N(0,1)$, by rescaling $\Delta W = \sqrt{\Delta t}\Delta Z$. Similarly, the bounded increments in \cref{eq:bounded increments} can be rescaled as $\Delta \widetilde{W} = \sqrt{\Delta t}\Delta \widetilde{Z}$.

\begin{definition}
[Milstein-Tretyakov Bounded normal increments]\label{Milstein-Tretyakov Bounded normal increments} Given $\Delta Z\sim N(0,1)$ and $\Delta t \in \mathbb{R}^{>0}$, Milstein and Tretyakov \cite{milstein2004stochastic} define a symmetric bounded increment $\Delta \widetilde{Z}_{\Delta t}$, from the following random variable, 
 \begin{align}
\Delta \widetilde{Z}_{\Delta t}:=\left\{\begin{array}{c}
\Delta Z,\quad |\Delta Z| \leq A_{\Delta t}, \\
A_{\Delta t}, \quad \Delta Z>A_{\Delta t},\\
-A_{\Delta t}, \quad \Delta Z<-A_{\Delta t}.
\end{array}\right.
\label{eq:bounded increments}
\end{align}
where $\Delta \widetilde{Z}_{\Delta t}$ is bounded by $A_{\Delta t}:=\sqrt{2 k|\ln {\Delta t}|}, k \geq 1$. 
\end{definition}
These increments satisfy $(\mathbb{E}[\Delta \widetilde{Z}_{\Delta t}] = \mathbb{E}[\Delta Z_{\Delta t}] = 0)$ and the following inequalities
\begin{align}
 0 \leq \mathbb{E}\left[(\Delta Z-\Delta \widetilde{Z} )^2\right] &\leq  {\Delta t}^k, \quad \mathbb{E}\left[(\Delta Z)^2\right] - \mathbb{E}\left[(\Delta \widetilde{Z})^2\right] \leq \left(1 + \frac{4}{\sqrt{\pi}}\sqrt{k|\ln {\Delta t}|}\right)\Delta t^k \label{eq:milstein inequality},\\
  \mathbb{E}[(\Delta Z^2 - \Delta \widetilde{Z}^2)(\Delta Z - \Delta \widetilde{Z})] & = 0.\label{eq:additional inequality}
\end{align}
The first two inequalities are established in \cite{milstein2002numerical} and are elaborated upon in \cref{sec:inequalities}, and all three are required later in this paper.

When using increments sampled from the normal distribution, the Euler Maruyama scheme \cref{method:EM}, is not convergent to the Stratonovich system \cref{stratonovich}, it converges to the Itô system 
(\cref{eq:Ito}) with mean square order 1/2. When using increments sampled from the normal distribution, the SARK method (\cref{method: Stochastic Additive Runge-Kutta}) converges to the following SDE
\begin{align}
    d \b q = \left( \lambda_0 \b f(\b q) + \lambda_1  DG|_{\b q} G(\b q) \right)dt + \lambda_2 G(\b q)  d \b W, \quad \lambda_1 = \sum_{j=1}^{s} \tilde{b}_j\left( \sum_{k=1}^{s} \tilde{a}_{jk} \right), \quad \lambda_0 = \sum_{j=1}^{s} b_j, \quad \lambda_2 = \sum_{j=1}^{s} \tilde{b}_j.
    \label{eq:Rumelin SDE}
\end{align}
This allows the SARK \cref{method: Stochastic Additive Runge-Kutta} convergence to either Itô \cref{eq:Ito} or Stratonovich \cref{stratonovich} upon appropriate Butcher tableau choices, and we shall refer to the SDE specified by the butcher tableau as the ``Ruemelin-SDE". The order one deterministic ARK scheme conditions $b^T e = \tilde{b}^Te = 1 = \lambda_2=\lambda_1$, are required to capture $\b f, G d\b W$ terms. The condition $\tilde{b}^T\tilde{A}e = 1/2 = \lambda_1$, allows convergence to the Stratonovich equation by attaining the Itô-Stratonovich correction through numerical approximation, whereas the condition $\tilde{b}^T\tilde{A}e = 0 = \lambda_1$, allows convergence to the Itô equation.

\begin{assumption}[MS convergence]
Let $\b q^{n+1}$, be the solution to the  Stochastic Additive Runge-Kutta \cref{method: Stochastic Additive Runge-Kutta} from the initial condition $\b q^{n}$, let $\b q(t^{n+1})$, be the solution to the Rumelin SDE \cref{eq:Rumelin SDE} from the initial condition $\b q^{n}$, using the normal increments $\Delta \b W$. We shall assume the following mean square local error estimates 
\begin{align}
||\mathbb{E}\left[\b q(t^{n+1}) - \b q^{n+1}\right]||_{2} &= \mathcal{O}(\Delta t^{p_1}), \label{eq:assumption 1}\\
\left(\mathbb{E}\left[||\b q(t^{n+1}) - \b q^{n+1}||_2^2\right]\right)^{1/2} &= \mathcal{O}(\Delta t^{p_2})\label{eq:assumption 2},
\end{align}
$p_1\geq p_2+1/2$, $p_2\geq 1$, where we shall for convenience take $p_2 = 1$. 
\end{assumption}
 These local estimate need not be assumed but may be proven by Taylor expanding the numerical method and taking the difference from a Stratonovich Taylor (Wagner-Platen) expansion. These local error estimates are commonly assumed in light of the fundamental theorem of mean square convergence \cite{milstein2002numerical}, required for the mean square order $1/2$ of the numerical scheme described below. 

\begin{theorem}[G. N. Milstein, M. V. Tretyakov, Fundamental theorem on mean square convergence \cite{milstein2004stochastic}]\label{thm:fundamental-thm-msc}
Suppose an approximation $\b q^{n+1}$ and the exact solution $\b q(t^{n+1})$ both starting from the arbitrary initial condition $\b q^n\in \mathbb{R}^d$  satisfy the following local error estimates, 
\begin{align}
|| \mathbb{E}[\b q(t^{n+1})-\b q^{n+1}]||_2 \leq L\left(1+||\b q^n||_2^2\right)^{1 / 2} \Delta t^{p_1}, \\
{\left[\mathbb{E}[||\b q(t^{n+1})-\b q^{n+1}||_2^2]\right]^{1 / 2} \leq L\left(1+||\b q^n||_2^2\right)^{1 / 2} \Delta t^{p_2} .}
\end{align}
Where $L$ denotes an arbitrary constant, not necessarily the same in each equation. 
 % and it also holds for arbitrary $t^{n+\gamma\Delta t} = t^{n}+\gamma \Delta t,$ $\gamma \in[0,1]$.
 If both
$$
p_1 \geq p_2+\frac{1}{2},\quad p_2 \geq \frac{1}{2}.
$$
Then for the entire time inteval $n=0,1, \ldots, N$ the following inequality holds:
\begin{align}
\left[\mathbb{E}[||\b q\left(t^n\right)-\b q\left(t^n\right)||_2^2]\right]^{1 / 2} \leq L\left(1+\mathbb{E}[||\b q_0 ||_2^2]\right)^{1 / 2} \Delta t^{p_2-1 / 2},
\end{align}
i.e. the global mean square order of accuracy of the method is $p=p_2-1 / 2$.
\end{theorem}
\begin{remark}
Since in this work we do not consider approximating higher-order stochastic integrals it is sufficient to show \begin{align}
    p_1\geq 3/2, \quad p_2\geq 1.
\end{align}
to prove global mean square convergence order $1/2$.
\end{remark}

We wish to prove convergence of the SRK, SARK, and SGARK, schemes when using bounded increments in \cref{eq: difference increments} and wish to use \cref{thm:fundamental-thm-msc}. To do so we use \cref{convergence lemma} below introduced in \cite{milstein2002numerical} where if one adds and subtracts the unbounded increment system, one can require local error estimates between the bounded and unbounded increment-driven system sufficient to prove convergence of the bounded increment-driven system to the exact solution of the SDE. 

\begin{lemma}[Convergence \cite{milstein2002numerical}]\label{convergence lemma}
Let $\b q(t^{n+1})$, be the analytic solution to the Rumelin SDE \cref{eq:Rumelin SDE} from the initial condition $\b q^{n}$. Let $\b q^{n+1}$, be the solution to the  Stochastic Additive Runge-Kutta \cref{method: Stochastic Additive Runge-Kutta} from the initial condition $\b q^{n}$, using a $P$-dimensional normally distributed variable $\Delta Z^p\sim N(0,1)$. Let $\tilde{\b q}^{n+1}$, be the solution to the  Stochastic Additive Runge-Kutta \cref{method: Stochastic Additive Runge-Kutta} from the initial condition $\b q^{n}$, using the bounded increments $\Delta \widetilde{Z}$ in \cref{eq:bounded increments}. 
Then the following conditions bounding the difference between the two numerical methods
\begin{align} \mathbb{E}[\b q^{n+1}-\tilde{\b q}^{n+1}] &=\mathcal{O}(\Delta t^{p_1})\\
\left(\mathbb{E}[||\b q^{n+1}-\tilde{\b q}^{n+1}||_2^2] \right)^{1/2}&=\mathcal{O}(\Delta t^{p_2}).
\end{align}
where $p_1\geq p_2+1/2$, $p_2\geq 1/2$,
are sufficient (by \cref{thm:fundamental-thm-msc}), to establish the bounded increment-driven system is also convergent with mean square order $p_2-1/2$. 
\end{lemma}
\begin{proof}
Established in \cite{milstein2002numerical} it is sufficient by the fundamental theorem on mean square convergence that the following identities are satisfied,
\begin{align}
\mathbb{E}\left[\b q(t^{n+1})  - \tilde{\b q}^{n+1}\right]  &= \mathbb{E}\left[\b q(t^{n+1}) - \b q^{n+1}+\b q^{n+1} - \tilde{\b q}^{n+1}\right] = \mathcal{O}(\Delta t^{p_1})\label{eq:condition 1},\\
\left(\mathbb{E}\left[||\b q(t^{n+1})  - \tilde{\b q}^{n+1}||^2_2\right]\right)^{1/2} & = \left(\mathbb{E}\left[||\b q(t^{n+1}) - \b q^{n+1} + \b q^{n+1} - \tilde{\b q}^{n+1}||^2_2\right]\right)^{1/2} = \mathcal{O}(\Delta t^{p_2})\label{eq:condition 2},
\end{align}
for $p_1\geq p_2+1/2$, $p_2\geq 1/2$, for order $p_2-1/2$ convergence. It has been assumed that the numerical method for unbounded increments already satisfies \cref{eq:assumption 1,eq:assumption 2}. Therefore the following equalities  
\begin{align}
\mathbb{E}[\b q^{n+1}-\tilde{\b q}^{n+1}] &=\mathcal{O}(\Delta t^{p_1}),\\
\left(\mathbb{E}[||\b q^{n+1}-\tilde{\b q}^{n+1}||_2^2]+2\mathbb{E}[(\b q^{n+1}-\tilde{\b q}^{n+1})^T(\b q(t^{n+1})-\b q^{n+1})] \right)^{1/2}&=\mathcal{O}(\Delta t^{p_2}),
\end{align}
are sufficient for order $p_2-1/2$. This further simplifies noting $\mathbb{E}[\b a^T \b b]\leq \mathbb{E}[|\b a^T \b b|]\leq \mathbb{E}[||\b a||_2^2]^{1/2} \mathbb{E}||\b b||_{2}^2]^{1/2}$, to \begin{align}
\mathbb{E}[\b q^{n+1}-\tilde{\b q}^{n+1}] &=\mathcal{O}(\Delta t^{p_1}),\\
\left(\mathbb{E}[||\b q^{n+1}-\tilde{\b q}^{n+1}||^2_2] \right)^{1/2}&=\mathcal{O}(\Delta t^{p_2}).
\end{align}

\end{proof}

% Wait until needed!
% Furthermore, a modification of the arguments in \cite{milstein2002numerical,milstein2004stochastic}, can be used for the following additional inequalities on the higher-order even moments,  
% \begin{align}
% \mathbb{E}[(\Delta Z - \Delta \tilde{Z})^m] = \frac{2}{\sqrt{2\pi}}\int_{0}^{\infty} y^m e^{-y^2/2 - Ay}dy e^{-A^2/2} \leq 2^{-m+1}\frac{m!}{(m/2)!}e^{-A^2/2} = \frac{m!}{2^{m-1}(m/2)!}\Delta t^k, \quad m \quad \text{even}\label{eq:moment bounds}
% \end{align} 
% The difference in the moments from the original can also be shown to be  
% \begin{align}
% \mathbb{E}[(\Delta Z)^m] - \mathbb{E}[(\Delta \widetilde{Z})^m] \leq O(\sqrt{\ln(\Delta t)}\Delta t^k)
% \end{align}
% similarly bounded as presented in the appendix. 

% \begin{align}
%   \Delta \widetilde{Z}_{\Delta t} \leq A_{\Delta t},\quad  0 \leq \mathbb{E}\left(\Delta Z^2-\Delta \widetilde{Z}_{\Delta t}^2\right) = 1 - \mathbb{E}[\Delta \widetilde{Z}_{\Delta t}^2]\leq(1+2 \sqrt{2 k|\ln {\Delta t}|}) {\Delta t}^k. \label{eq:milstein inequality}
% \end{align}

% Then, by using the fundamental theorem of the mean square convergence theorem and the above lemma, Milstein and Tretyakov \cite{milstein2002numerical} show that a BE scheme driven by $\Delta Z$ assumed to converge to the Itô system \cref{eq:Ito} with order $p$, and the same scheme driven by the $\widetilde{\Delta Z}$ increments has a difference estimate sufficient to conclude the system-driven by $\widetilde{\Delta Z}$, also converges with order $p$ provided, $A_k$ is chosen with $k\geq 1$. 

Using \cref{convergence lemma} we show the EM scheme converges to the Itô system \cref{eq:Ito} when using bounded increments $ \Delta\widetilde{\b Z}$ in \cref{eq:bounded increments} rather than $\Delta \b Z$ normal ones. 

\begin{example}[\cite{milstein2004stochastic}]Let $\b q^{(1)}:= \b q + \Delta t f(\b q) + t^{1/2} G(\b q) \Delta \b Z$, be a solution of one step EM scheme, converging with strong order 1/2 to the Itô system \cref{eq:Ito}. Let $\b q^{(2)}:= \b q + \Delta t f(\b q) + t^{1/2} G(\b q) \Delta  \widetilde{\b Z}$, be a solution of the one step bounded increment driven system \cref{eq:bounded increments}. Then the difference between the two solutions from the same initial condition $\b q$ satisfies
\begin{align}
   \b q^{(1)}-\b q^{(2)} = t^{1/2}G(\b q)[\Delta \b Z-\Delta  \widetilde{\b Z}],
\end{align}
after one time step.
Through the centered properties of the increments $(\b Z,\widetilde{\b Z})$ we have the first local error estimate $||\mathbb{E}\left[\b q^1 - \b q^2 \right]||_{2}= 0$, and by the properties of matrix norms and \cref{eq:milstein inequality} one has, 
\begin{align}
\mathbb{E}\left[||\b q^1 - \b q^2||_2^2 \right]\leq t ||G(\b q)||^2_2 \mathbb{E}\left( ||\Delta \b Z - \Delta \widetilde{\b Z}||_2^2\right) = t ||G(\b q)||^2_2 \sum_{p=1}^{P} \mathbb{E}\left[ (\Delta  Z^p - \Delta\widetilde{  Z}^p)^2\right] \leq P||G||_{2}^2 \Delta t^{k+1}.
\end{align}%(1+\sqrt{2k|\ln{\Delta t}|}))
Then by the growth bound assumption required for well-posedness of the SDE in \cref{sec:well posedness estimates}, one has 
\begin{align}
\left(\mathbb{E}\left[||\b q^1 - \b q^2||_2^2 \right] \right)^{1/2}\leq \sqrt{PL}\left( 1+ ||\b q||_{2}^2\right)^{1/2}\Delta t^{(k+1)/2}
\end{align}%\sqrt{(1+\sqrt{2k|\ln{\Delta t}|})}
% In \cite{milstein2002numerical,milstein2004stochastic} it is argued that this is sufficient for the following bound (in terms of an arbitrary constant $L_2$)
% \begin{align}
% \left(\mathbb{E}\left[||\b q^1 - \b q^2||^2 \right] \right)^{1/2} \leq L_2\left( 1+ ||\b q^n||_{2}^2\right)^{1/2}\Delta t^{(k+1)/2}.     
% \end{align}
% This is not formally true, as can be shown through a contradiction argument there exists $\Delta t$, small enough that this cannot be established. Instead one can establish the following bound, 
% \begin{align}
% \left(\mathbb{E}\left[||\b q^1 - \b q^2||^2 \right] \right)^{1/2}\leq L_2\left( 1+ ||\b q^n||_{2}^2\right)^{1/2}\Delta t^{(k+1)/2 -\epsilon}, \quad \epsilon >0.
% \end{align}
% These conditions are sufficient for mean square convergence order $1/2$ when $k\geq 1-2\epsilon$, by the FTMSC as we have $p_1=\infty$, and $p_2 = (k+1)/2 -\epsilon$ such that taking epsilon small gives 
% $k>1$ is sufficient for order 1/2.
These local error estimates are sufficient when $k\geq 1$, by the FTMSC \cref{thm:fundamental-thm-msc} and \cref{convergence lemma} with $p_1=\infty$, and $p_2 = (k+1)/2$ for mean square convergence order $1/2$ of the bounded increment driven system. 
\end{example}

% Milstein and Tretyakov \cite{milstein2002numerical,milstein2004stochastic} state that using similar arguments to the one above one can make the same claim for Stochastic Runge-Kutta schemes similar to \cref{method: Stochastic Additive Runge-Kutta}. However, a detailed proof involving regularity conditions is not provided, below we fill in some of these details regarding the \cref{method: Stochastic Additive Runge-Kutta}.

Using \cref{convergence lemma,sec:taylors-theorem,sec:inequalities}, we similarly establish \cref{method: Stochastic Additive Runge-Kutta} when driven by bounded increments as in \cref{Milstein-Tretyakov Bounded normal increments} one attains the same mean square convergence to the Rumelin-SDE as when using normally distributed increments. A similar statement is made in \cite{milstein2002numerical,milstein2004stochastic} however, a detailed proof involving regularity conditions is not provided.

\begin{theorem}[MS Convergence of bounded increment SARK \cref{method: Stochastic Additive Runge-Kutta}.]\label{lem: mse}
The difference between the one step bounded increment ($\Delta \widetilde{\b Z}$) system $\tilde{\b q}^{n+1}$, and the one step normally distributed driven ($\Delta \b Z$) system $\b q^{n+1}$, both starting from the same initial condition $\b q = \b q^n$ and using the same ARK \cref{method: Stochastic Additive Runge-Kutta}, satisfy the following local error estimates
\begin{align}
\mathbb{E}[\b q^{n+1}-\tilde{\b q}^{n+1}] = \mathcal{O}(\Delta t^{\min \lbrace 3/2,k+1-\epsilon \rbrace}),\quad (\mathbb{E}[ || \b q^{n+1}-\tilde{\b q}^{n+1}||_2^2])^{1/2}=\mathcal{O}(\Delta t^{\min \lbrace 1,(k+1)/2 \rbrace}),
\end{align}
where $\epsilon$ is a small positive number.
Provided that the Butcher tableau's $((A,b),(\tilde{A},\tilde{b}))=((a_{i,j},b_j),(\tilde{a}_{i,j},\tilde{b}_j))$ are componentwise bounded, the drift and diffusion are componentwise twice continuously differentiable $f^{k} \in C^{2}(\mathbb{R}^d;\mathbb{R})$, $g^{k}_p \in C^{2}(\mathbb{R}^d; \mathbb{R})$,  $\forall p\in \lbrace 1,...,P\rbrace$, $\forall k \in \lbrace 1,...,d\rbrace$, and $DG|_{\b q}$, $G(\b q)$, $D\b f|_{\b q}$, $\b f(\b q)$ are bounded. The above local error estimates in combination with \cref{thm:fundamental-thm-msc,convergence lemma} are sufficient, to prove mean square convergence order $1/2$ of the SARK \cref{method: Stochastic Additive Runge-Kutta} (to the Rumelin-SDE) when using the bounded increments in \cref{Milstein-Tretyakov Bounded normal increments} with $k\geq 1$. 
\end{theorem}
\begin{proof}\label{proof: mse} \cref{lem: mse}. 
The difference between, the SSPARK \cref{method: Stochastic Additive Runge-Kutta}, with bouned increments $\Delta \widetilde{\b Z}$ \cref{eq:bounded increments} and without bounded increments using normally distributed random variable $\Delta \b Z$, both starting from the initial condition $\b q$, is
\begin{align}
\b q^{n+1} - \tilde{\b q}^{n+1} = \Delta t\sum_{i=1}^{s}b_{i} [\b f(\b k^{i})-\b f(\tilde{\b k^{i}})]
+\sqrt{\Delta t}
\sum_{i=1}^{s} \tilde{b}_{i} 
\left[G(\b k^{i}) \Delta \b Z -  G( \tilde{\b k^{i}})\Delta \widetilde{\b Z}\right], \label{eq: difference increments}
\end{align}
where the substages are
\begin{align}
\b k^{i} &= \b q^{n} + \Delta t \sum_{j=1}^{s}a_{i,j}\b f(\b k^j) + (\Delta t)^{1/2}\sum_{j=1}^{s}\tilde{a}_{i,j}G(\b k^j)\Delta \b Z,\\
\tilde{\b k^{i}} &= \b q^{n} + \Delta t \sum_{j=1}^{s}a_{i,j}\b f(\tilde{\b k^{j}}) + (\Delta t)^{1/2}\sum_{j=1}^{s}\tilde{a}_{i,j}G(\tilde{\b k^{j}})\Delta \widetilde{\b Z}.
\end{align} 
Now Taylor expand $G$, about $\b q^n$. Note that $G : \mathbb{R}^d \rightarrow \mathbb{R}^{d\times P}$, is defined with columns $\b g_p: \mathbb{R}^d \rightarrow \mathbb{R}^d$, such that the derivative of $G$ with respect to $\b q$ evaluated at $\b a$, is denoted 
$DG|_{\b a}$ and can be thought of as a linear map taking values to $\mathbb{R}^{d\times P \times d}$. We require that the $k$-th component of $\b g_p$ is continuous, and twice continuously differentiable within an open ball centered at $\b q$ containing $\b q+\b a$, to apply a component-wise multivariate Taylor theorem.
\begin{align}
G(\b k^i) = G(\b q) + DG|_{\b q} \underbrace{\left(\Delta t \sum_{j=1}^{s}a_{i,j}\b f(\b k^j) + (\Delta t)^{1/2}\sum_{j=1}^{s}\tilde{a}_{i,j}G(\b k^j)\Delta \b Z\right)}_{:=\b a} + R_1(\b q, \b a) ,\\
G(\tilde{\b k}^i) =G(\b q) + DG|_{\b q} \underbrace{\left(\Delta t \sum_{j=1}^{s}a_{i,j}\b f(\tilde{\b k}^j) + (\Delta t)^{1/2}\sum_{j=1}^{s}\tilde{a}_{i,j}G(\tilde{\b k}^j)\Delta \tilde{\b Z}\right)}_{:=\tilde{\b a}}  + R_2(\b q, \tilde{\b a}). 
\end{align}
Where $R_1(\b q, \b a)$, $R_2(\b q, \tilde{\b a})$, are remainder terms satisfying the following bounds $||R_1(\b q, \b a)||_{2}\leq L ||\b a||^2$, $||R_2(\b q, \tilde{\b a})||_{2} \leq L ||\tilde{\b a}||^2$. 
We similarly require that the $k$-th component of $\b f$ is continuous and twice continuously differentiable within an open ball centered at $\b q$ containing $\b q+\b a$, to apply a component-wise multivariate Taylor theorem as follows
\begin{align}
\b f(\b k^i) &= \b f(\b q) + D \b f|_{\b q}\b a + \b R_3(\b q,\b a),\\
\b f(\tilde{\b k}^i) &= \b f(\b q) + D\b f|_{\b q}\tilde{\b a} + \b R_4(\b q,\tilde{\b a}).
\end{align}
Where $||\b R_3(\b q, \b a)||_2\leq L ||\b a||^2_2$, $||\b R_4(\b q, \tilde{\b a})||_2\leq L ||\tilde{\b a}||^2_2$. So that the difference can be written as
\begin{align}
\b q^{n+1} - \tilde{\b q}^{n+1} &= \Delta t\sum_{i=1}^{s}b_{i} \left[DF|_{\b q}(\b a -\tilde{\b a})\right]  +\Delta t\sum_{i=1}^{s}b_{i}\left[ R_3(\b q,\b a) - R_4(\b q,\tilde{\b a})\right]
+\sqrt{\Delta t}
\sum_{i=1}^{s} \tilde{b}_{i} 
G(\b q) \left[\Delta \b Z - \Delta \widetilde{\b Z}\right]\\
&+ \sqrt{\Delta t}
\sum_{i=1}^{s} \tilde{b}_{i} DG|_{\b q}(\b a \Delta \b Z- \tilde{\b a}\Delta \widetilde{\b Z}) + \sqrt{\Delta t}
\sum_{i=1}^{s} \tilde{b}_{i}\left( R_1(\b q,\b a)\Delta \b Z- R_2(\b q,\b a)\Delta \widetilde{\b Z}\right).\label{eq:deo}
\end{align}
% We require estimates of the form $||\mathbb{E}[\b q^{n+1} - \tilde{\b q}^{n+1}]||_2$, $\mathbb{E}[||\b q^{n+1} - \tilde{\b q}^{n+1}||^2]$, it is convenient to note that upon taking expectation one can bound the higher order remainder terms, for example using the triangle inequality and the multivariate chain on $||\b a||^2_2,||\tilde{\b a}||^2_2$ one can establish that
% \begin{align}
% \mathbb{E}\left[\Delta t\sum_{i=1}^{s}b_{i}\left[ R_3(\b q,\b a) - R_4(\b q,\tilde{\b a})\right]\right] = O(\Delta t^2).
% \end{align}

Now Taylor expanding terms in $\b a, \tilde{\b a}$, about $\b q$ gives the following
\begin{align}
\Delta t\sum_{i=1}^{s}b_{i} \left[DF|_{\b q}(\b a -\tilde{\b a})\right] &=  
\Delta t\sum_{i=1}^{s}b_{i} DF|_{\b q} \left((\Delta t)^{1/2}\sum_{j=1}^{s}\tilde{a}_{i,j}G(\b q)[\Delta \b Z-\Delta \widetilde{\b Z}] + \mathcal{O}(\Delta t)\right),
\end{align}
\begin{align}
\Delta t^{1/2}\sum_{i=1}^{s}\tilde{b}_i DG|_{\b q} (\b a \Delta \b Z - \tilde{\b a}\Delta \widetilde{\b Z})
% =
% \Delta t^{1/2}\sum_{i=1}^{s}\tilde{b}_i DG|_{\b q}
% \Bigg[&
% \Delta t \sum_{j=1}^{s}a_{ij}\left(DF|_{\b q}(\b a - \tilde{\b a}) + R_3-R_4\right) +\Delta t^{1/2}\sum_{j=1}^{s}\tilde{a}_{ij}G(q)(\Delta \b Z - \Delta \tilde{ \b Z})\\
% &+\Delta t^{1/2}\sum_{j=1}^{s}\tilde{a}_{ij}DG|_{\b q} (\b a \Delta \b Z- \tilde{\b a}\Delta \tilde{\b Z} )+\Delta t^{1/2}\sum_{j=1}^{s}\tilde{a}_{ij}(\Delta \b Z R_1 - \tilde{\b Z} R_2)
% \Bigg]\\
&=
\Delta t^{1/2}\sum_{i=1}^{s}\tilde{b}_i DG|_{\b q}
\Bigg[\Delta t^{1/2}\sum_{j=1}^{s}\tilde{a}_{ij}G(q)(\Delta \b Z - \Delta \widetilde{ \b Z})+\mathcal{O}(\Delta t)
\Bigg].
\end{align}
Substituting these expressions (into \cref{eq:deo}), gives the following expression
\begin{align}
\b q^{n+1}-\tilde{\b q}^{n+1} = \sqrt{\Delta t}\sum_{i=1}^{s}\tilde{b}_i G(\b q^n)[\Delta \b Z - \Delta \widetilde{\b Z}] + (\Delta t)^{3/2} \sum_{i=1}^s DF|_{\b q} \sum_{j=1}^{s}\tilde{a}_{i,j} \left[G(\b k^j)\Delta \b Z - G(\tilde{\b k}^j)\Delta \widetilde{\b Z}\right]\\
+ \Delta t \sum_{i=1}^{s}\tilde{b}_i DG|_{\b q} \sum_{j=1}^{s}\tilde{a}_{i,j} \left[ G(\b k^{j})\Delta \b Z\otimes \Delta \b Z -G(\b k^{j})\Delta\widetilde{\b Z}\otimes \Delta\widetilde{\b  Z} )\right] + \mathcal{O}(\Delta t^{3/2})
\end{align}
Where for notational convenience we have written $DG_{\b q} (G_{\b q} \Delta \b Z) \Delta \b Z = DG_{\b q} G_{\b q} (\Delta \b Z \otimes \Delta \b Z )$, and have used the remainder bounds from Taylor's theorem \cref{sec:taylors-theorem} and \cref{sec:inequalities} to establish the other terms are higher order terms.
The subsequent Taylor expansion of $G$ again leads to (at leading order) the following expression 
\begin{align}
\b q^{n+1}-\tilde{\b q}^{n+1} = \sqrt{\Delta t}G(\b q)\sum_{i=1}^{s}\tilde{b}_i [\Delta \b Z - \Delta \widetilde{\b Z}] + \Delta t DG|_{\b q}G(\b q)\sum_{i=1}^{s}\tilde{b}_i  \sum_{j=1}^{s}\tilde{a}_{i,j}   \left[\Delta \b Z\otimes \Delta \b Z -\Delta \widetilde{\b Z}\otimes \Delta \widetilde{\b Z}\right]\\
+ (\Delta t)^{3/2} DF|_{\b q}G(\b q)\sum_{i=1}^s  \sum_{j=1}^{s}\tilde{a}_{i,j} \left[\Delta \b Z - \Delta \widetilde{\b Z}\right] +  \mathcal{O}(\Delta t^{3/2}).
\end{align}
Now since $(a_{i,j},b_j)$, $(\tilde{a}_{i,j},\tilde{b}_j)$ are assumed finite, taking the expectation sets the first term to zero by the symmetric property of $\Delta Z$ and $\Delta \widetilde{Z}$. The $i,j$-th component of $(\Delta Z\otimes \Delta Z)_{i,j}$, also vanishes for $i\neq j$ in expectation, leaving at leading order an error associated with numerically approximating the Itô-Stratonovich correction by using bounded increments (as opposed to normally distributed ones). This can be bounded as follows 
\begin{align}
\mathbb{E}[\b q^{n+1}-\tilde{\b q}^{n+1}] &= \Delta t \sum_{i=1}^s \tilde{b}_i\sum_{i=1}^s \tilde{a}_{i,j}
DG|_{\b q} G(\b q) P (1-\mathbb{E}[\Delta \widetilde{Z}^p]) + \mathcal{O}(\Delta t^{3/2})  \\
&\leq   (\tilde{\b b}^T\tilde{A}\b e)
DG|_{\b q} G(\b q) P \left(1+\frac{4}{\sqrt{\pi}}\sqrt{k|\ln(\Delta t)|}\right)\Delta t^{k+1}+ \mathcal{O}(\Delta t^{3/2}). \label{eq1:estimate_ok}
\end{align}

Sufficient to conclude that the leading order term is $\mathcal{O}(\Delta t^{\min\lbrace k+1-\epsilon,3/2\rbrace})$ where $\epsilon>0$ is a small number. Then consider the square of the difference, 
\begin{align}
\mathbb{E}[||\b q^{n+1}-\tilde{\b q}^{n+1}||_2^2] &\leq  \Delta t ||G(\b q)||_2^2 \left(\sum_{i=1}^{s} b_i\right)^2 \mathbb{E}[|| \Delta \b Z - \Delta \widetilde{\b Z}||_2^2] \\
&+ \mathbb{E}\bigg[ \underbrace{\Delta t^{3/2} (\tilde{\b b}^T \b e )(\tilde{\b b}^T \tilde{A}\b e)\frac{\partial g_p^k}{\partial q^j}g^j_q g^k_r \left(\Delta Z^p \Delta Z^q - \Delta \widetilde{Z}^p \Delta \widetilde{Z}^q\right) (\Delta Z^r - \Delta \widetilde{Z}^r)}+ \mathcal{O}(\Delta t^2)\bigg] .
\end{align}
It needs to be shown that the underbraced term (where summation notation is adopted for convenience) is order higher order than 2. This underbraced term vanishes due to Isserlis theorem and symmetry of the distribution and using the inequality established when considering the total law of expectation conditioned on the events $(Z<-A_{\Delta t},|Z|<A_{\Delta t}, Z>A_{\Delta t})$ when $p=q=r$. So that we attain the following estimate  
\begin{align}
\mathbb{E}[||\b q^{n+1}-\tilde{\b q}^{n+1}||_2^2]^{1/2} &\leq \mathcal{O}\left(\Delta t^{\min \left\lbrace1,(1+k)/2\right \rbrace}\right)\label{eq2:estimate_ok}
\end{align}
Then by the FTMSC \cref{thm:fundamental-thm-msc} and \cref{convergence lemma}, \cref{eq2:estimate_ok,eq1:estimate_ok} imply mean square order $p_2-1/2$ convergence when the conditions $p_2 = \min \lbrace (k+1)/2,1\rbrace\geq 1$, $p_1=\min \lbrace k+1-\epsilon,3/2\rbrace \geq 3/2$ are met, $k\geq 1$, is sufficient in the limit of small $\Delta t$ for MS convergence order 1/2. 
\end{proof}

We would also like to remark that this can extend to the generalised setting of SGARK schemes including \cref{method: Stochastic Generalised Additive Runge-Kutta} in a similar manner in \cref{lem: mse GARK} below. 

\begin{lemma}[GARK extension]\label{lem: mse GARK}
Let $\tilde{\b q}^{n+1}$ be the solution driven by bounded increments \cref{Milstein-Tretyakov Bounded normal increments} with $k\geq 1$, and the normally distributed driven system $\b q^{n+1}$, for any GARK scheme, then for $ f^{k} \in C^{2}(\mathbb{R}^d;\mathbb{R})$, $ g^{k}_p \in C^{2}(\mathbb{R}^d; \mathbb{R})$, $\forall p\in \lbrace 1,...,P\rbrace$, $\forall k \in \lbrace 1,...,d\rbrace$. One has the following mean square estimates,
\begin{align}
\mathbb{E}[\b q^{n+1}-\tilde{\b q}^{n+1}] = \mathcal{O}(\Delta t^{\min \lbrace k+1-\epsilon,3/2\rbrace}),\quad (\mathbb{E}[ || \b q^{n+1}-\tilde{\b q}^{n+1}||_2^2])^{1/2}= \mathcal{O}(t^{\min \lbrace (k+1)/2,1\rbrace}),
\end{align}
given $DG|_{\b q}$,$G(\b q)$,$DF|_{\b q}$,$F(\b q)$ are bounded, and all components of all the Butcher-tableau's are component-wise bounded.
\end{lemma}
\begin{proof}
Consider the difference
\begin{align}
\b q^{n+1} - \tilde{\b q}^{n+1} &=  \Delta t \sum_{i=1}^{s^{f}}b^f_i \left( \b f( \b k^f_i)-\b f(\tilde{\b k}^f_i) \right) + \sqrt{\Delta t}\sum_{i=1}^{s^g}b^g_{i} \left( G(\b k^g_i)\Delta \b Z - G(\tilde{\b k}^{g}_{i})\Delta \widetilde{\b Z}\right), 
\end{align}
and then consider Taylor expansions around $\b q$ where
\begin{align}
\b k_{i}^f &= \b q + \Delta t \sum_{j=1}^{s^f}a^{f,f}_{i,j}\b f(\b k^f_j) + \sqrt{\Delta t}\sum_{j=1}^{s^g} a^{f,g}_{i,j}G(\b k^g_j)\Delta \b Z,\quad \tilde{\b k}_{i}^f = \b q + \Delta t \sum_{j=1}^{s^f} a^{f,f}_{i,j}\b f(\tilde{\b k}^f_j) + \sqrt{\Delta t}\sum_{j=1}^{s^g} a^{f,g}_{i,j} G(\tilde{\b k}^g_j)\Delta \widetilde{\b Z}, \\
\b k_{i}^g &= \b q + \Delta t \sum_{j=1}^{s^f} a^{g,f}_{i,j}\b f(\b k^f_j) + \sqrt{\Delta t}\sum_{j=1}^{s^g}a^{g,g}_{i,j}G(\b k^g_j)\Delta \b Z, \quad \tilde{\b k}_{i}^g = \b q + \Delta t \sum_{j=1}^{s^f}a^{g,f}_{i,j}\b f(\tilde{\b k}^f_j) + \sqrt{\Delta t}\sum_{j=1}^{s^g} a^{g,g}_{i,j}G(\tilde{\b k}^g_j)\Delta \widetilde{\b Z},
\end{align}
so that 
\begin{align}
\b q^{n+1}-\tilde{\b q}^{n+1} &= \sqrt{\Delta t}\sum_{i=1}^{s^g}\tilde{b}^g_{i} G(\b q) (\Delta\b Z - \Delta \widetilde{\b Z})\\
&+\Delta t \sum_{i=1}^{s^g}\tilde{b}^g_{i} \left( \sum_{j=1}^{s^g}\tilde{a}^{g,g}_{i,j}DG|_{\b q}(G(\b q)\Delta \b Z)\Delta\b Z - \sum_{j=1}^{s^g}\tilde{a}^{f,g}_{i,j}DG|_{\b q}(G(\b q)\Delta \widetilde{\b Z}) \Delta \widetilde{\b Z}\right) + \mathcal{O}(\Delta t^{3/2}), 
\end{align}
as before in expectation one can show, that the leading order terms are 
\begin{align}
\mathbb{E}[\b q^{n+1}-\tilde{\b q}^{n+1}]
&\leq  (\tilde{b}^{g}\tilde{A}^{g,g}e + \tilde{b}^{g}\tilde{A}^{f,g}e)
DG|_{\b q} G(\b q) P \left(1+\frac{4}{\sqrt{\pi}}\sqrt{k|\ln(\Delta t)|}\right)\Delta t^{k+1}+ \mathcal{O}(\Delta t^{3/2}) \\
&\leq \mathcal{O}(\Delta t^{\min \lbrace k+1 -\epsilon ,3/2\rbrace}),\\
\mathbb{E}[||\b q^{n+1}-\b q^{n+1}||_2^2] &\leq  \Delta t ||G(\b q)||_2^2 \left(\sum_{i=1}^{s^{g}} b^{g}_i\right)^2 \mathbb{E}[|| \Delta \b Z - \Delta \widetilde{\b Z}||_2^2] + \mathcal{O}(\Delta t^2) \leq \mathcal{O}(t^{\min \lbrace k+1,2 \rbrace}).
\end{align}
These are sufficient for the bounded increment-driven system to converge to the same solution as the unbounded increment-driven system, with the same mean square order of 1/2.  
\end{proof}

\section{Practical Methods and Numerical Demonstrations}

\subsection{Practical Methods}\label{sec:Practical Methods}
We state some common strong stability preserving numerical integrators in their canonical Shu-Osher form when we replace the forward Euler method with the Euler Maruyama scheme. The Shu-Osher representation of the numerical schemes is typically computationally efficient in terms of memory allocation and allows easy implementation. \newline

\textbf{SSP Stochastic Runge-Kutta}
\begin{method}[SSP22 Stratonovich (Heun)]\label{method:ssp22}
The two-stage second-order strong stability preserving Stochastic Runge-Kutta method
\begin{align}
\b q^1 &= \operatorname{EM}(\b q^n)\\
\b q^{n+1} &= \frac{1}{2}\b q^n + \frac{1}{2}\operatorname{EM}(\b q^1). 
\end{align}
converges to the Stratonovich form of the equation  \cref{stratonovich}, with weak order 1 using the increments in \cref{eq:2point_random_variable}, and mean square (strong order) 1/2 using the increments \cref{Milstein-Tretyakov Bounded normal increments}. The scheme is SSP in the stochastic setting with a radius of monotonicity $1$.
\end{method}

\begin{method}[SSP33 Stratonovich (Shu-Osher method)]\label{method:ssp33}
\begin{align}
\b q^1 &= \operatorname{EM}(\b q^n)\\
\b q^1 &= \frac{3}{4} \b q^n + \frac{1}{4}\operatorname{EM}(\b q^1)\\
\b q^{n+1} &= \frac{1}{3} \b q^n + \frac{2}{3}\operatorname{EM}(\b q^1)
\end{align}
converges to the Stratonovich form of the equation \cref{stratonovich}, with weak order 1 when using the bounded increments in \cref{eq:2point_random_variable}, and (mean square) strong order 1/2 when using bounded increments in \cref{eq:bounded increments}. And is contractive/SSP when the Euler Maruyama scheme is, with radius of monotonicity $1$.
\end{method}

\begin{method}[SSP104-Stratonovich-(Ketcherson \cite{ketcheson2008highly})]\label{method:SSP104}
Let the time Scaled Euler Maruyama be denoted by 
\begin{align}
\b q^{n+1} &= \operatorname{EM}(a\Delta t,\b q)  = \b q^n + a \b f(\b q^n)\Delta t +  a^{1/2} \sum_{p} \b g_p (\b q^n) \Delta \b S^{p}, 
\end{align}
then the SSP104-Stratonovich-(Ketcherson) method can be defined as follows
\begin{align}
\b q^{0} &= \b q^{n},\\
\b q^{i+1} &= \operatorname{EM}(1/6,\b q^{i}), \quad \text{for} \quad i=0,1,2,3, \\
\b q^{5} &= \frac{3}{5}\b q^{0} +\frac{2}{5}\operatorname{EM}(1/6,\b q^{4}),\\
\b q^{i+1} &= \operatorname{EM}(1/6,\b q^{i}), \quad\text{for} \quad i=5,6,7,8, \\
\b q^{n+1} &= \frac{1}{25}\b q^n + \frac{9}{25}\operatorname{EM}(1/6,\b q^{4}) + \frac{15}{25}\operatorname{EM}(1/6,\b q^{9}).
\end{align}
converges to the Stratonovich form of the equation  \cref{stratonovich}, with weak order 1 using the increments in \cref{eq:2point_random_variable}, and mean square (strong order) 1/2 using the increments in \cref{Milstein-Tretyakov Bounded normal increments}. And is contractive with respect to arbitrary seminorms, when Euler Maruyama is, with radius of monotonicity $6$.
\end{method}
% \begin{method}[ITO-SSP22-1]\todo[inline]{(no)}
% \begin{align}
%     \b q^{1} &= FE(\b q^n)\\
%     \b q^{n} &= 1/2(\b q^n + EM(\b q^{1}))
% \end{align} 
% This misses the Itô Stratonovich correction and converges to the Itô system [..] with weak order 1 when using the increments \cref{eq:2point_random_variable}, but captures the essential (mixed cross deterministic derivatives in higher dimensions) deterministic modelling term $\Delta t^2 /2 ff'$, unlike a usual Euler Maruyama scheme. 
% \end{method}
% \begin{method}[ITO-SSP22-2]\todo[inline]{(no)}
% \begin{align}
%     \b q^{1} &= EM(\b q^n)\\
%     \b q^{n} &= 1/2(\b q^n + FE(\b q^{1}))
% \end{align} 
% This misses the Ito Stratonovich correction, but captures some useful deterministic dynamical terms unlike a usual Euler Maruyama scheme. 
% \end{method}

\textbf{SSP Stochastic Generalised Additive Runge-Kutta}
\begin{method}[SSP- Sequential Operator Splitting]\label{method:SSP- Sequential Operator Splitting}
We consider the following Sequential operator splitting method associated with Strang 
\begin{align}
\b q^{a} &= \operatorname{SSP2m2}(\b q^n,\b f,\Delta t/2),\\
\b q^b &= \operatorname{SSP2n2}(\b q^{a},G,\Delta \b S),\\
\b q^{n+1} &= \operatorname{SSP2m2}(\b q^b,\b f,\Delta t/2).
\end{align}    
SSP2m2 is a $2m$-stage RK method with radius of monotonicity $C = m$, equivalent to m stages of the SSP22 scheme above. Here the omission of $G$, $\Delta \b S$ notationally implies this is a drift-only scheme, and the omission of $\b f, \Delta t$ refers to this being a diffusion only scheme. The total method is SSP with time-step restriction $\min(2m \tau_f , n\tau_g)$. This scheme converges with weak order 1 when using the bounded increments \cref{eq:2point_random_variable}, and mean square (strong) order 0.5 when using increments in \cref{Milstein-Tretyakov Bounded normal increments}.
\end{method}

\begin{method}[SSP-Additive Operator splitting]\label{method:SSP- Additive Operator Splitting}
We consider the following additive operator splitting method (also attributable to Strang)
\begin{align}
\b q^{a} &= \operatorname{SSP2m2}(\b q^n,\b f,\Delta t),\quad
\b q^{b} = \operatorname{SSP2n2}(\b q^n,G,\Delta \b S),\\
\b q^{ba} &= \operatorname{SSP2n2}(\b q^a,G,\Delta \b S),\quad
\b q^{ab} = \operatorname{SSP2m2}(\b q^b, \b f,\Delta t),\\
\b q^{n+1} &= 1/2(\b q^{ab}+\b q^{ba}).
\end{align}
Since this is also a convex combination this method has an SSP timestep criterion given by $\min(m\tau_f,n\tau_g)$.
Therefore one can adjust $m,n$ to attain a scheme of larger monotonicity and larger timestep based on $\b f, G$. This scheme converges with weak order 1 to the Stratonovich system when using the bounded increments \cref{eq:2point_random_variable}, and mean square (strong) order $1/2$ to the Stratonovich system when using increments  \cref{eq:bounded increments}. 
\end{method}

\subsection{Numerical Demonstrations}

\subsubsection{Example 1a: Stochastic 2D Burgers equation, with slope limiters. }\label{sec:example 1a burgers}

Theoretically, we have the following three sufficient properties required for a monotone solution, 

\begin{enumerate}
\item A numerical method with a provable monotonic property for the Euler Maruyama map.
\item Bounded increments. 
\item SSP time integration, with a non-zero radius of monotonicity.
\end{enumerate}  
We wish to numerically test what happens when these conditions are individually violated.
We consider the following 2D extension of the inviscid Burgers equation, 
\begin{align}
d q_t+\left( \left(\frac{1}{2}q^2\right)_x+\left(\frac{1}{2} q^2\right)_y \right) dt + \left( a(q)_x + b(q)_y \right)\circ dW_t= 0,
\end{align}
stochastically translated by a single uniform vector field $(a,b)$ integrated in the Stratonovich sense against a one-dimensional Wiener process. This transforms the 2D-Burgers equation into a stochastic frame of reference, and the solution properties are unchanged from the 2D deterministic case. We shall first describe a numerical method with a specific provable monotonic property for the Euler Maruyama map, consisting of an approximate stochastic Reimann solver and a slope limiter, summarised below.

\begin{method}[FV2 with $N^2(K)\cup N(K)$ limiter \cite{woodfield2024higher}]\hspace{0pt}
\begin{enumerate}
    \item Project cell mean values to pointwise cell centered values $q_{i,j} = \mathcal{P}_2 \bar{q}_{i,j} = \bar{q}_{i,j}$, this direct evaluation is second order $\mathcal{O}(\Delta x^2+\Delta y^2)$.  
    \item Get gradients within each cell $\lbrace q_x, q_y\rbrace$ using the pointwise cell centered values, the second order finite difference stencil $1/2([1,0,-1])^T$.
    \item Create a linear subcell reconstruction within each cell $p_{i,j}(x) = \bar{q}_{i,j} + (x-x_i)q_x + (y-y_i)q_y$ from cell mean values and pointwise reconstructed gradients.
    \item Reconstruct flux contributing edge defined values $q_{i,j}^{R},q_{i,j}^{L},q_{i,j}^{U},q_{i,j}^{D}$ by evaluating the reconstructed polynomial. For example, the evaluation of $p_{i,j}$ at the right edge of cell $(i,j)$, at $(x_{i+1/2},y_{j})$, gives $q^{R}_{i,j} = q_{i,j} + 1/4  (q_{i+1,j} - q_{i-1,j})$.
\item $N^{2} \cup N(K)$-MP Limiting procedure is employed to ensure an edge-defined local maximum principle specified in \cite{woodfield2024higher}.
\item $\mathcal{R}$ Resolve Reimann problem. Consider the Reimann problem at $(x_{i+1/2},y_{j})$. Where to the left of the edge discontinuity $q^{L}_{i+1/2} = q^{R}_{i}$, and to the right edge of the discontinuity we have $q^{R}_{i+1/2} = q^{L}_{i+1}$. This creates a discontinuous initial value problem known as a Reimann problem. This is resolvable exactly for the deterministic Burger's equation (Godunov's approach), and here in the stochastic case will be treated with a stochastic extension of a Local Lax Friedrich flux \cref{def:LLF}.
\item $\mathcal{F}$ Flux computation, quadrature, this is done through second order Gauss quadrature. 
\item $\mathcal{E}$ Evolve, the cell mean value by the fluxes, in a flux form forward Euler or Euler-Maruyama stage. 
\end{enumerate}
\end{method}
Whilst it may be possible to solve the exact Reimann problem for this particular problem, approximate Reimann solvers are cheaper and widely adopted for more complicated systems. 

\begin{definition}[LLF-EM-Stochastic Burgers equation Flux]\label{def:LLF}
Define the stochastic Euler Maruyama flux functions, as follows, using ideas from Kurganov and Tadmore \cite{kurganov2000new}, 
\begin{align}
 \mathbb{F}(q,a) = \frac{1}{2} q^2 + a q \frac{\Delta W_t}{\Delta t},
\end{align}
Compute the stochastic Euler-Maruyama maximum wave speed
\begin{align}
    \alpha_{i+1/2} = \max_{q\in \lbrace q^L_{i+1/2},q^R_{i+1/1}\rbrace} \left| \frac{d}{dq}\mathbb{F}(q,a)\right| = \max \left\lbrace \left| a \Delta W /\Delta t + q^L_{i+1/2}\right|, \left|a \Delta W /\Delta t +q^R_{i+1/2} \right| \right\rbrace. \label{eq:alpha wavespeed}
\end{align}
% \begin{align}
%     \beta_{i+1/2} = \max_{q\in \lbrace q^L_{i+1/2},q^R_{i+1/1}\rbrace} \left| \frac{d}{dq}\mathbb{F}(q,a)\right| = \max \left\lbrace \left| a \Delta W /\Delta t \right| + \left|q^L_{i+1/2}\right|, \left|a \Delta W /\Delta t \right|+\left|q^R_{i+1/2} \right| \right\rbrace,
% \end{align}
and use these to define the stochastic Euler-Maruyama Local Lax Friedrich (Rusanov) flux 
\begin{align}
f_{i+1/2} &= \frac{1}{2}\left(\mathbb{F}(q^L_{i+1/2},a) + \mathbb{F}(q^R_{i+1/2},a)\right) - \frac{1}{2}\alpha_{i+1/2}(q^R_{i+1/2} - q^L_{i+1/2}).
\end{align}
\end{definition}

This defines an EM flow map for the stochastic Burgers equation. The notion of monotonicity is that the EM numerical flow map is a monotonically nondecreasing function of quadrature point evaluations, and the slope limiting enforces a specific discrete local maximum principle. More specifically, the numerical method can be decomposed into 4 separate 3-point HHLK monotone schemes at each flux-contributing quadrature point. This is done by decomposing the cell mean value into 4 edge-defined points using its linear subcell representation, for example, the right edge takes the following form
\begin{align}
H_{i+1/2} &= \frac{1}{4}q^R_{i+1/2} - \Delta t F_{i+1/2}(q^L_{i+1/2},q^R_{i+1/2})  \\
F_{i+1/2}(q^L_{i+1/2},q^R_{i+1/2})  &= 1/2\left(\mathbb{F}(q^L,a) + \mathbb{F}(q^R,a)- \alpha (q^R -q^L)\right)
\end{align}
where $\alpha$ is a yet to be defined constant. The numerical flux is monotone in the sense $\partial_{q^{L}} F_{i+1/2} \geq 0$, $\partial_{q^{R}} F_{i+1/2} \leq 0$ if one then defined $\alpha$ to be as \cref{eq:alpha wavespeed}. Therefore, the three-point scheme $H_{i+1/2}$ is monotonically increasing in terms of quadrature points $\frac{\partial H}{\partial q^R}, \frac{\partial H}{\partial q^L} \geq 0$. Therefore the Euler Maruyama scheme is a monotone function of the quadrature point evaluations. If the $N^2(K)\cup N(K)$-Limiter is employed, the quadrature points themselves are bounded in terms of locally defined cell mean values in such a way one has the resulting local maximum principle of the resulting scheme
\begin{align}
\bar{q}^{n+1}_{i,j}\in \left[\min_{i,j\in S_{ij}}\bar{q}^{n}_{i,j} \max_{i,j\in S_{ij}}\bar{q}^{n}_{i,j} \right]
\end{align}
where for a 2D Cartesian mesh one has the set of face-sharing neighbours and their corresponding face-sharing neighbours describes a 13-point stencil. The SSP22 time integration turns this into a slightly wider different local maximum principle (as elaborated on in the appendix of \cite{woodfield2024new}), required when establishing internal local maximum principles for the substages.

The numerical method described above is not proven to be convergent or even claimed to be sensible for this equation, the method proposed simply remains range-bounded, in the sense
\begin{align}
\bar{q}^{n+1}_{i,j}\in \left[\min_{i,j}\bar{q}^{n}_{i,j} \max_{i,j}\bar{q}^{n}_{i,j} \right].
\end{align}

Nevertheless, we wish to demonstrate that the solution will be bounded, provided SSP integration, limiting and Bounded increments are used, and wish to investigate when these conditions are individually violated. To demonstrate this we consider the following cases. 
\begin{enumerate}
    \item SSP22 + Limiter + Bounded increments. 
    \item SSP22 + Limiter + Unbounded increments. 
    \item NON SSP integration + Limiter + Bounded increments.
    \item SSP22 + Unlimited + Bounded increments. 
\end{enumerate}

Case one is theoretically monotone. In case two, the unbounded increments could produce a timestep larger than the radius of monotonicity. In case three a non-SSP integration method has no theoretical guarantees of monotone behaviour. In case 4 without nonlinear limiting strategies, the underlying numerical method is not guaranteed to be monotone.

The numerical setup has the following parameters. We use  mesh resolution $n_x=128,n_y=128,n_t = 512$, on the space time interval $[0,1]\times[0,1]\times[0,1/2]$. We use a small stochastic basis of noise given by $(a,b) = 1/256(1,1)$. We use the discontinuous square initial condition given by
\begin{align}
q_0=\begin{cases}
        1,\quad \text{where}\quad  (x,y)\in [0.1,0.6]\times[0.1,0.6],\\
        0, \quad \text{else}.
    \end{cases}
\end{align}
For the unbounded increments, we use the increments $\Delta W\sim N(0,\Delta t)$. For the bounded increments we use the two-point bounded increments from \cref{eq:2point_random_variable}. For the timestepper, we use SSP22 \cref{method:ssp22} with a radius of monotonicity 1 or the RK2 method with no radius of monotonicity. For the limiter, we use the $N^2(K)\cup N(K)$- limiter. We plot the final time solution in \cref{fig:Burgers}, and we also plot the maximum and minimum as a function of time in \cref{fig:Burgers}.

\begin{figure}[H]
\centering
\begin{subfigure}[t]{0.24\textwidth}
\centering
\includegraphics[width=.95\textwidth]{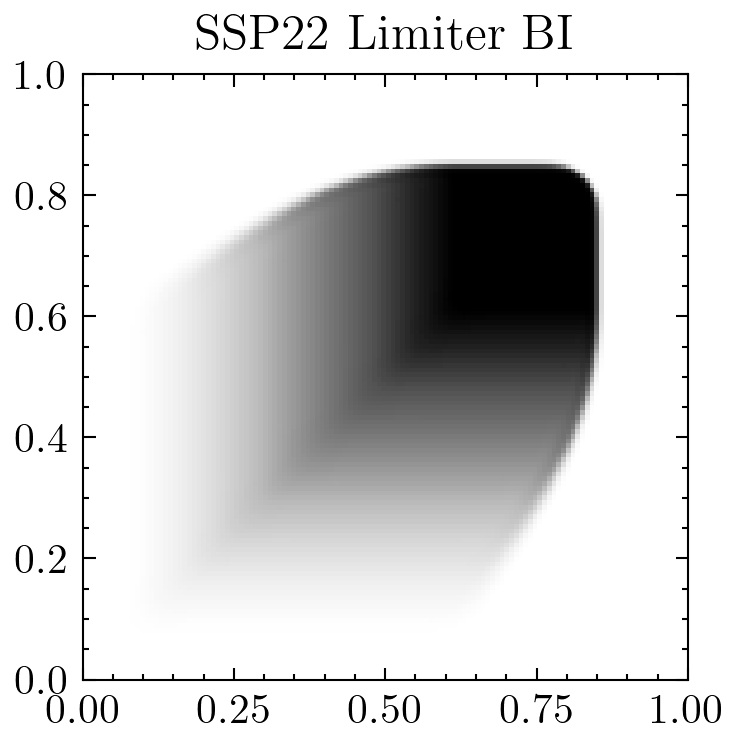}\caption{\hfill}\label{fig:SSP22_Limiter_BI_Final}
\end{subfigure}
\begin{subfigure}[t]{0.24\textwidth}
\centering
\includegraphics[width=.95\textwidth]{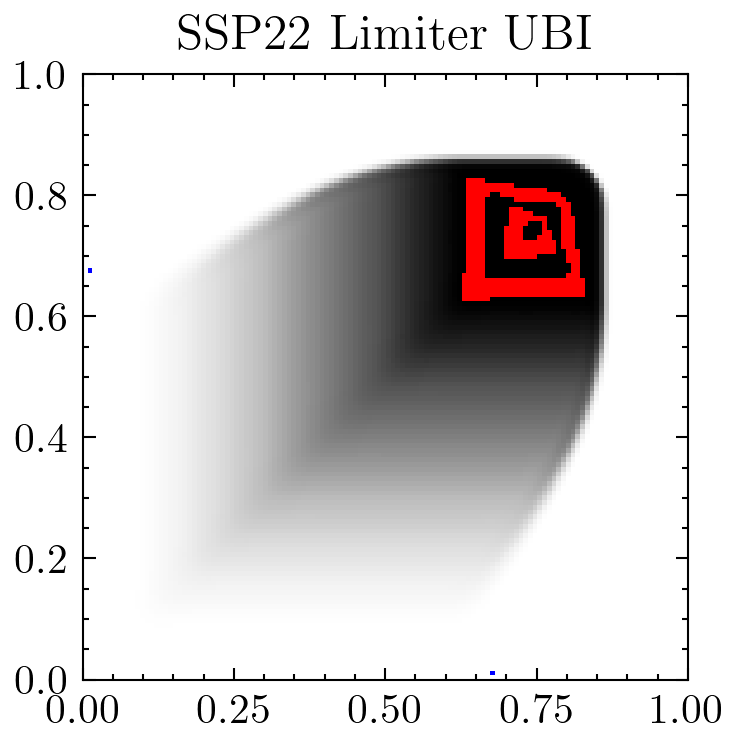}
\caption{\hfill}
\label{fig:SSP22_Limiter_UBI_Final}
\end{subfigure}
\begin{subfigure}[t]{0.24\textwidth}
\centering
\includegraphics[width=.95\textwidth]{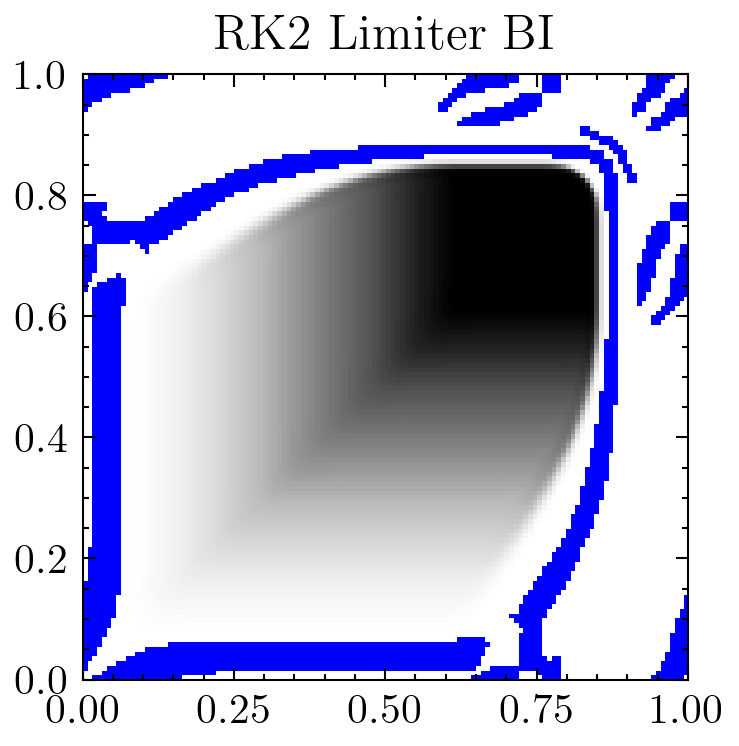}
\caption{\hfill}
\label{fig:RK2_Limiter_BI_Final}
\end{subfigure}
\begin{subfigure}[t]{0.24\textwidth}
\centering
\includegraphics[width=.95\textwidth]{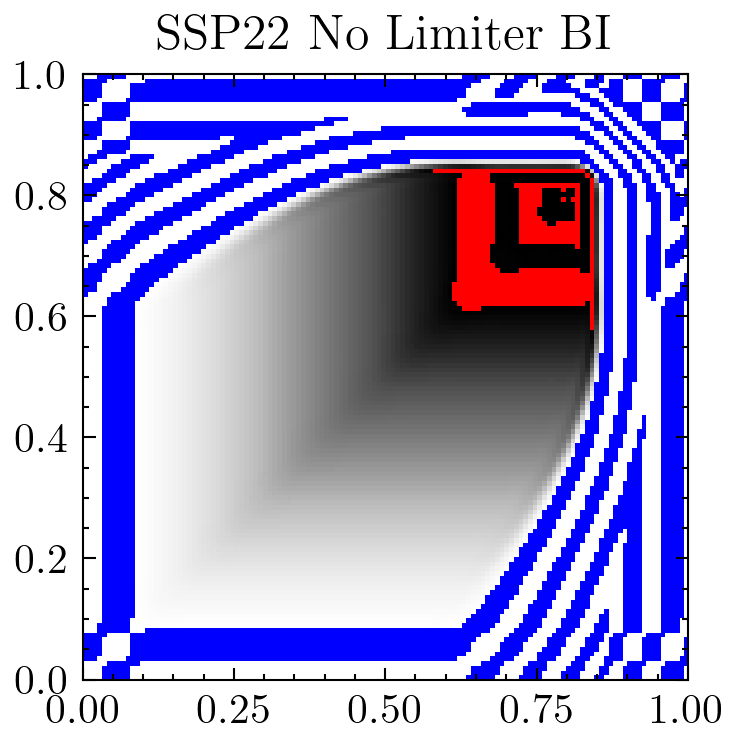}
\caption{\hfill}
\label{fig:SSP22_NO_Limiter_BI_Final}
\end{subfigure}\\
\begin{subfigure}[t]{0.24\textwidth}
\centering
\includegraphics[width=.95\textwidth]{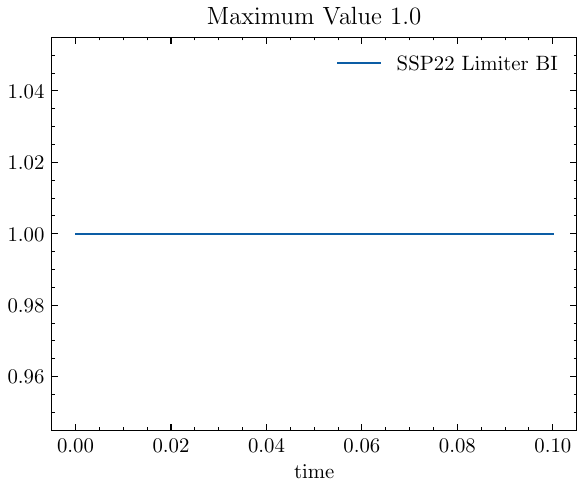}\caption{\hfill}\label{fig:SSP22_Limiter_BI_max}
\end{subfigure}
\begin{subfigure}[t]{0.24\textwidth}
\centering
\includegraphics[width=.95\textwidth]{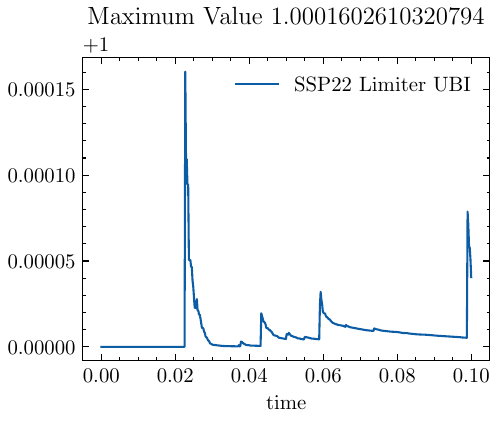}
\caption{\hfill}
\label{fig:SSP22_Limiter_UBI_max}
\end{subfigure}
\begin{subfigure}[t]{0.24\textwidth}
\centering
\includegraphics[width=.95\textwidth]{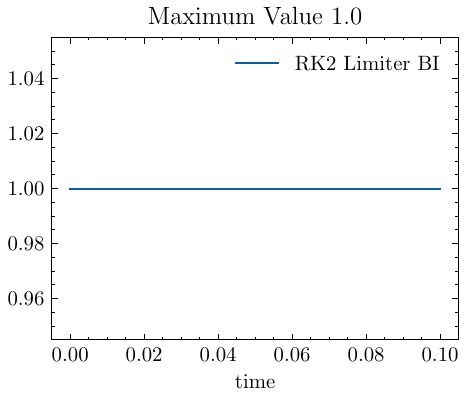}
\caption{\hfill}
\label{fig:RK2_Limiter_BI_max}
\end{subfigure}
\begin{subfigure}[t]{0.24\textwidth}
\centering
\includegraphics[width=.95\textwidth]{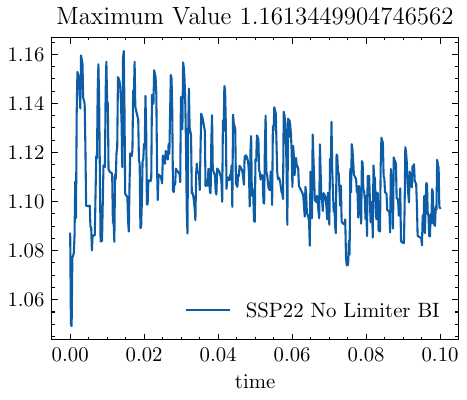}
\caption{\hfill}
\label{fig:SSP22_No_Limiter_BI_max}
\end{subfigure}\\
\begin{subfigure}[t]{0.24\textwidth}
\centering
\includegraphics[width=.95\textwidth]{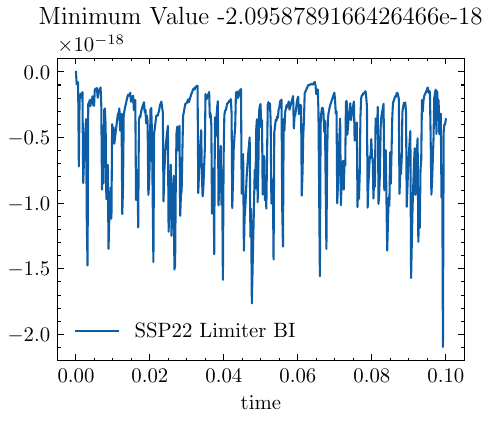}\caption{\hfill}\label{fig:SSP22_Limiter_BI_min}
\end{subfigure}
\begin{subfigure}[t]{0.24\textwidth}
\centering
\includegraphics[width=.95\textwidth]{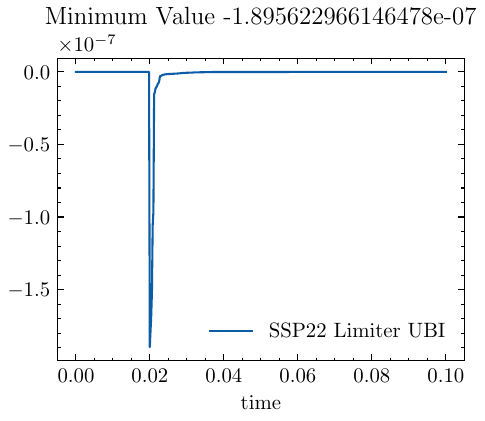}
\caption{\hfill}
\label{fig:SSP22_Limiter_UBI_min}
\end{subfigure}
\begin{subfigure}[t]{0.24\textwidth}
\centering
\includegraphics[width=.95\textwidth]{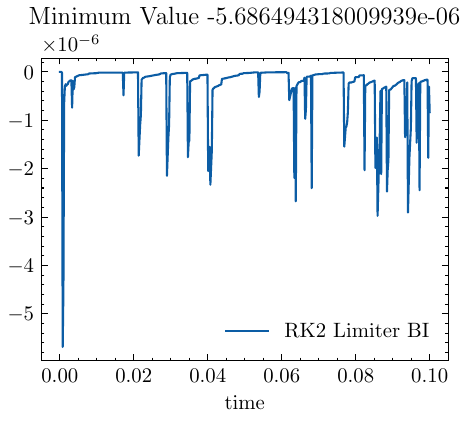}
\caption{\hfill}
\label{fig:RK2_Limiter_BI_min}
\end{subfigure}
\begin{subfigure}[t]{0.24\textwidth}
\centering
\includegraphics[width=.95\textwidth]{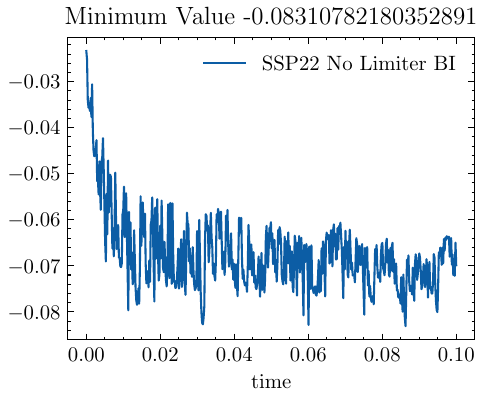}
\caption{\hfill}
\label{fig:SSP22_No_Limiter_BI_min}
\end{subfigure}
\caption{
We plot the final time plot, of the stochastic Burgers equation in row 1. Red indicates maxima over 1, blue indicates minima under 0 (to machine precision). We plot the maximum value as a function in time, in row 2. We plot the minimum value as a function in time, in row 3. In column 1 we plot the SSP22-Limiter-BI system, in column 2 we plot the SSP-Limiter with unbounded increments. In column 3 we plot the RK2 limited scheme with bounded increments. In column 4 we plot the SSP22 unlimited scheme with bounded increments.
}
\label{fig:Burgers}
\end{figure}

Results of Example 1a.  In the first row of \cref{fig:Burgers}, we plot the final timestep solution of all 4 methods. Where it is observed that SSP time integration, slope limiting and using bounded increments, were all required for a range-bounded numerical solution. In \cref{fig:SSP22_Limiter_BI_Final,fig:SSP22_Limiter_BI_max,fig:SSP22_Limiter_BI_min} we see that SSP22 with limiting and bounded increments makes no global maxima or minima violations over the entire time window. In \cref{fig:SSP22_Limiter_UBI_max,fig:SSP22_Limiter_UBI_min} we see that the SSP22 scheme with limiting using unbounded increments generated maxima and minima violations of order $10^{-7}, 10^{-4}$ respectively. In \cref{fig:RK2_Limiter_BI_max,fig:RK2_Limiter_BI_min} the RK2 Limiter with bounded increments did not generate maxima violations but generated $10^{-6}$ sized minima violations. The SSP22 scheme with no limiter with bounded increments generated $10^{-1}$ sized maxima and minima violations. In terms of the magnitude of the monotonicity violations, the absence of a slope limiter produced the largest range boundedness violations, using unbounded increments produced the second largest, and using a non-SSP timestepping method produced the smallest range boundedness violations.

\subsubsection{Example 1b: sufficient not necessary}\label{sec:Example 1b}
The previous example demonstrates practical merit to each of the three sufficient conditions for monotonic solutions. This example indicates that these conditions may not be strictly necessary in all cases. To demonstrate the practical importance of using a Stochastic SSP method, we follow \cite{ketcheson2005practical}. The RK method defined as 
\begin{align}
\b k^{(1)} &= \b q^n + a_{12} \Delta t \b f(\b q^n),\\
\b q^{n+1} &= \b q^n +b_1\Delta t \b f(\b q^n)+b_2 \Delta t\b f(\b k^{(1)}),
\end{align}
is second order when $a_{12}=\frac{1}{2\gamma}$ $b_1=1-\gamma$, $b_2=\gamma$. Its stochastic extension as a Stratonovich converging Stochastic Runge-Kutta \cref{method: Stochastic Runge-Kutta} is 
\begin{align}
\b k^{(1)} &= \b q^n + a_{12}\Delta t \b f(\b q^n)+ a_{12} G(\b q^n)\Delta \b S^n, \\
\b q^{n+1} &= \b q^n +b_1\Delta t \b f(\b q^n)+ b_1  G(\b q^n)\Delta \b S^n + b_2 \Delta t\b f(\b k^{(1)}) + b_2 G(\b k^{(1)})\Delta \b S^n.
\end{align}
Dependent on the choice of $\gamma$ the above scheme has different SSP properties. When $\gamma=1/2$ the SSP coefficient is $C=1$ and the scheme is the SSP22 scheme. When $\gamma = -1/40$ the SSP coefficient is $C=0$, and this particular scheme is often (\cite{gottlieb1998total}) used to advocate for the use of SSP methods. For this particular numerical scheme and SPDE, this $\gamma = -1/40$ method blows up entirely. Whilst, this is evidence indicating the merit of SSP time integration over non-SSP integration there may be fairer comparisons. In practice, the $\gamma = -1/40$ scheme is not implemented, it has an unusually large truncation error \cite{ketcheson2005practical}. We will instead test the schemes when $\gamma = 1/4$, $C=1/2$, and when $\gamma = 3/4$, $C=1/2$. These are numerically viable Runge-Kutta methods, in particular when $\gamma = 3/4$ the scheme has the minimum truncation error. We will use bounded increments and limiters such that the SSP22 scheme method is provably monotone, but the $\gamma = 3/4$, $\gamma = 1/4$, are not provably monotone through the SSP property as they are run slightly beyond the radius of monotonicity. The experiment setup is the same as the previous example, and the results are also presented similarly in \cref{fig:Lucky}.

\begin{figure}[H]
\centering
\begin{subfigure}[t]{0.16\textwidth}
\includegraphics[width=.95\textwidth]{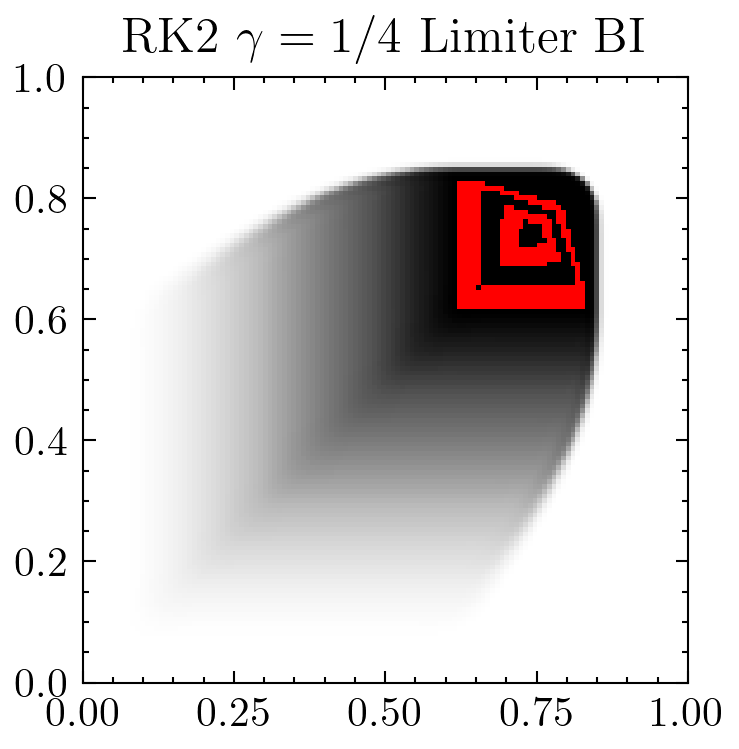}
\caption{\hfill}
\label{fig:2D_RK2_a_Limiter_BI_Final}
\end{subfigure}
\begin{subfigure}[t]{0.16\textwidth}
\includegraphics[width=.95\textwidth]{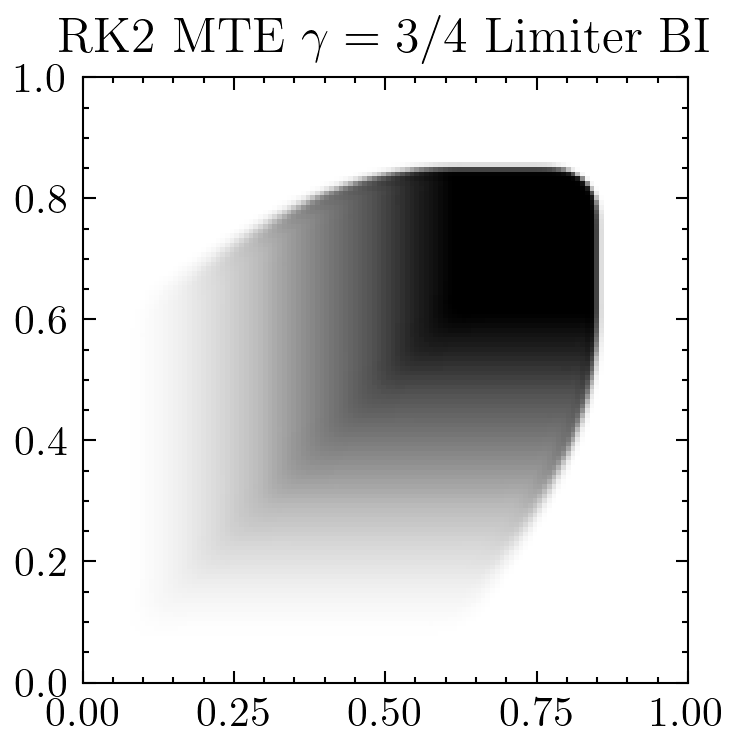}
\caption{\hfill}
\label{fig:2D_RK2_MTE_Limiter_BI_Final}
\end{subfigure}
\begin{subfigure}[t]{0.16\textwidth}
\includegraphics[width=.95\textwidth]{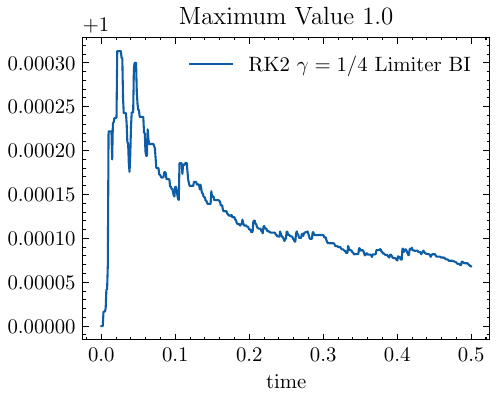}
\caption{\hfill}
\label{fig:RK2_a_Limiter_BI_max}
\end{subfigure}
\begin{subfigure}[t]{0.16\textwidth}
\includegraphics[width=.95\textwidth]{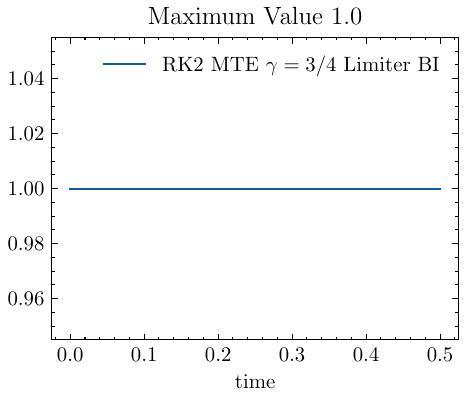}
\caption{\hfill}
\label{fig:RK2_MTE_Limiter_BI_max}
\end{subfigure}
\begin{subfigure}[t]{0.16\textwidth}
\includegraphics[width=.95\textwidth]{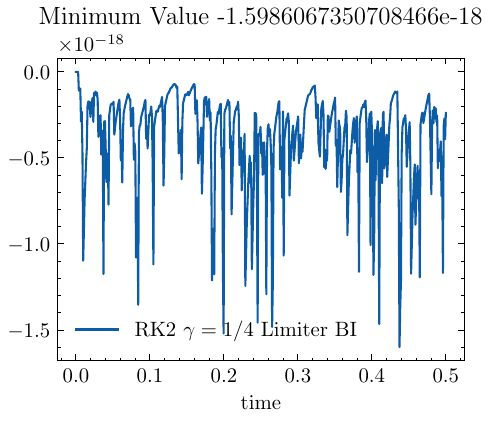}
\caption{\hfill}
\label{fig:RK2_a_Limiter_BI_min}
\end{subfigure}
\begin{subfigure}[t]{0.16\textwidth}
\includegraphics[width=.95\textwidth]{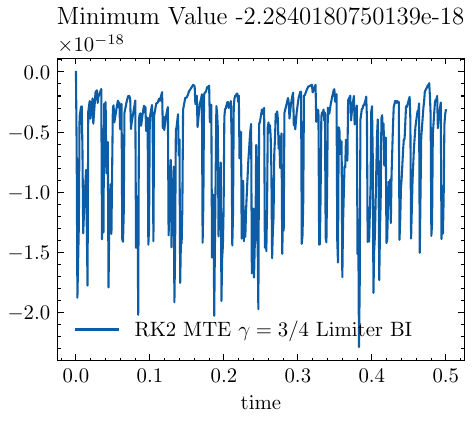}
\caption{\hfill}
\label{fig:RK2_MTE_Limiter_BI_min}
\end{subfigure}
\caption{In \cref{fig:2D_RK2_a_Limiter_BI_Final,fig:2D_RK2_MTE_Limiter_BI_Final} we plot the final time solution of the $\gamma = 1/4$, $\gamma = 3/4$ respectively. In \cref{fig:RK2_a_Limiter_BI_max,fig:RK2_MTE_Limiter_BI_max} the maximum value is plotted with time for the $\gamma = 1/4$, $\gamma = 3/4$ respectively.}\label{fig:Lucky}
\end{figure}

In \cref{fig:2D_RK2_MTE_Limiter_BI_Final,fig:RK2_Limiter_BI_max,fig:RK2_Limiter_BI_min}, one finds that for one particular scheme, in one particular example, running beyond the radius of monotonicity (using a Non-SSP timestepper) was not a strictly necessary condition for the range boundedness preservation in practice. Yet in almost the same setup with a numerical scheme with the same SSP property, the lack of the SSP property in another scheme allowed a non-monotone solution in \cref{fig:RK2_MTE_Limiter_BI_max,fig:RK2_MTE_Limiter_BI_min}. 

Similarly, one could also imagine choosing a particular finite realisation of Brownian motion sufficiently bounded that one does not observe monotonicity violations numerically. Indicating sampling from bounded distributions is not a strictly necessary condition for finite numerical examples.

\subsubsection{Example 1c: 2D Advection}\label{sec:Example 1c: 2D Advection}
In this experiment, we solve the 2D stochastic advection problem 
\begin{align}
d q + \operatorname{div}(\b u q)dt +  \operatorname{div}(\b \xi q)\circ dW= 0,\quad q(0,x) = q_0(\b x)
\end{align}
at resolution $n_x,n_y,n_t = 128,128,512$, on a $[0,1]^3$ space-time mesh subject to the  initial conditions, 
\begin{align}
	q_{0}  = 
	\begin{cases}
			1 & \text{for}\quad  r_{zal} = \sqrt{ (x-0.5)^2+(y-0.75)^2} \leq 0.15, \quad\text{and}\quad (x\leq 0.475), \\
			1 & \text{for}\quad r_{zal} \leq 0.15, \quad\text{and}\quad (x>0.525), \\
			1 & \text{for}\quad r_{zal} \leq 0.15, \quad \text{and}\quad (y\geq 0.85),\quad\text{and}\quad (0.475<x\leq 0.525), \\
			(1-\frac{r_{cone}}{0.15}) &\text{for}\quad (r_{cone} = \sqrt{(x-0.5)^2+(y-0.25)^2}\leq 0.15), \\
			\frac{1}{2}(1+\cos(\pi  \frac{r_{cos}}{0.15}) & \text{for}\quad (r_{cos} = \sqrt{(x-0.25)^2+(y-0.5)^2}\leq 0.15),\\
			0&\mbox{otherwise.} 
	\end{cases}\label{test:LeVeque} 
\end{align}
specified in 
\cite{leveque1996high}. The incompressible vector fields $\b u, \b \xi$ are specified as follows
\begin{align}
(u^x,u^y) = (-2\pi(y - 1/2) , 2\pi(x -1/2)),\quad
(\xi^x,\xi^y) = \frac{2\pi}{10}\left(x(x - 1)(2 y - 1) , -(2x - 1)y(y - 1) \right).
\end{align}

The deterministic velocity $(u^x,u^y)$ is solid body rotation, and the advection noise $(\xi^x,\xi^y)$ is deformational but incompressible.
The solution is theoretically range bounded when using the $N^2(K)\cup N(K)$-limiter \cite{woodfield2024higher} and the SSP104 \cref{method:SSP104}, with the bounded three-point increments described in \cref{eq:3point_rv}. The diffusion term in the SPDE causes the shape to deform.
We plot 16 ensemble members at the final timestep in \cref{fig:sbr104}, where range boundedness is observed in practice, as is theoretically expected.

\begin{figure}[H]
\centering
\begin{subfigure}[t]{0.1\textwidth}
\includegraphics[width=.95\textwidth]{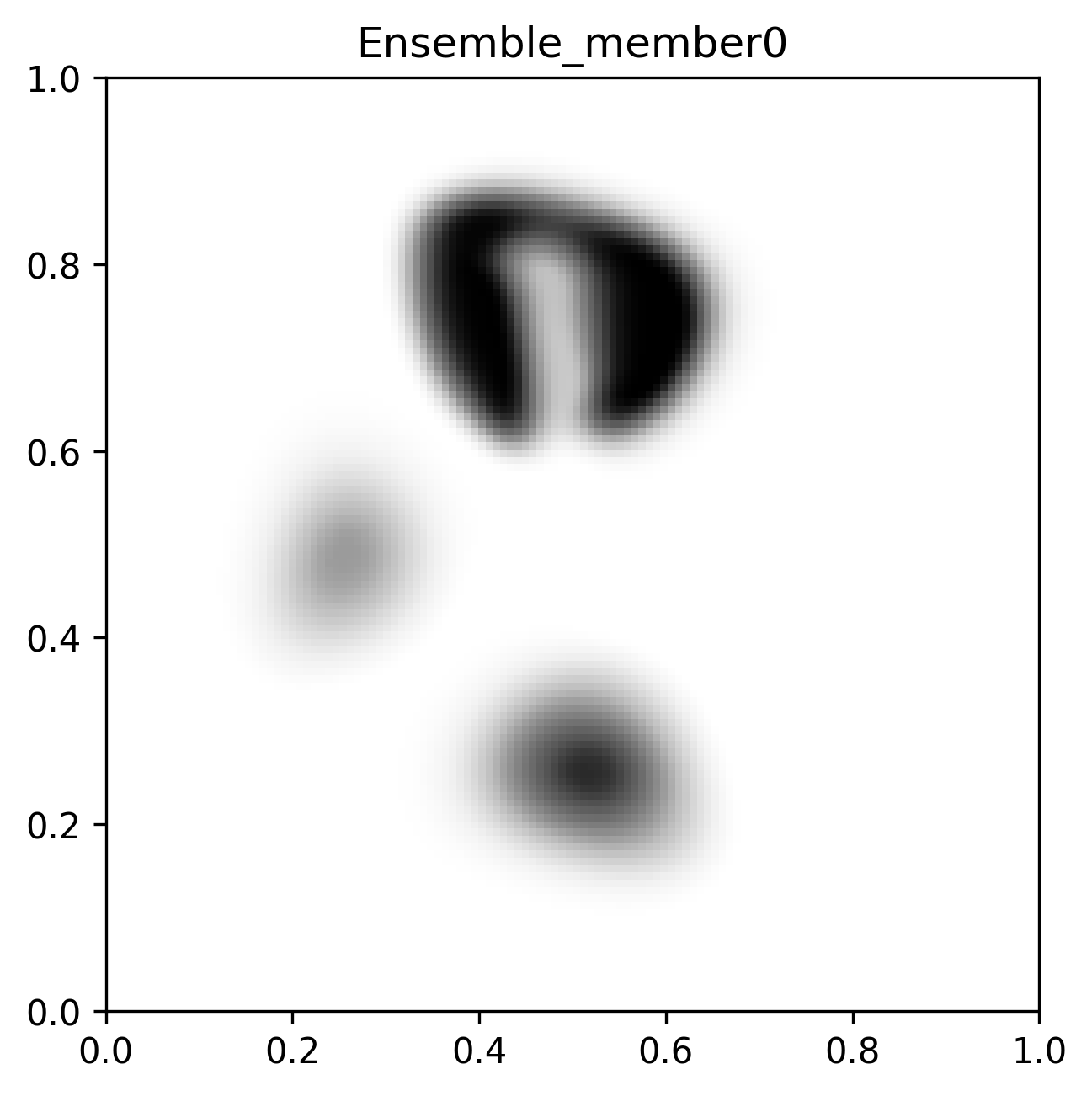}
\end{subfigure}
\begin{subfigure}[t]{0.1\textwidth}
\includegraphics[width=.95\textwidth]{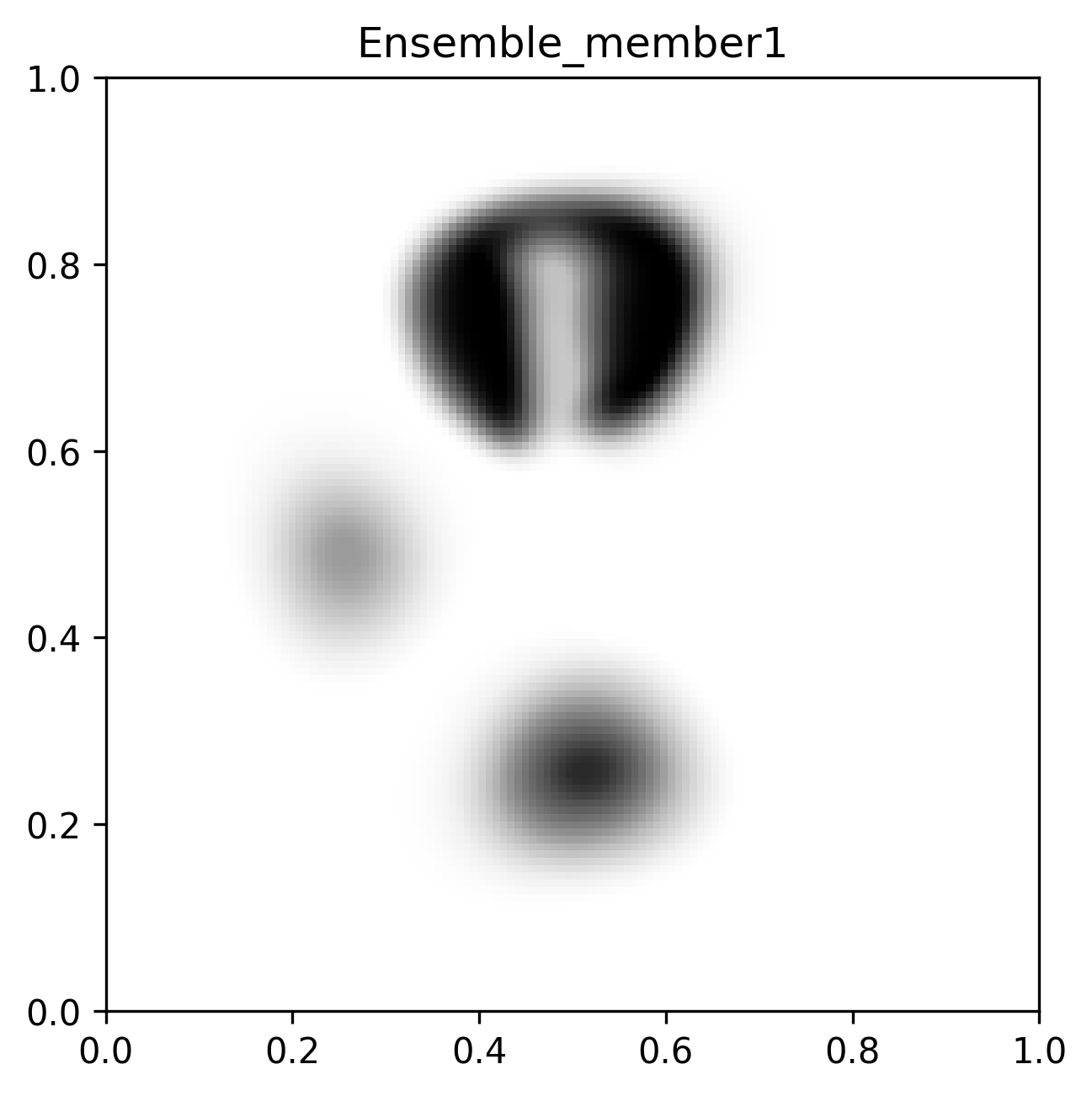}
\end{subfigure}
\begin{subfigure}[t]{0.1\textwidth}
\includegraphics[width=.95\textwidth]{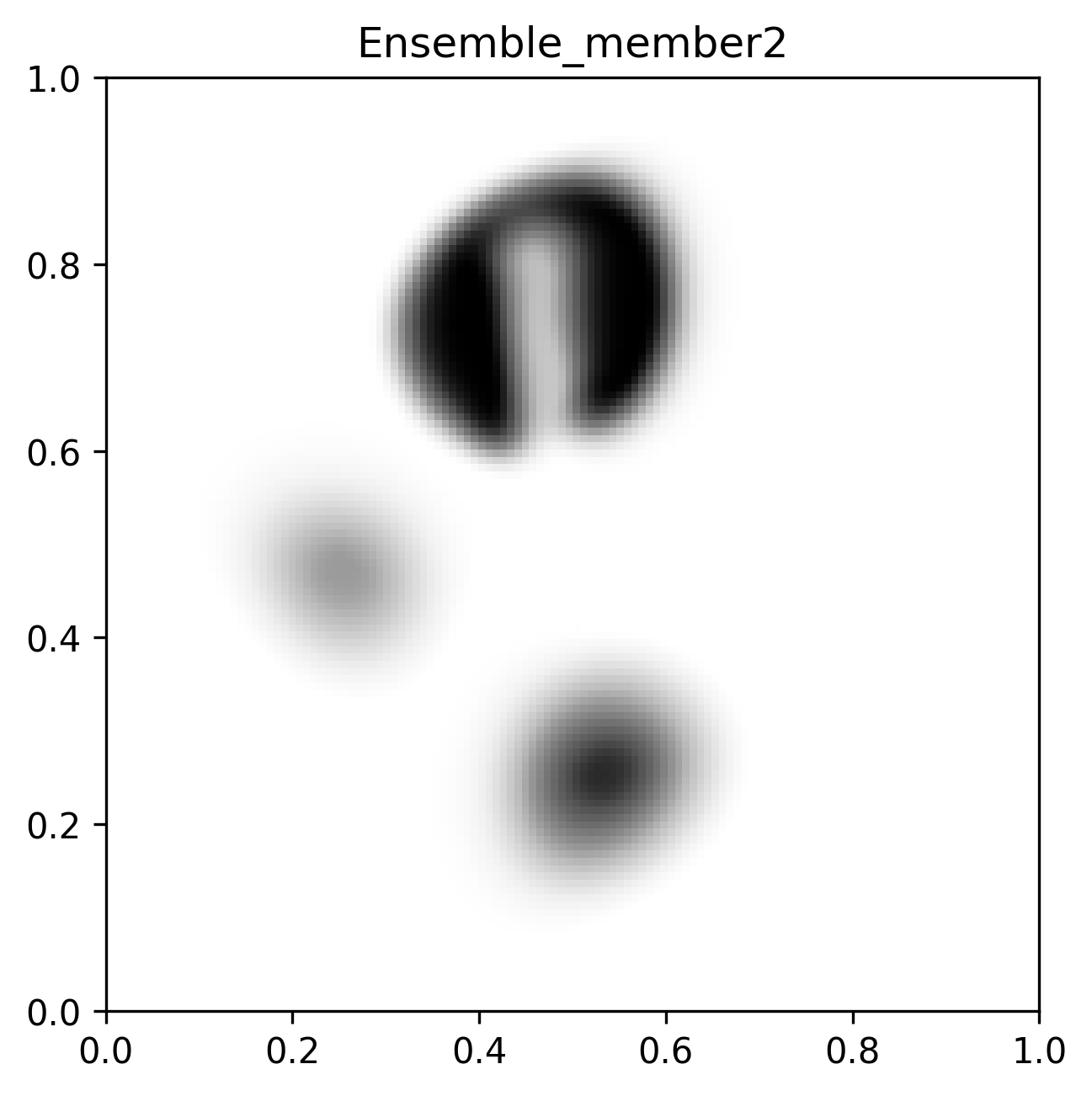}
\end{subfigure}
\begin{subfigure}[t]{0.1\textwidth}
\includegraphics[width=.95\textwidth]{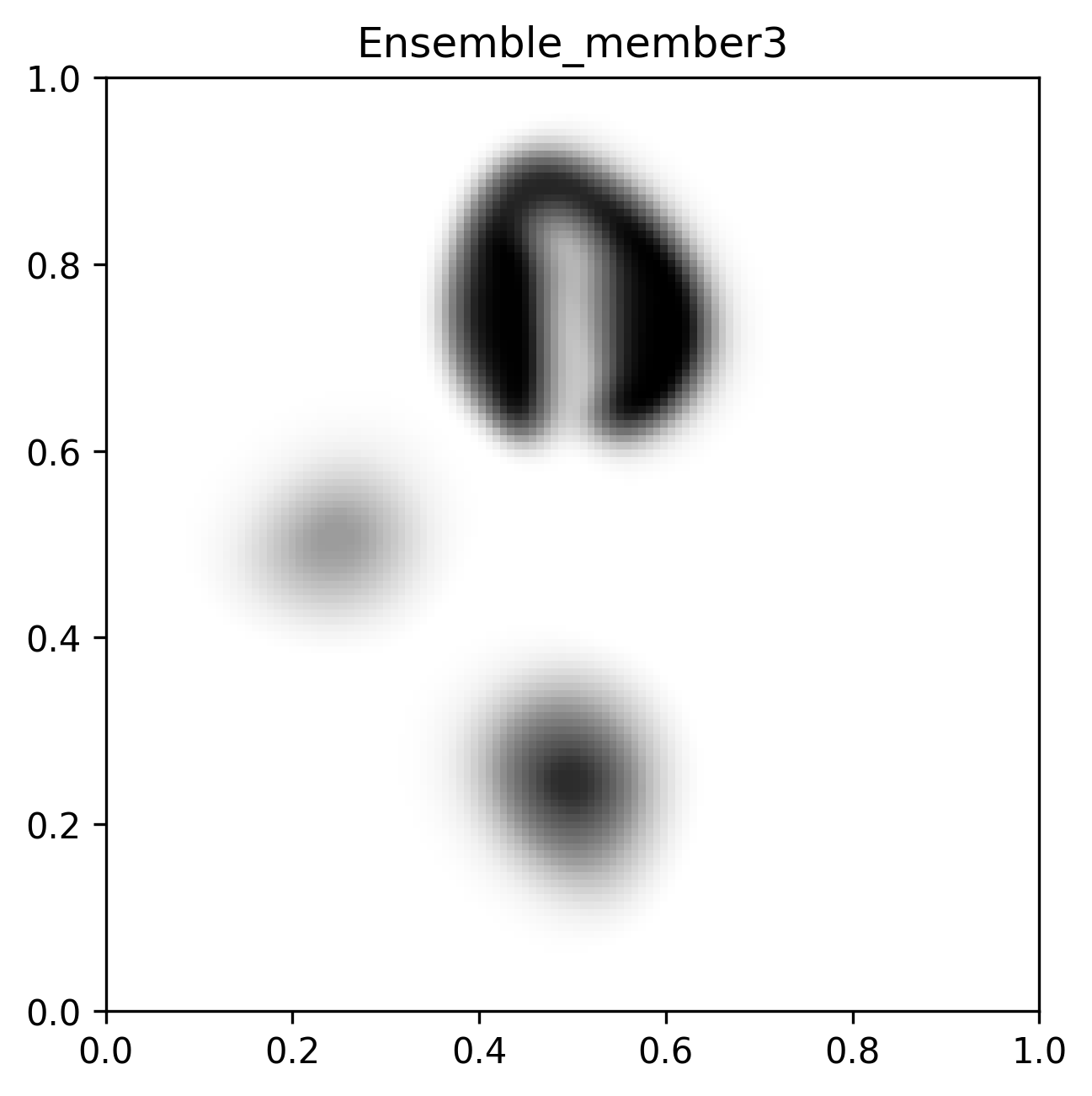}
\end{subfigure}
\begin{subfigure}[t]{0.1\textwidth}
\includegraphics[width=.95\textwidth]{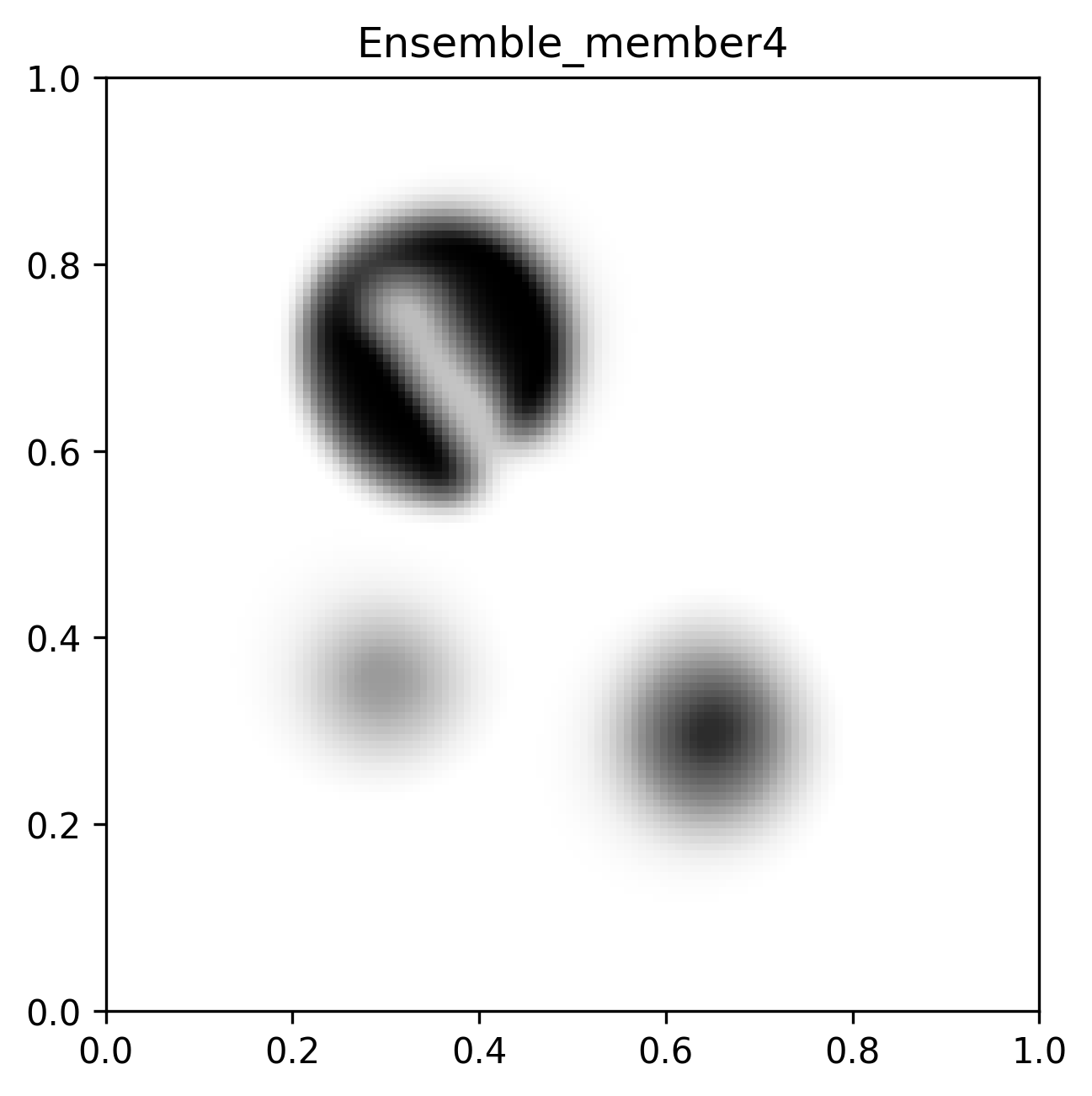}
\end{subfigure}
\begin{subfigure}[t]{0.1\textwidth}
\includegraphics[width=.95\textwidth]{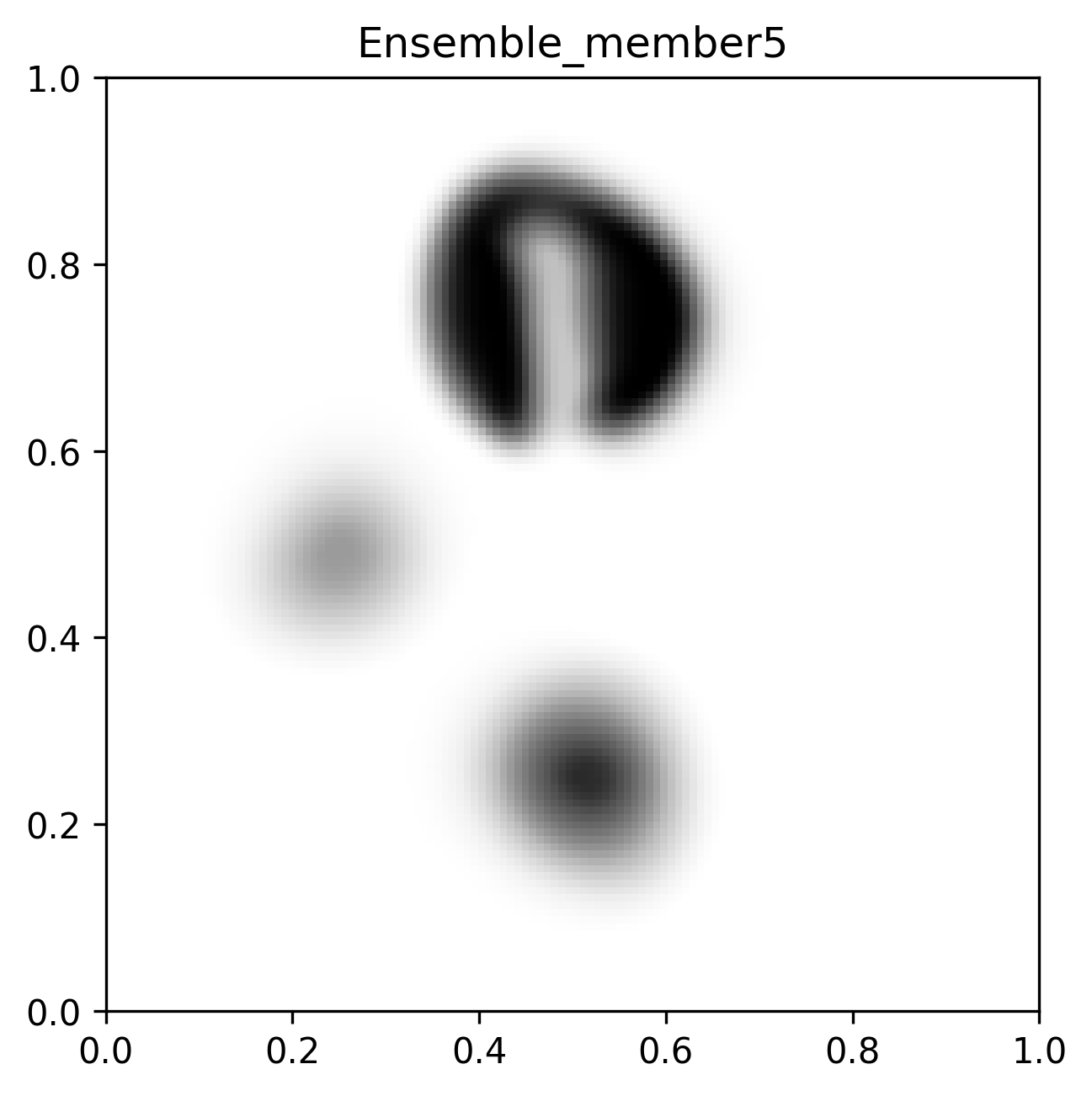}
\end{subfigure}
\begin{subfigure}[t]{0.1\textwidth}
\includegraphics[width=.95\textwidth]{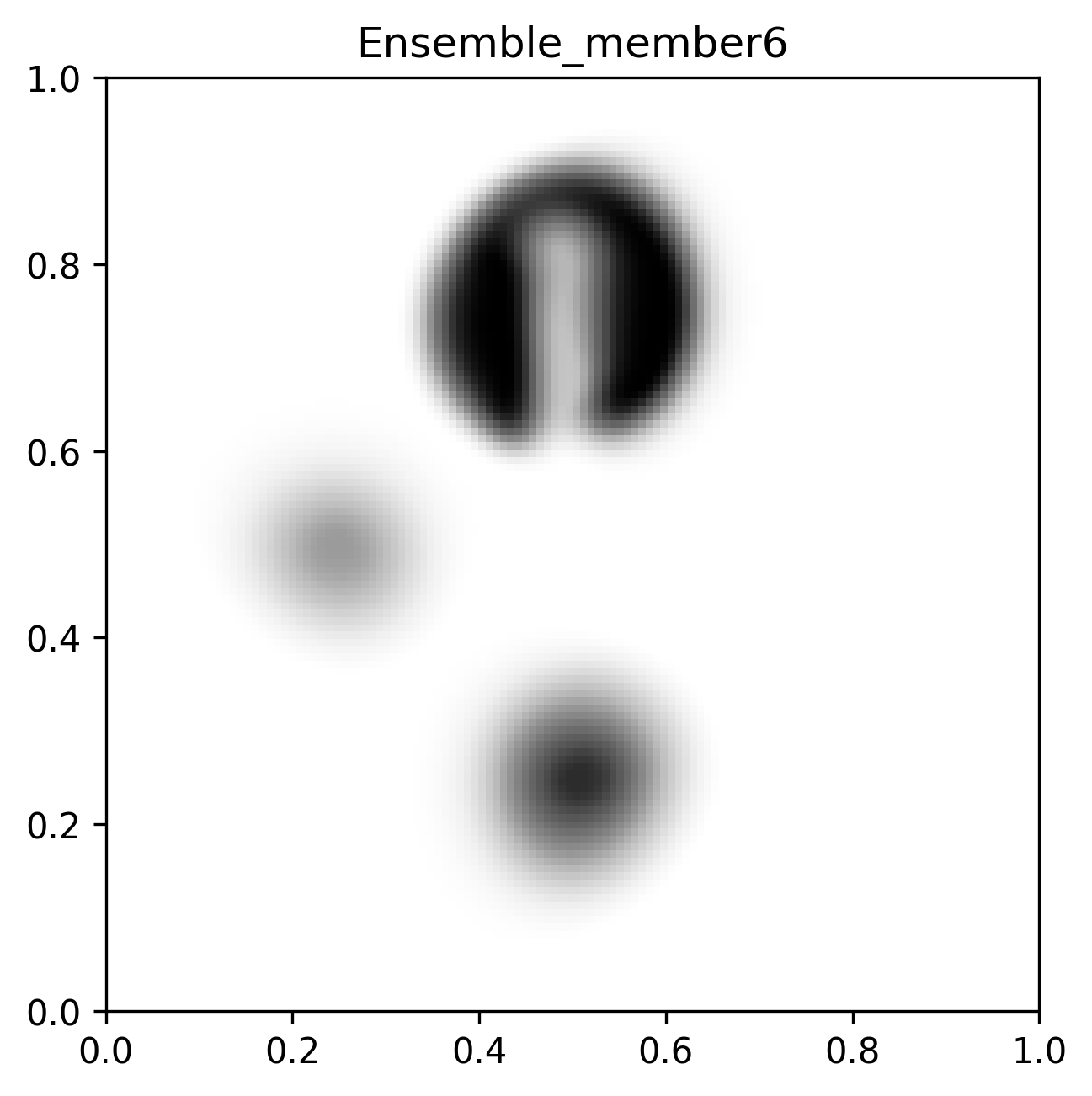}
\end{subfigure}
\begin{subfigure}[t]{0.1\textwidth}
\includegraphics[width=.95\textwidth]{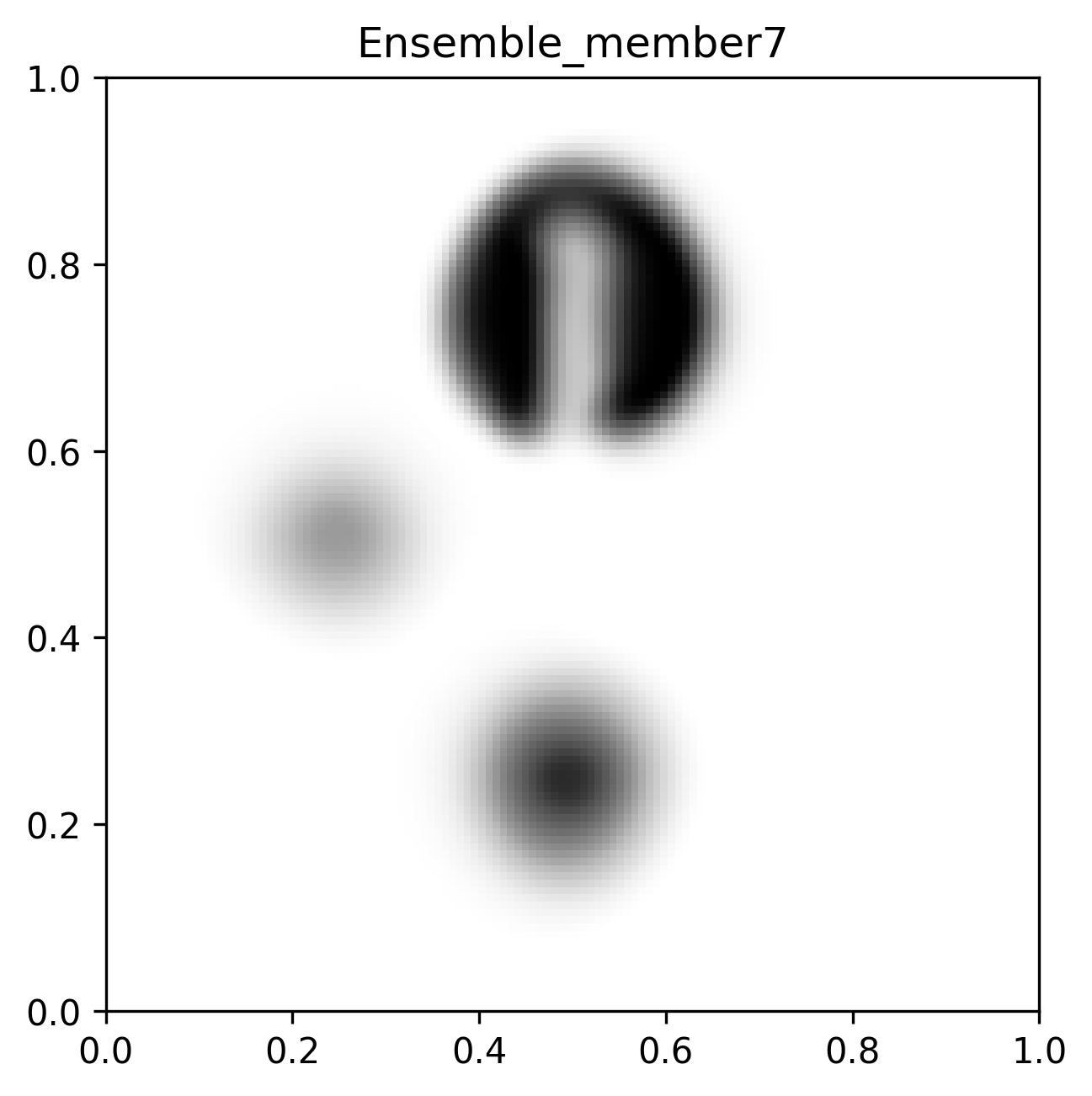}
\end{subfigure}\\
\begin{subfigure}[t]{0.1\textwidth}
\includegraphics[width=.95\textwidth]{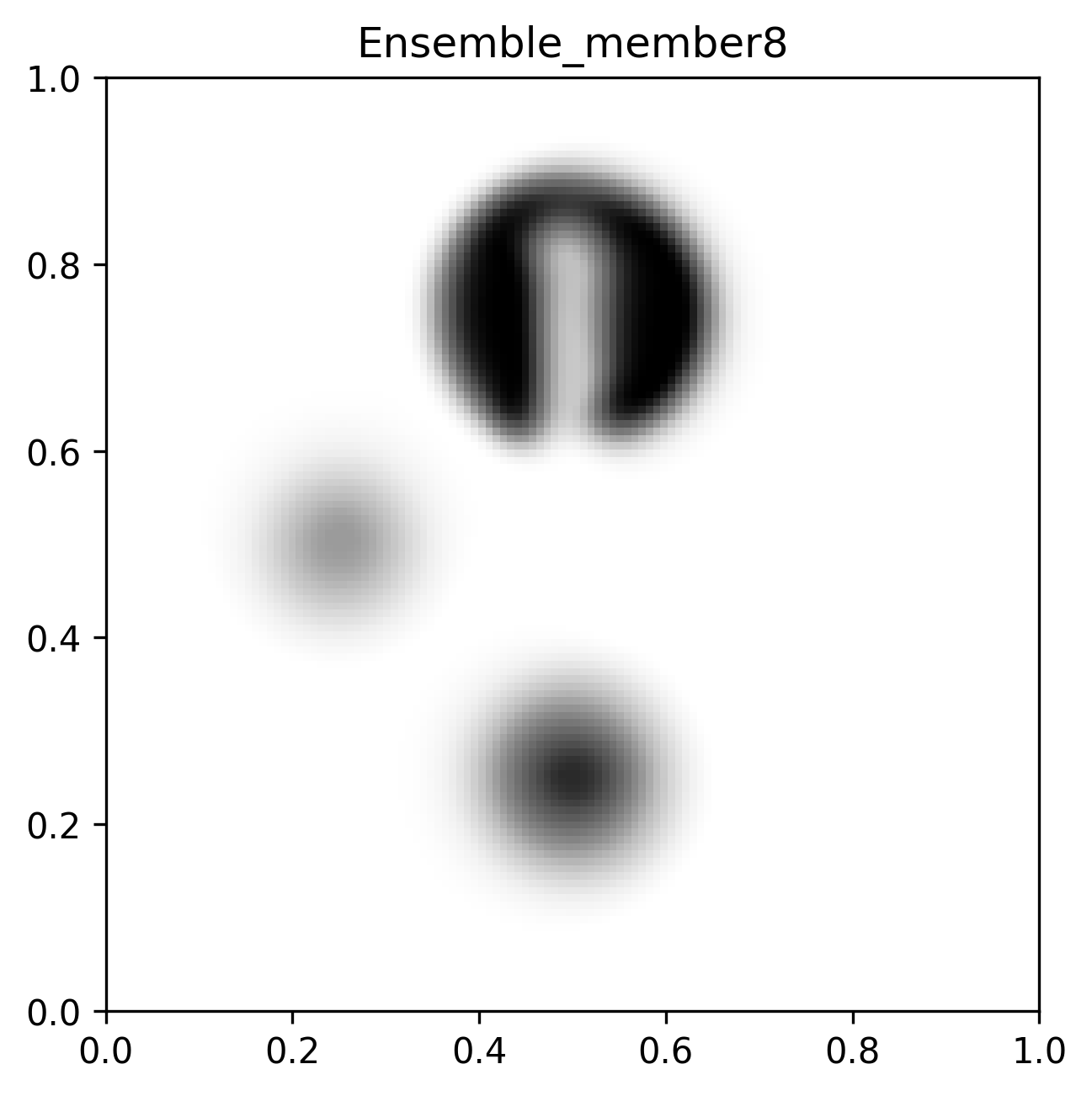}
\end{subfigure}
\begin{subfigure}[t]{0.1\textwidth}
\includegraphics[width=.95\textwidth]{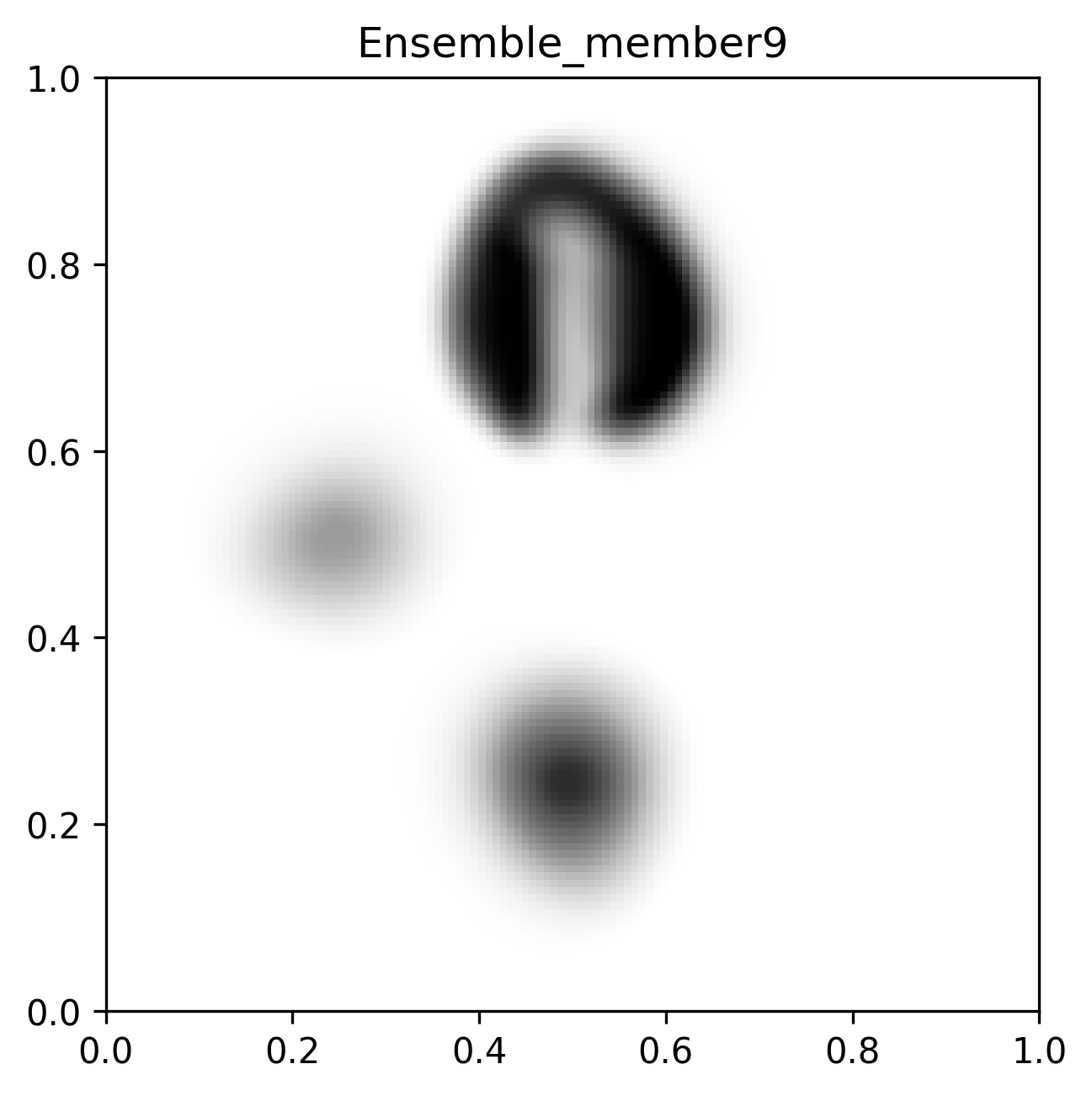}
\end{subfigure}
\begin{subfigure}[t]{0.1\textwidth}
\includegraphics[width=.95\textwidth]{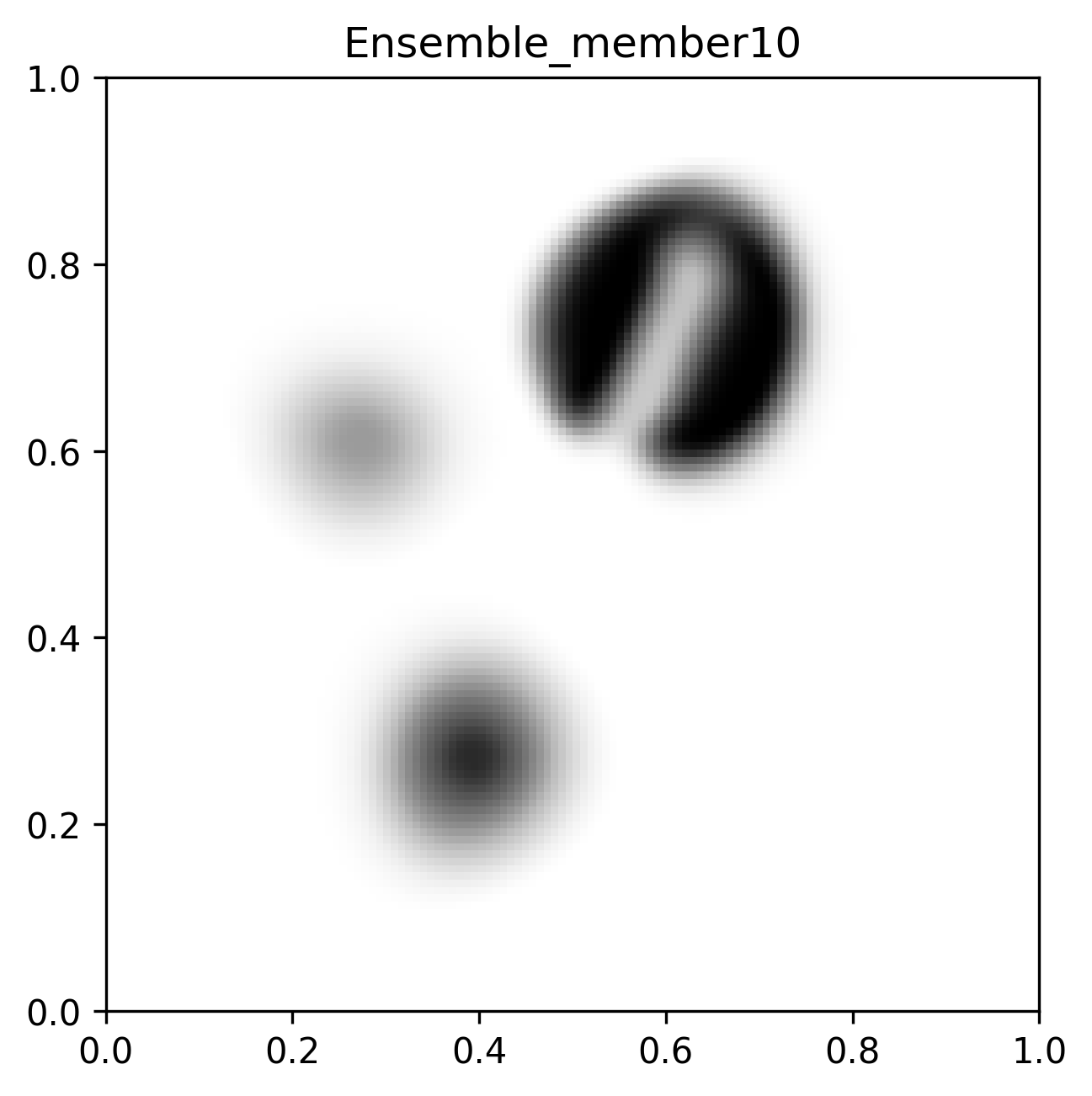}
\end{subfigure}
\begin{subfigure}[t]{0.1\textwidth}
\includegraphics[width=.95\textwidth]{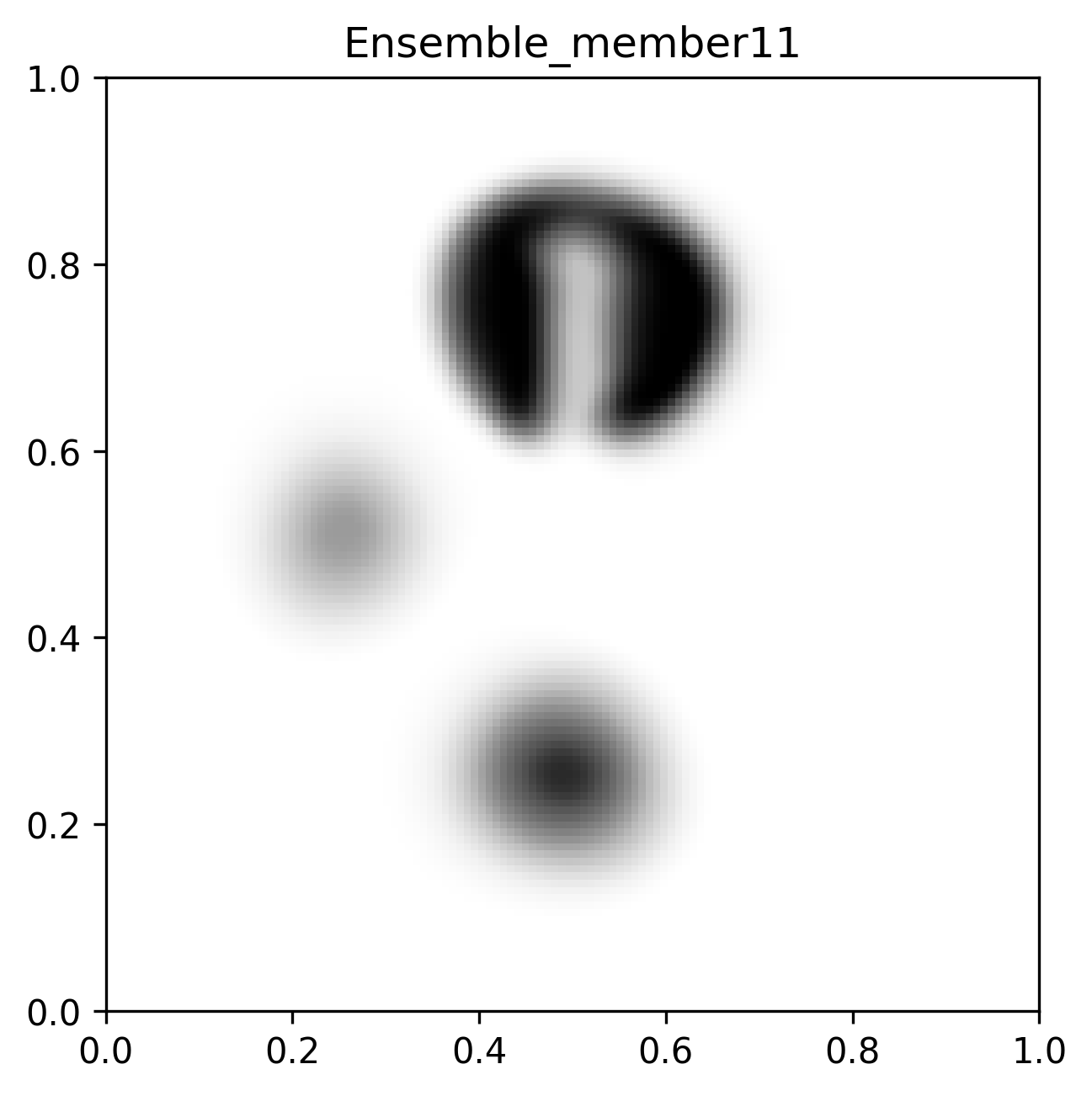}
\end{subfigure}
\begin{subfigure}[t]{0.1\textwidth}
\includegraphics[width=.95\textwidth]{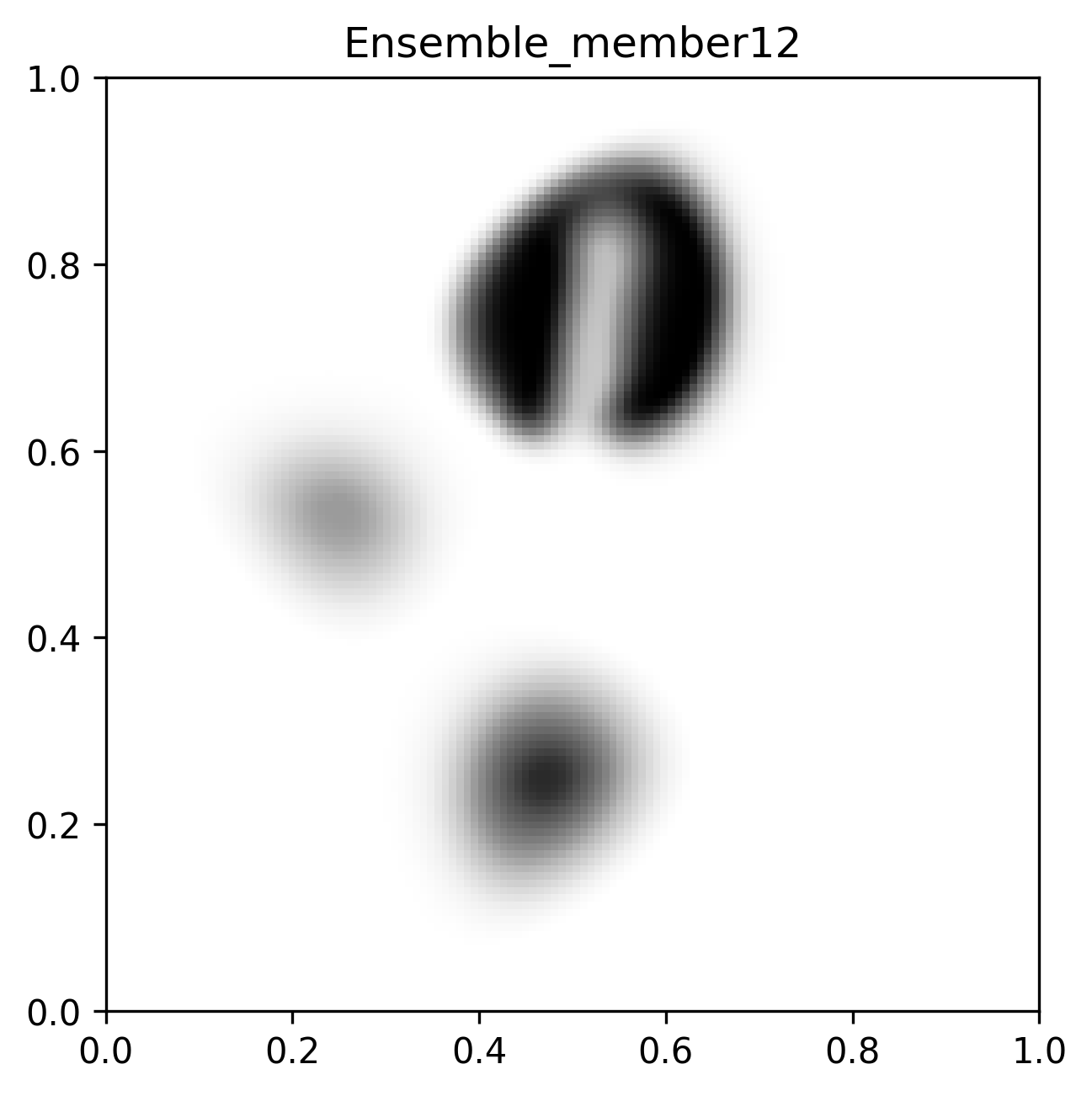}
\end{subfigure}
\begin{subfigure}[t]{0.1\textwidth}
\includegraphics[width=.95\textwidth]{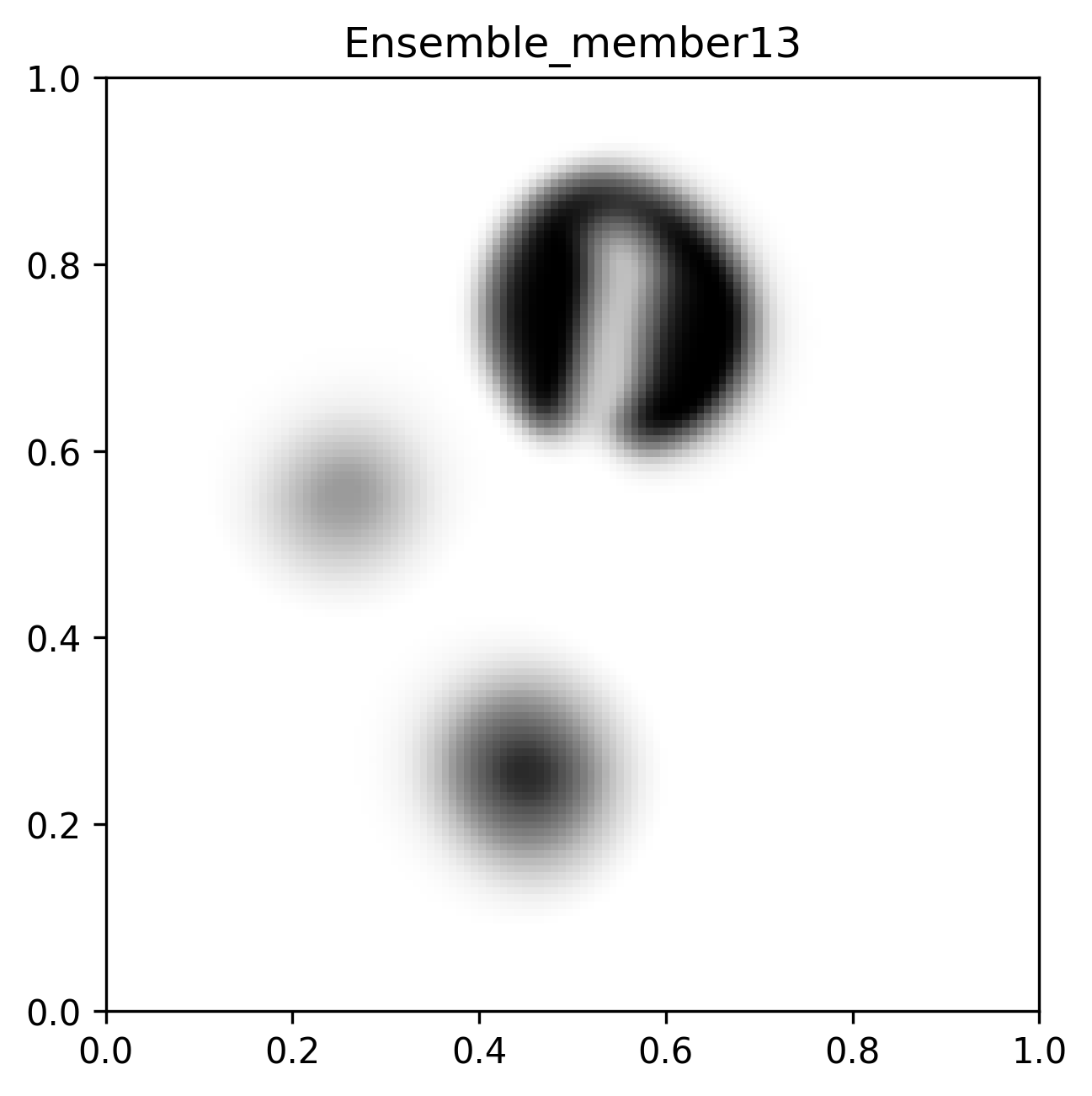}
\end{subfigure}
\begin{subfigure}[t]{0.1\textwidth}
\includegraphics[width=.95\textwidth]{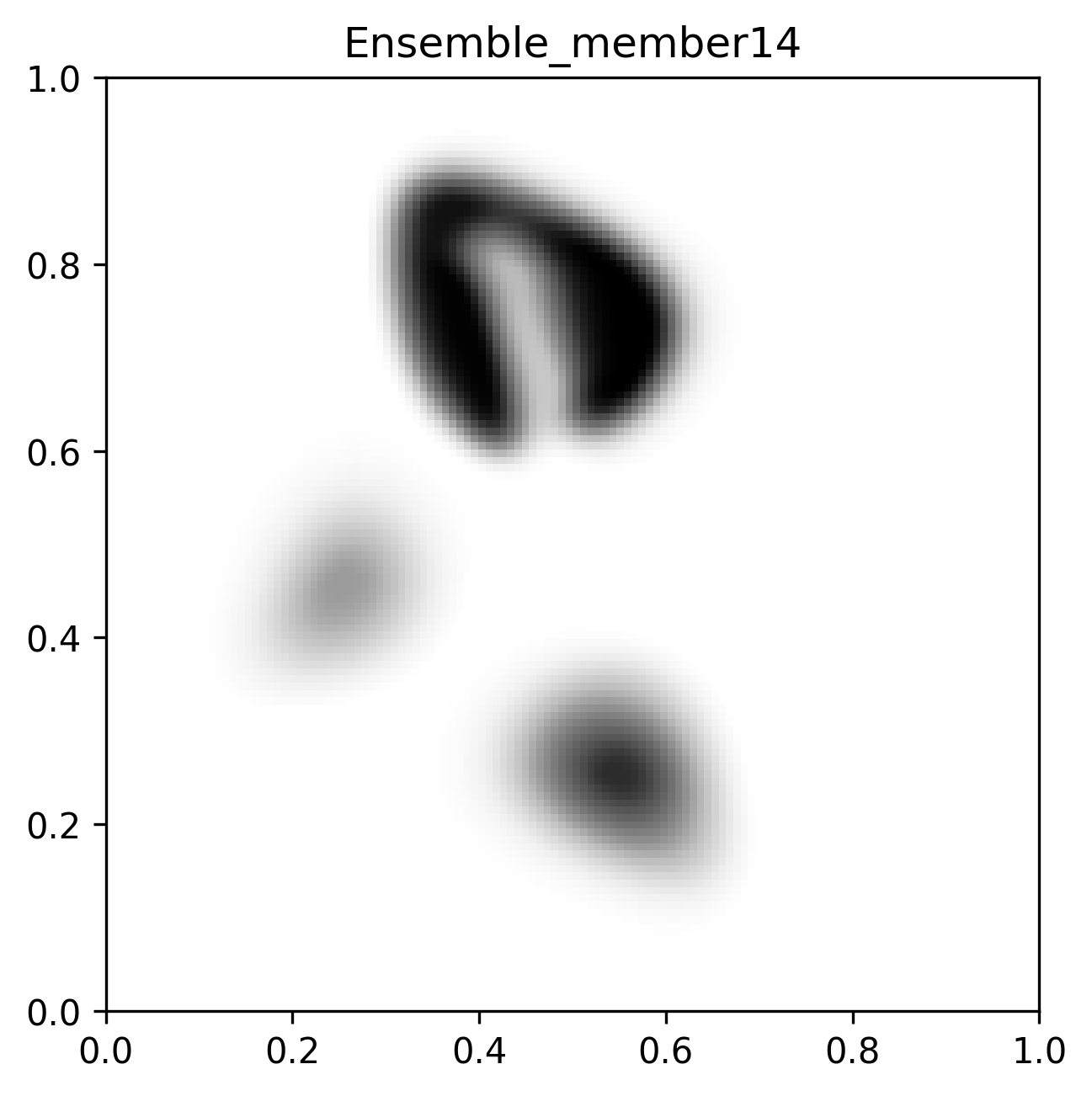}
\end{subfigure}
\begin{subfigure}[t]{0.1\textwidth}
\includegraphics[width=.95\textwidth]{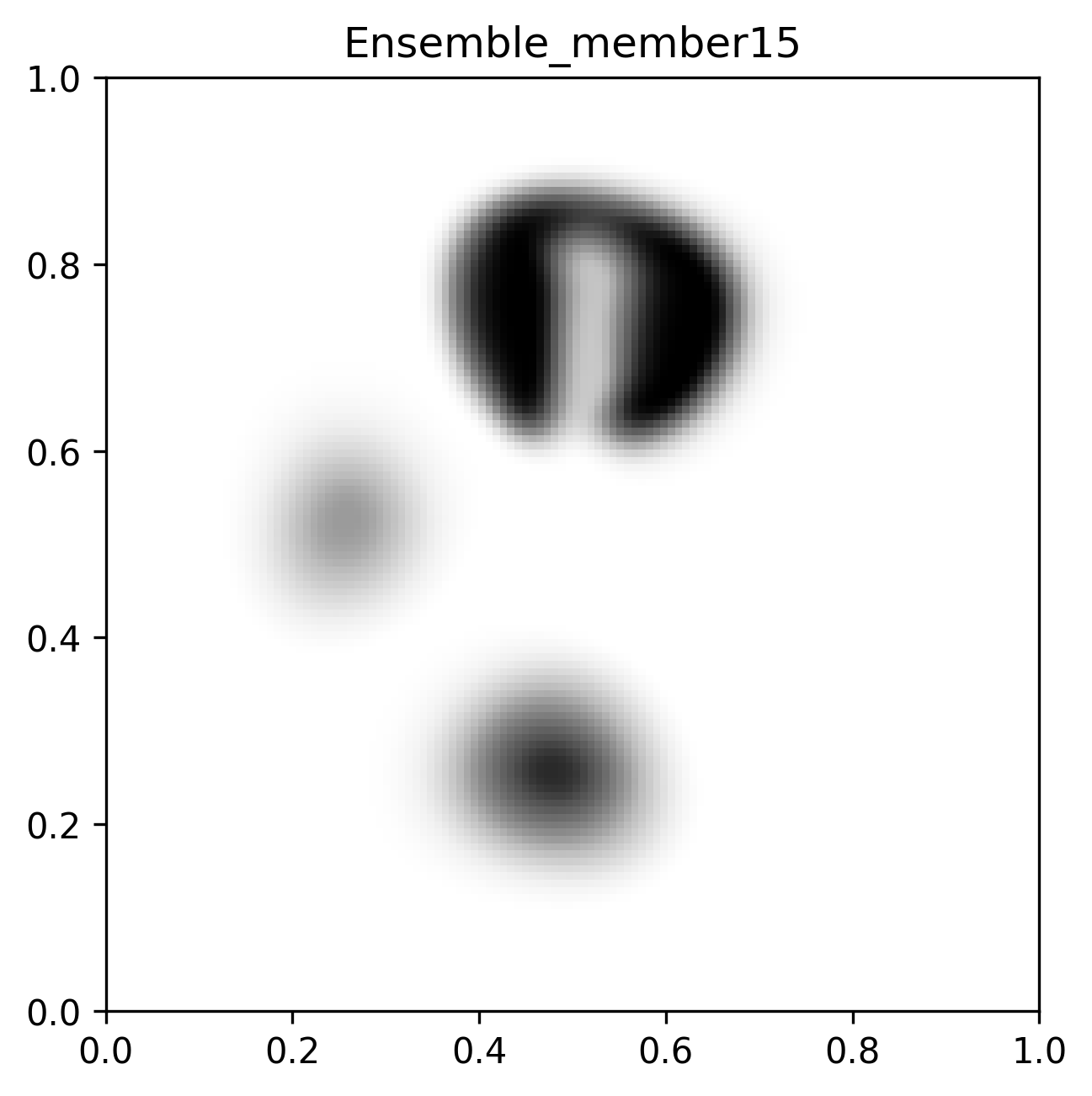}
\end{subfigure}
\caption{We plot 16 ensemble member solutions of the final timestep on a perceptually uniform grey colour scale between [0,1] undershoots are in blue overshoots are in red (there are no over/undershoots).}
\label{fig:sbr104}
\end{figure}

\subsubsection{Example 1d: Incompressible Euler}\label{Example 1d: Incompressible Euler}
In this experiment we solve the incompressible 2D Euler equation, subject to SALT (Stochastic Advection by Lie Transport)-noise \cite{holm2015variational}, the SPDE in vorticity form is given by
\begin{align}
d q+\operatorname{div}(\b u q) + \sum_{p=1}^{P}\div(\b \xi_p q) \circ dW^{p} &= 0,\quad \b u = - \nabla^{\perp} \psi,\quad \psi = -\Delta^{-1}q,\quad q(0,\b x) = q_0(\b x).
\end{align}
Where the stream-functions specifying the vector fields $\b \xi_p = -\nabla^{\perp} \Psi_p$, are chosen to be
\begin{align}
\Psi_{p}(x,y) &= 0.0001 \sin(2 p \pi x) \sin(2 p \pi y),\quad p = \lbrace 0,1,...,7 \rbrace.
\end{align}
Like the deterministic 2D incompressible Euler's equation, this SPDE preserves the local conservation of vorticity and remains bounded in $||\omega||_{L^{\infty}}$, and also possesses a discrete local maximum principle. To numerically locally preserve the vorticity and maintain a discrete local maximum principle for the EM scheme one can employ a standard C-Grid formulation with one-dimensional slope limiters (for incompressible flow) in \cite{woodfield2024new}. This is described below.
We solve the elliptic problem spectrally,
\begin{align}
\psi_{ij}^n = -\mathcal{F}^{-1}\left[(k_x^2 + k_y^2)^{-1}\mathcal{F}(q_{ij}^n) \right], 
\end{align}
where the Nihilist frequency is set manually to account for division by zero, and $2\pi i$ is absorbed into the wavenumber. The value of the streamfunction, is modified as follows $\psi = \psi  +\sum_{p=1}^{P} \Psi_p \Delta \widetilde{W}/\Delta t$. The cell-centered values are then projected to the cell corners
\begin{align}
\psi_{i+1/2,j+1/2}^n = (\psi_{i,j}^n+\psi_{i,j+1}^n+\psi_{i+1,j}^n+\psi_{i+1,j+1}^n)/4.
\end{align}
The normal component of the velocities at faces are computed using a standard C-grid approach 
\begin{align}
(u_{i+1/2,j}^n,v_{i,j+1/2}^n) = \left(- (\psi_{i,j+1/2}^n - \psi_{i,j-1/2}^n)/\Delta y, (\psi_{i+1/2,j}^n - \psi_{i-1/2,j}^n)/\Delta x \right).
\end{align}
Then one uses the advection algorithm described in \cite{woodfield2024higher} for the Euler-Maruyama step, with one-dimensional slope limiters. This described Euler Maruyama flow map (with bounded increments) provably has a discrete local maximum principle
\begin{align}
q_{i,j}^{n+1}\in [m_{i,j},M_{i,j}], \quad m_{i,j}:=\min_{(a,b) \in S_{i,j}} q^{n}_{(a,b)},\quad M_{i,j}:=\max_{(a,b) \in S_{i,j}} q^{n}_{(a,b)},
\end{align}
where $S_{i,j} = \lbrace (i,j),(i+1,j),(i,j+1),(i-1,j),(i,j-1)\rbrace$ is a 5point stencil. In this test case, we use the Le-Veque initial condition \cref{test:LeVeque}, and solve for
$E = 8$ ensemble members over the space time domain $[0,1]\times[0,1]\times [0,16]$, with periodic boundary conditions. The equation is solved at resolution $512\times 512\times 8192$, using the bounded increments in \cref{eq:3point_rv} and the SSP33 timestepping  \cref{method:ssp33}. The one dimensional slope limiter used is the Differentiable$(r)$ limiter specified in \cite{woodfield2024new}, suitable for strictly incompressible flows.

\begin{figure}[H]
\centering
% \begin{subfigure}[t]{0.1\textwidth}
% \includegraphics[width=.95\textwidth]{Figures/stream_function_0.pdf}
% \end{subfigure}
% \begin{subfigure}[t]{0.1\textwidth}
% \includegraphics[width=.95\textwidth]{Figures/stream_function_1.pdf}
% \end{subfigure}
% \begin{subfigure}[t]{0.1\textwidth}
% \includegraphics[width=.95\textwidth]{Figures/stream_function_2.pdf}
% \end{subfigure}
% \begin{subfigure}[t]{0.1\textwidth}
% \includegraphics[width=.95\textwidth]{Figures/stream_function_3.pdf}
% \end{subfigure}
% \begin{subfigure}[t]{0.1\textwidth}
% \includegraphics[width=.95\textwidth]{Figures/stream_function_4.pdf}
% \end{subfigure}
% \begin{subfigure}[t]{0.1\textwidth}
% \includegraphics[width=.95\textwidth]{Figures/stream_function_5.pdf}
% \end{subfigure}
% \begin{subfigure}[t]{0.1\textwidth}
% \includegraphics[width=.95\textwidth]{Figures/stream_function_6.pdf}
% \end{subfigure}
% \begin{subfigure}[t]{0.1\textwidth}
% \includegraphics[width=.95\textwidth]{Figures/stream_function_7.pdf}
% \end{subfigure}\\
\begin{subfigure}[t]{0.19\textwidth}
\includegraphics[width=.95\textwidth]{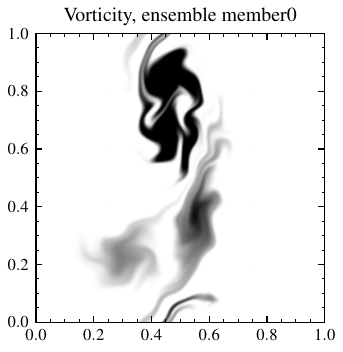}
\end{subfigure}
\begin{subfigure}[t]{0.19\textwidth}
\includegraphics[width=.95\textwidth]{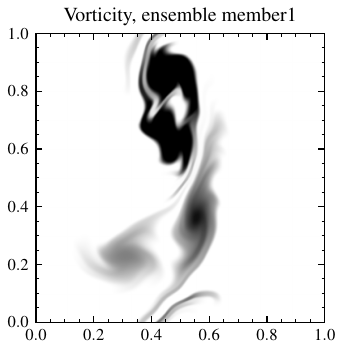}
\end{subfigure}
\begin{subfigure}[t]{0.19\textwidth}
\includegraphics[width=.95\textwidth]{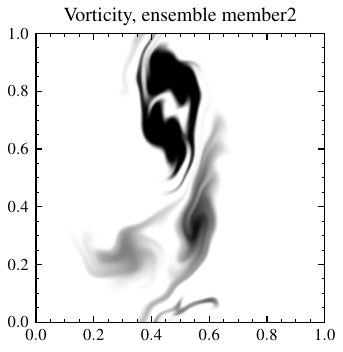}
\end{subfigure}
\begin{subfigure}[t]{0.19\textwidth}
\includegraphics[width=.95\textwidth]{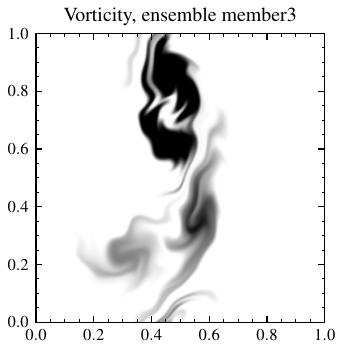}
\end{subfigure}
\begin{subfigure}[t]{0.19\textwidth}
\includegraphics[width=.95\textwidth]{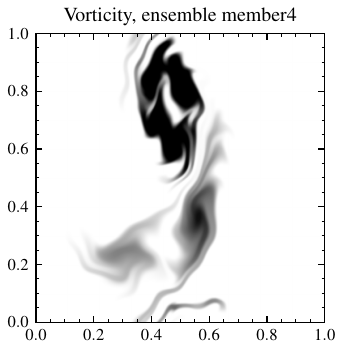}
\end{subfigure}\\
\begin{subfigure}[t]{0.19\textwidth}
\includegraphics[width=.95\textwidth]{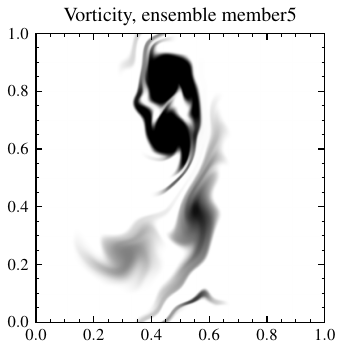}
\end{subfigure}
\begin{subfigure}[t]{0.19\textwidth}
\includegraphics[width=.95\textwidth]{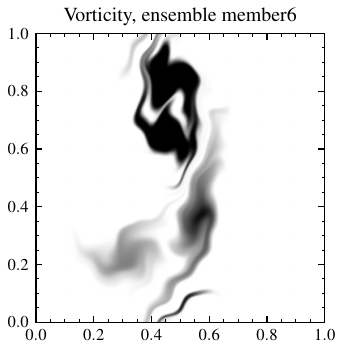}
\end{subfigure}
\begin{subfigure}[t]{0.19\textwidth}
\includegraphics[width=.95\textwidth]{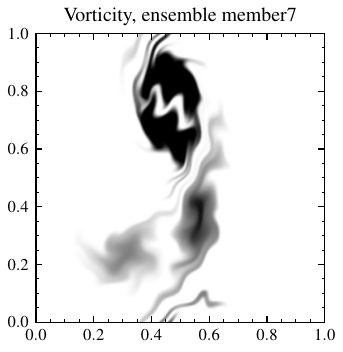}
\end{subfigure}
\begin{subfigure}[t]{0.19\textwidth}
\includegraphics[width=.95\textwidth]{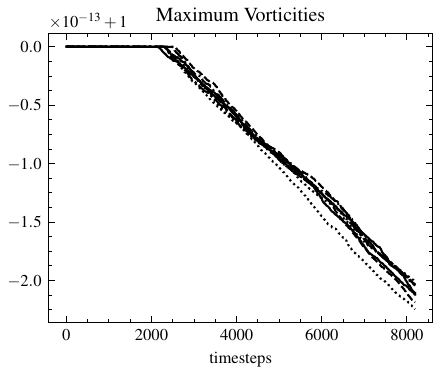}
\end{subfigure}
\begin{subfigure}[t]{0.17\textwidth}
\includegraphics[width=.95\textwidth]{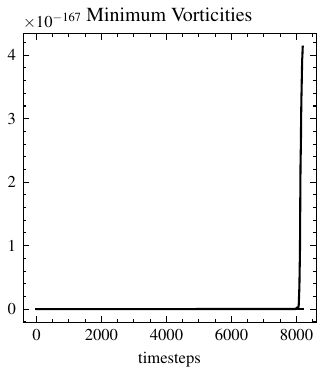}
\end{subfigure}
\caption{Stochastic-2D Euler. We plot 8 ensemble member solutions of the final timestep on a perceptually uniform grey colorscale between [0,1] undershoots are shown in blue and overshoots are shown in red (there are no over/undershoots). Maximums and minima of all ensemble members are plotted and remain bounded. }
\label{fig:salt_euler}
\end{figure}

\begin{figure}[H]
\centering
\begin{subfigure}[t]{0.29\textwidth}
\includegraphics[width=.95\textwidth]{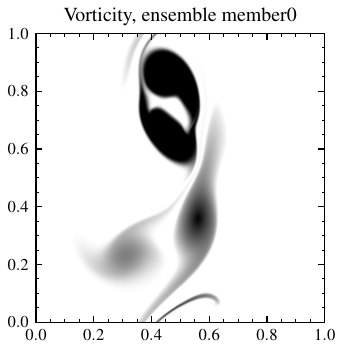}
\end{subfigure}
\begin{subfigure}[t]{0.28\textwidth}
\includegraphics[width=.95\textwidth]{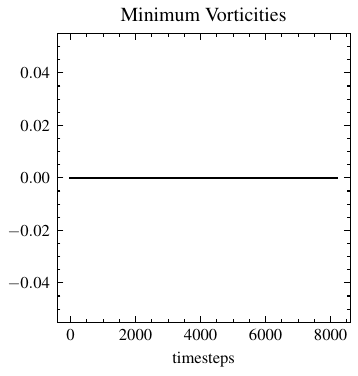}
\end{subfigure}
\begin{subfigure}[t]{0.33\textwidth}
\includegraphics[width=.95\textwidth]{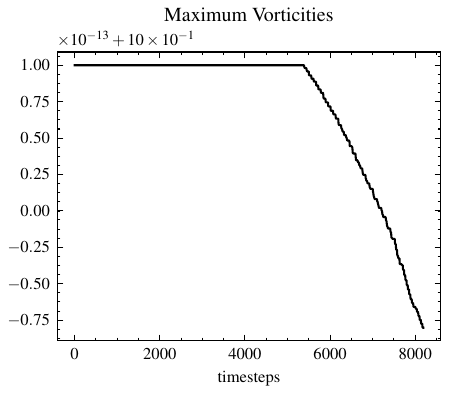}
\end{subfigure}
\caption{Deterministic 2D Euler. We plot a 1 ensemble member solution of the deterministic incompressible Euler equation at the final timestep on a perceptually uniform grey colorscale between [0,1] undershoots are shown in blue overshoots are shown in red (there are no over/undershoots). Maximums and minima are also plotted as a function of timestep.}
\label{fig:deterministic_euler}
\end{figure}

In \cref{fig:deterministic_euler}, we plot the solution to deterministic 2D Euler equation, where the slotted cylinder, cone and cosine lump rotate about themselves and one another. In \cref{fig:salt_euler} we plot 8 ensemble members of the solution to the SALT-2D-Euler equation, the stochastic solutions in \cref{fig:salt_euler} had additional small-scale dynamics due to the stochastic transport term. In \cref{fig:salt_euler}  all ensemble members remained range-bounded. As expected when using an EM scheme with a probable monotone property, bounded increments and an SSPRK method run below the radius of monotonicity.
Such stochastic transport models have been proposed for uncertainty quantification in \cite{holm2015variational,cotter2020particle}. However, range boundedness or notions of monotonicity have not been achieved numerically and may be advantageous. 

\subsubsection{Example 1e: Operator-Splitting GARK}\label{sec:example:1e operator splitting GARK}
Consider the following Burgers equation with transport type noise, 
\begin{align}
    d q_t+\left( (\frac{1}{2}q^2)_x + (\frac{1}{2}q^2)_y \right) dt + \operatorname{div}\left( \sum_{p=1}^{P}(\b \xi_{p}(\b x)q)\right)\circ dW_t= 0.
\end{align}
Designing an approximate Reimann solver suitable for such an SPDE is nontrivial, particularly when the basis of noise $\b \xi(\b x)$ varies in space. As a result, it may be difficult to construct a monotone numerical flux function and it may not be possible to prove that the Euler Maruyama scheme has a monotone property. Therefore, we have no monotonicity property for the SSP-SRK \cref{method: Stochastic Runge-Kutta} to inherit from the EM scheme. 

This example demonstrates how one can instead use the SSP-SGARK \cref{method: Stochastic Generalised Additive Runge-Kutta} method and operator splitting to split the problem into sub-problems. Both sub-problems 
\begin{align}
    d_t q + f dt = 0,\\
    d_t q + G \circ d\b W = 0,
\end{align}
can be solved with monotonic properties, whereas the combination may not (due to the lack of a well-defined approximate Reimann solver). The nonlinear term
\begin{align}
f = \left( (\frac{1}{2}q^2)_x + (\frac{1}{2}q^2)_y \right)
\end{align}
is treated as a deterministic drift and the well-established deterministic Godunov's flux \cref{method:godunov-method} is employed in each spatial direction. Whereas the advection operator \begin{align}
G \circ d\b S^p=\operatorname{div}\left( \sum_{p=1}^{P}(\b \xi_{p}(\b x)q)\right)\circ d\b S^p
\end{align}
will be treated with an established flux form upwind advection algorithm on a C-grid.  Then we can apply SGARK or SARK methodology to retain both monotonic properties, without trying to develop a Stochastic approximate Reimann solver as naively attempted previously. The Sequential operator splitting \Cref{method:SSP- Sequential Operator Splitting} was used, with $m=1$, and $n=4$, the three-point bounded random variable \cref{eq:3point_rv} was used for the increments, and the FV2 scheme with $N^2(K)\cup N(K)$-mp limiter was used for both subproblems. 

In this setup, we use the initial conditions in \cref{test:LeVeque} and use a single incompressible vector field,
\begin{align}
(\xi^x_1,\xi^y_2) = (2 \pi \sin(8 \pi x)\cos(8\pi y)/8,2 \pi \cos(8 \pi x) \sin(8 \pi y)/8 )
\end{align} 
which can generate local deformations at a specific length scale. The resolution is $128\times 128\times 256$, the space time domain is $[0,1]\times[0,1]\times[0,1/6]$, the boundary conditions are periodic.

\begin{figure}[H]
\centering
\begin{subfigure}[t]{0.1\textwidth}
\includegraphics[width=.95\textwidth]{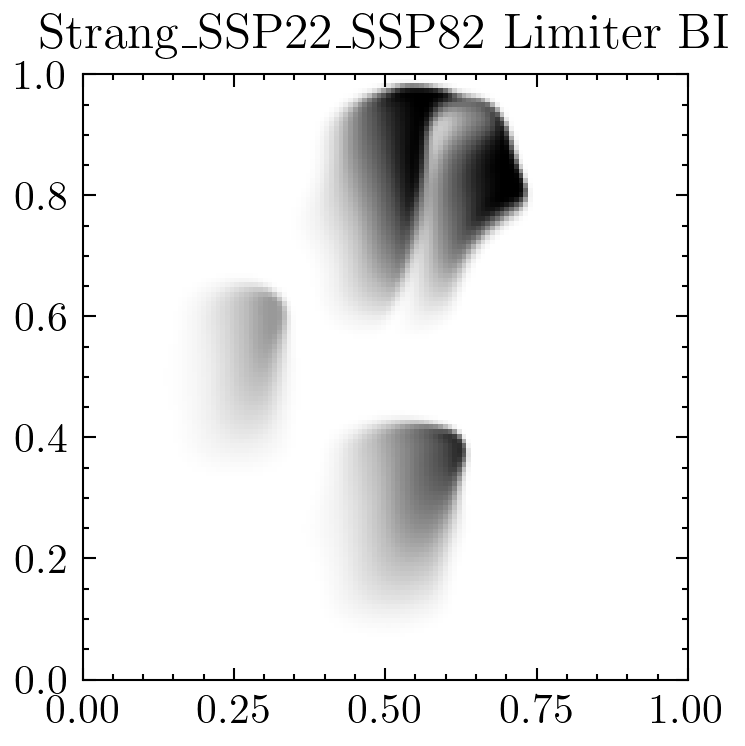}
\end{subfigure}
\begin{subfigure}[t]{0.1\textwidth}
\includegraphics[width=.95\textwidth]{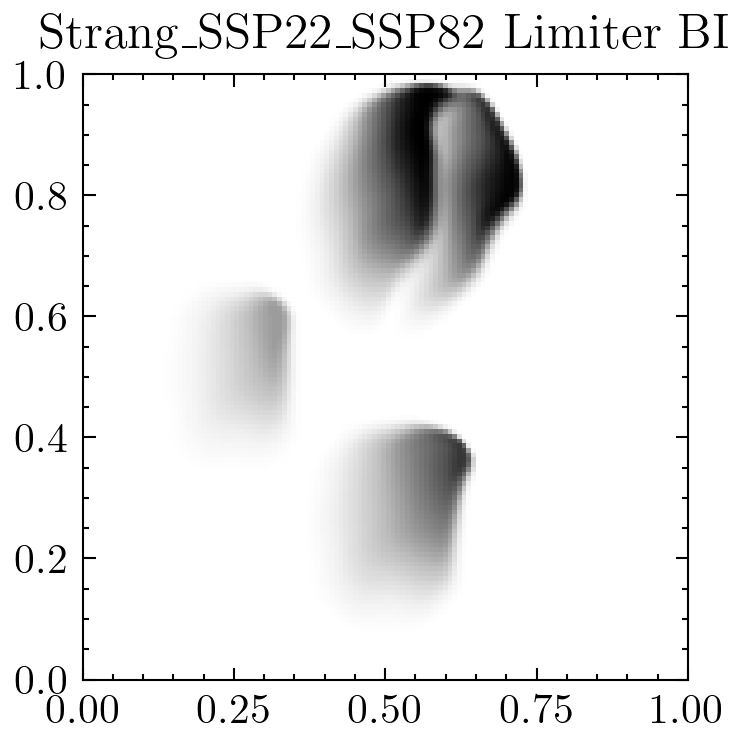}
\end{subfigure}
\begin{subfigure}[t]{0.1\textwidth}
\includegraphics[width=.95\textwidth]{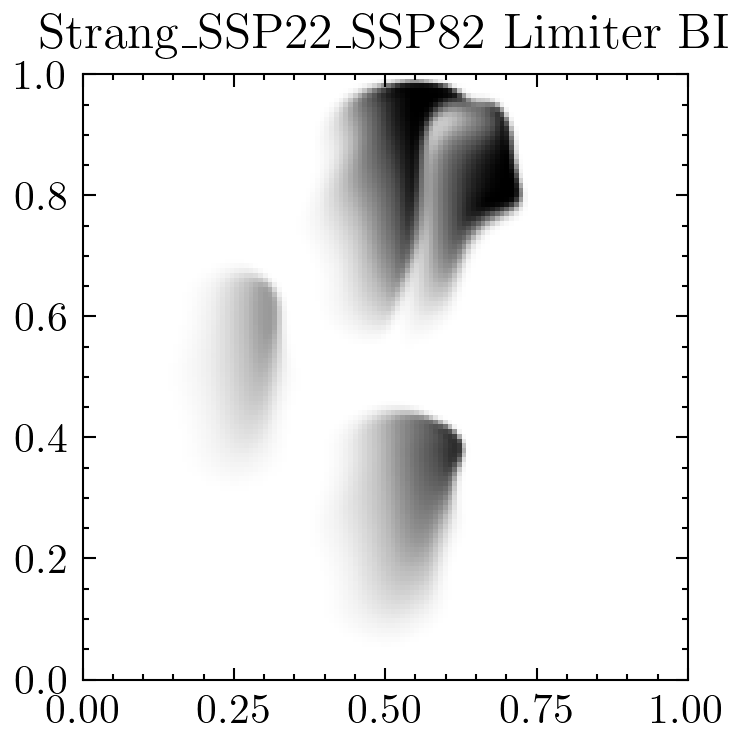}
\end{subfigure}
\begin{subfigure}[t]{0.1\textwidth}
\includegraphics[width=.95\textwidth]{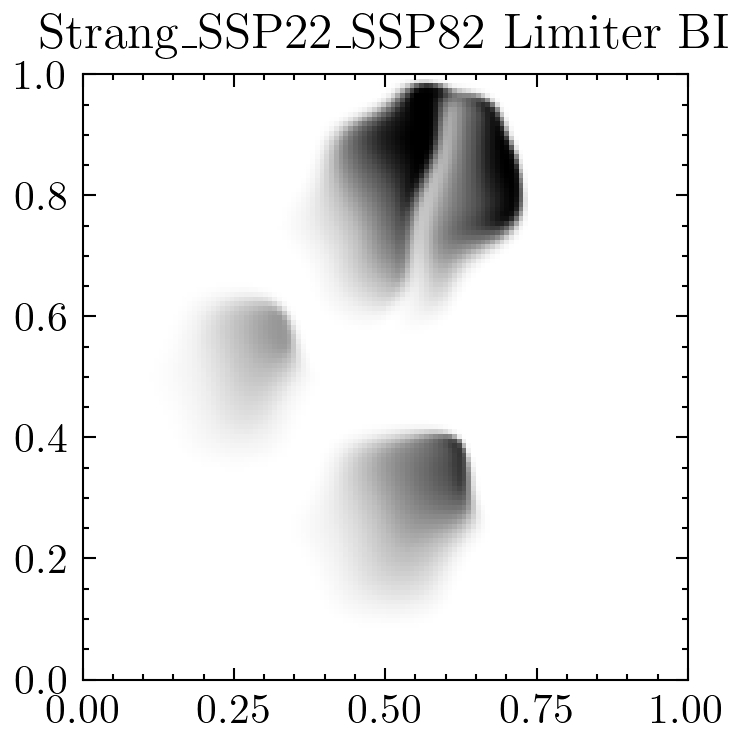}
\end{subfigure}
\begin{subfigure}[t]{0.1\textwidth}
\includegraphics[width=.95\textwidth]{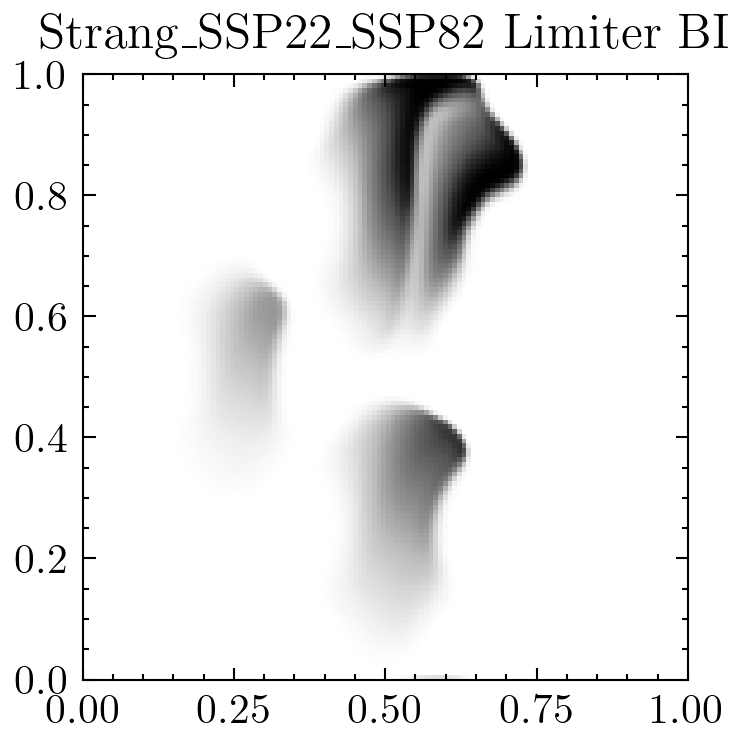}
\end{subfigure}
\begin{subfigure}[t]{0.1\textwidth}
\includegraphics[width=.95\textwidth]{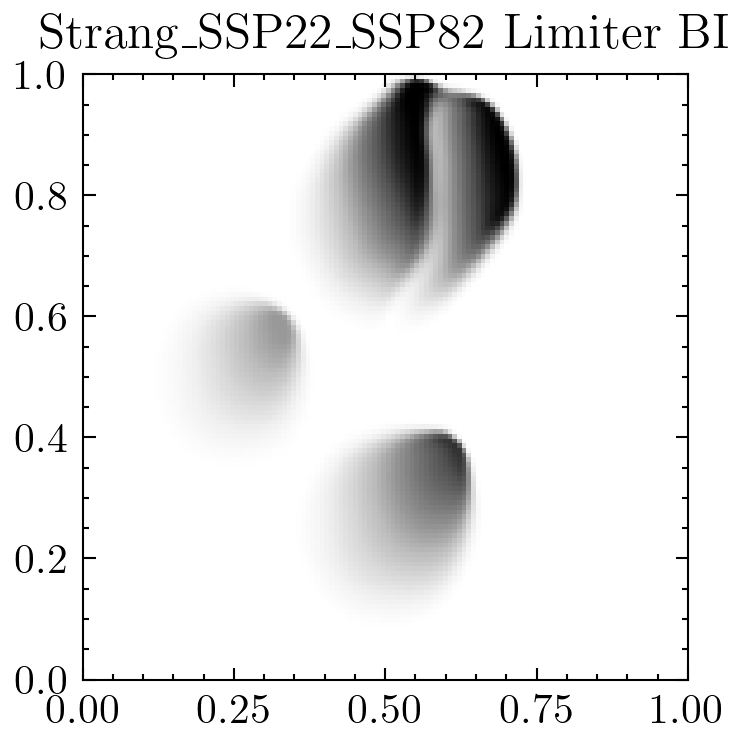}
\end{subfigure}
\begin{subfigure}[t]{0.1\textwidth}
\includegraphics[width=.95\textwidth]{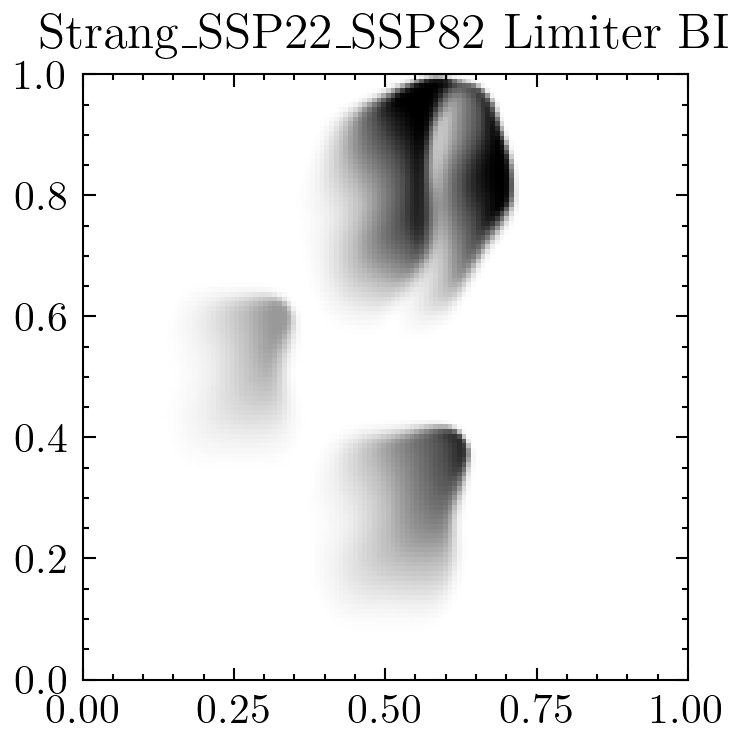}
\end{subfigure}
\begin{subfigure}[t]{0.1\textwidth}
\includegraphics[width=.95\textwidth]{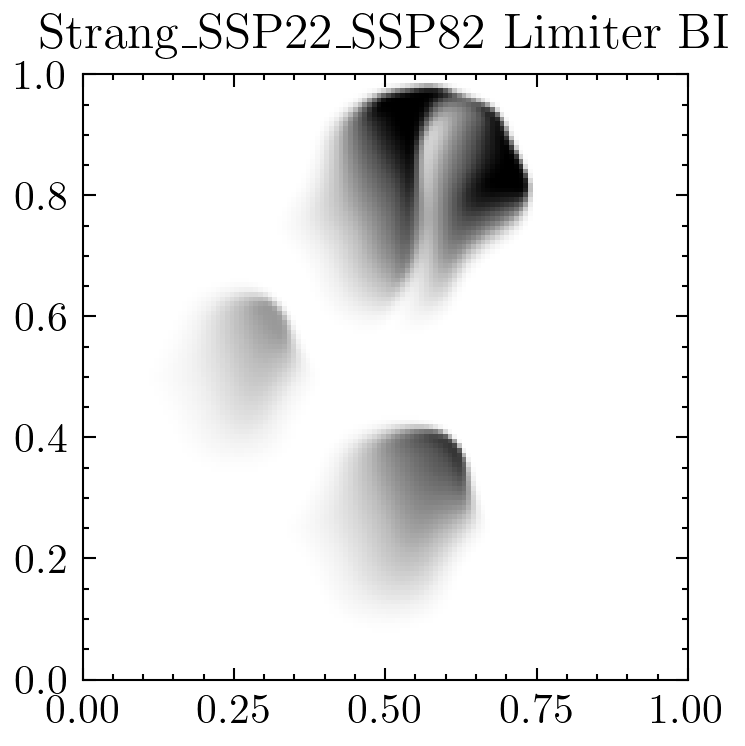}
\end{subfigure}\\
\begin{subfigure}[t]{0.1\textwidth}
\includegraphics[width=.95\textwidth]{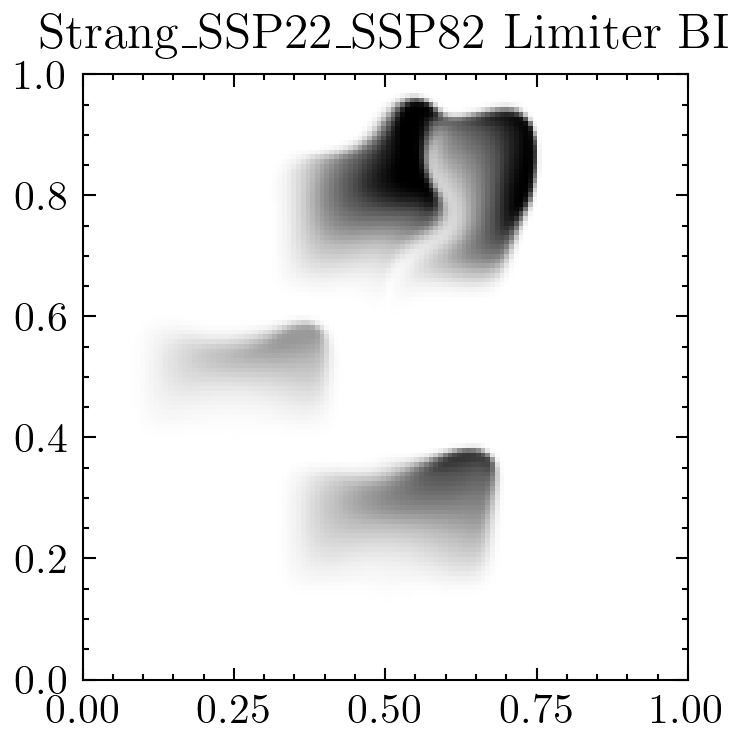}
\end{subfigure}
\begin{subfigure}[t]{0.1\textwidth}
\includegraphics[width=.95\textwidth]{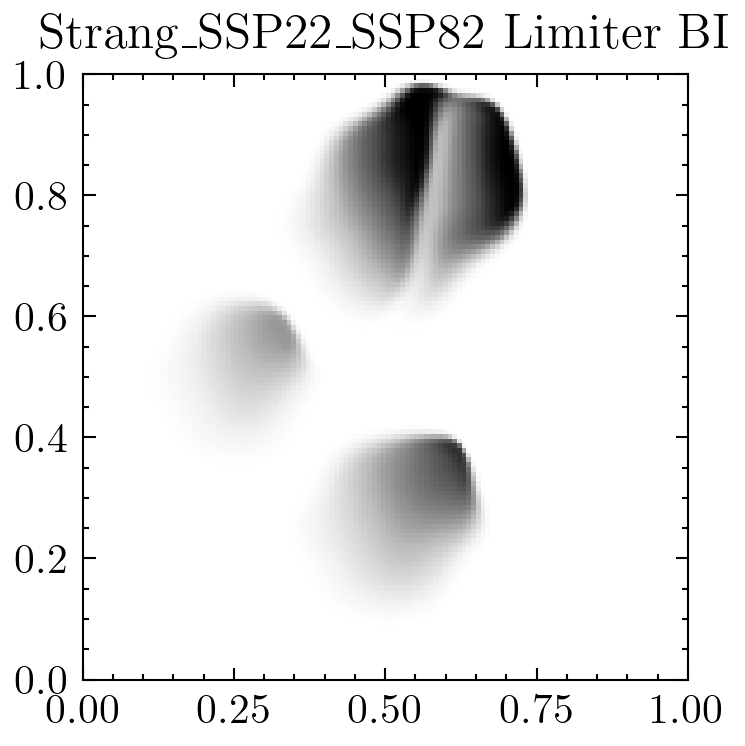}
\end{subfigure}
\begin{subfigure}[t]{0.1\textwidth}
\includegraphics[width=.95\textwidth]{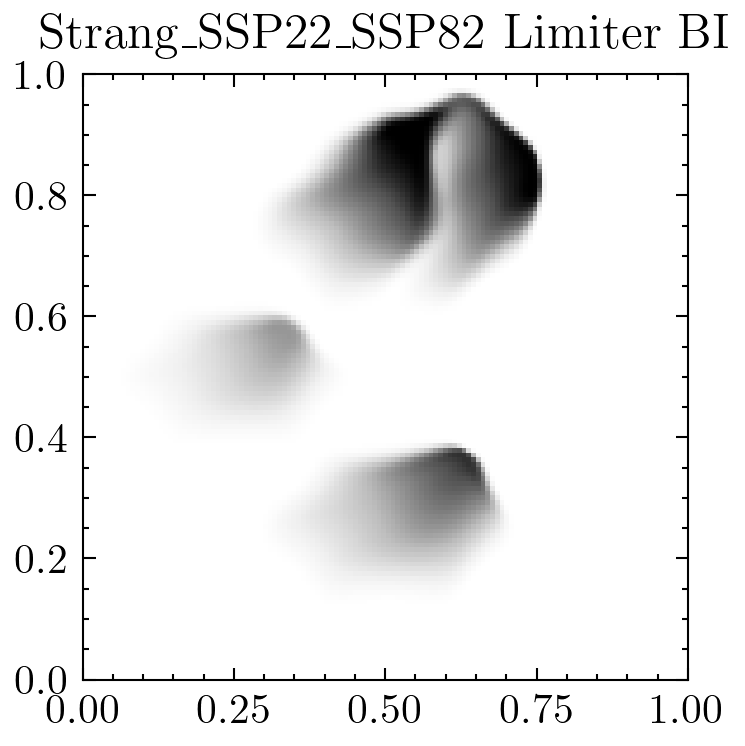}
\end{subfigure}
\begin{subfigure}[t]{0.1\textwidth}
\includegraphics[width=.95\textwidth]{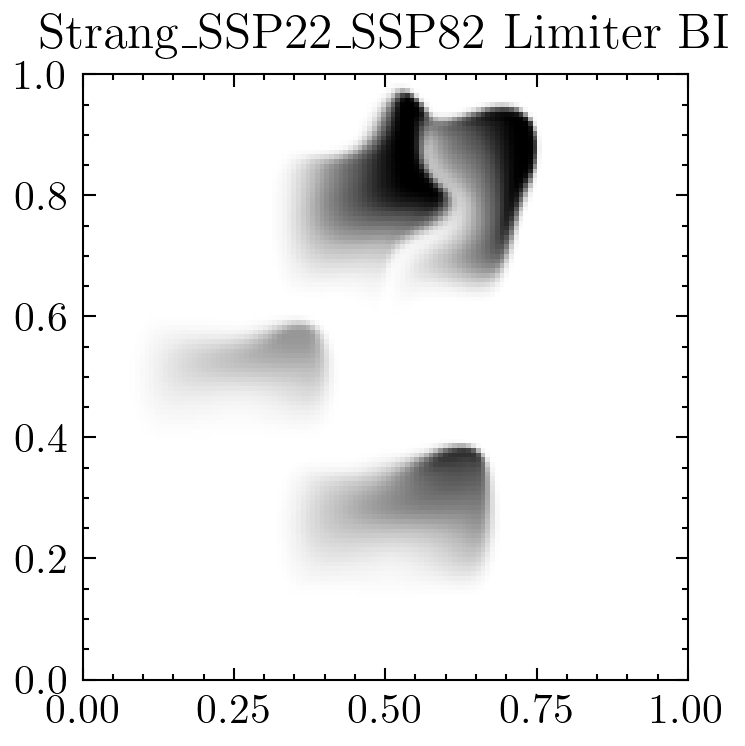}
\end{subfigure}
\begin{subfigure}[t]{0.1\textwidth}
\includegraphics[width=.95\textwidth]{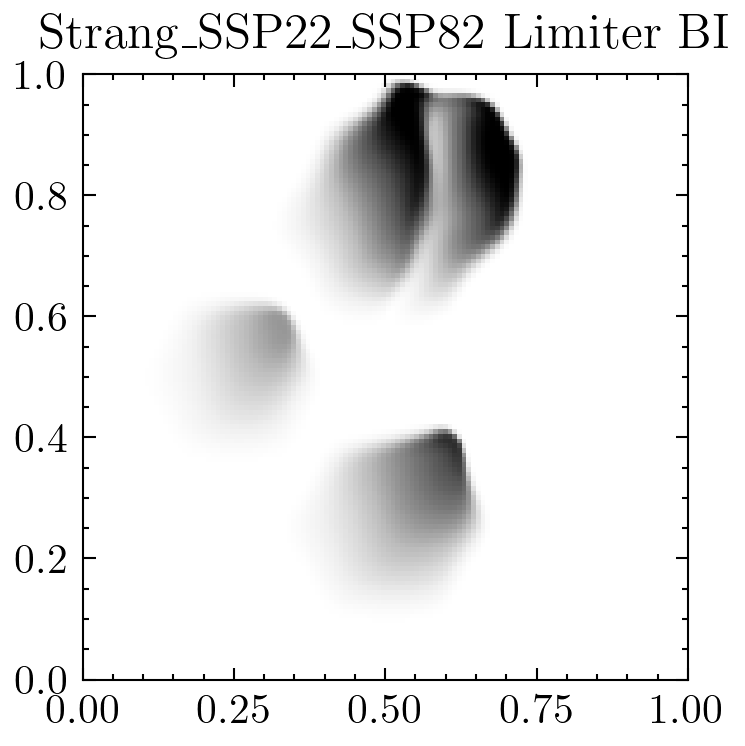}
\end{subfigure}
\begin{subfigure}[t]{0.1\textwidth}
\includegraphics[width=.95\textwidth]{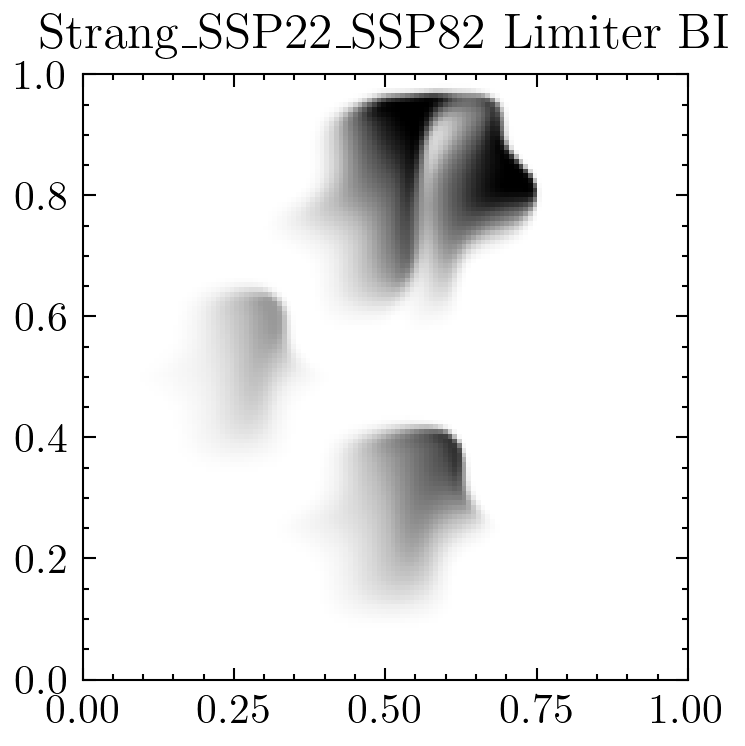}
\end{subfigure}
\begin{subfigure}[t]{0.1\textwidth}
\includegraphics[width=.95\textwidth]{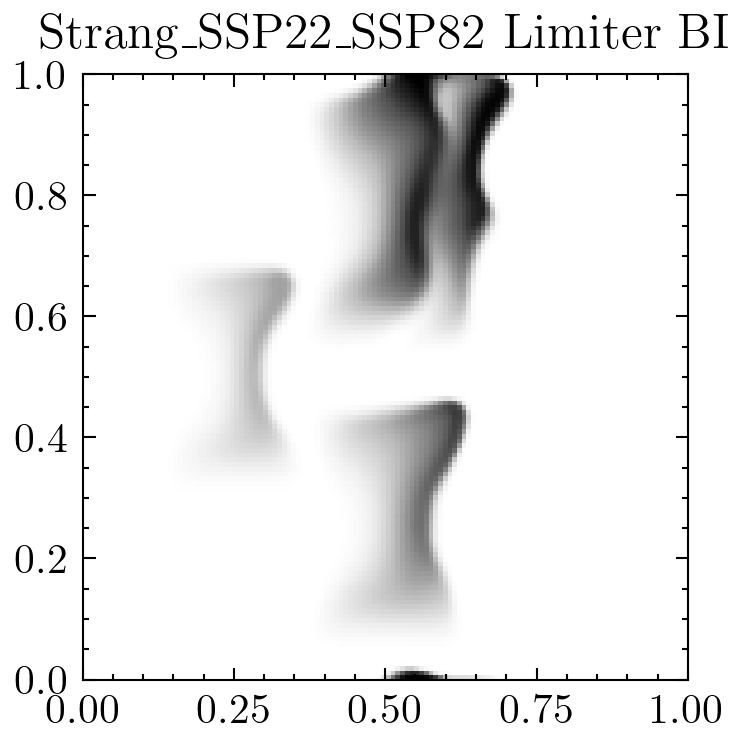}
\end{subfigure}
\begin{subfigure}[t]{0.1\textwidth}
\includegraphics[width=.95\textwidth]{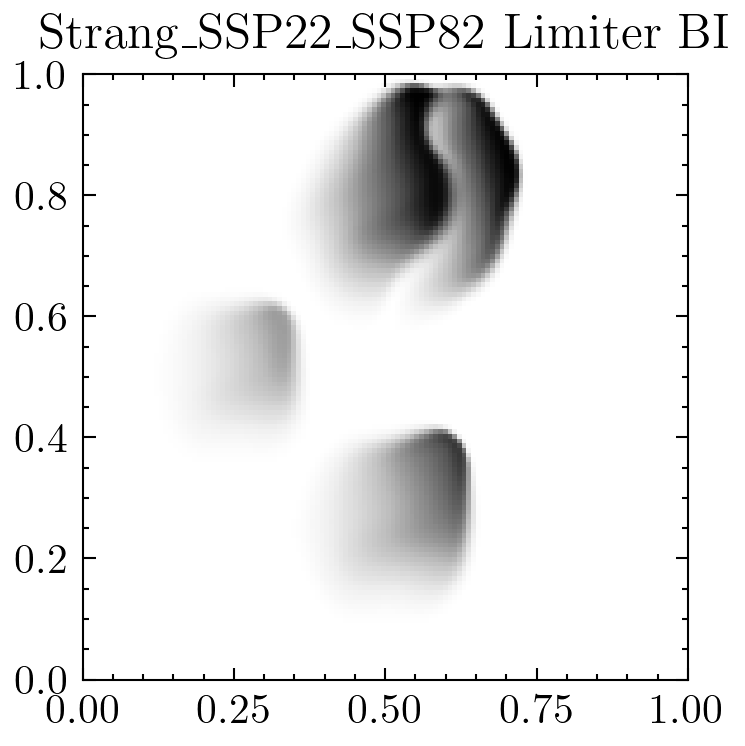}
\end{subfigure}
\begin{subfigure}[t]{0.1\textwidth}
\end{subfigure}
\caption{Burgers equation with deformational transport noise. We plot 16 ensemble member solutions of the final timestep on a perceptually uniform grey colorscale between [0,1] undershoots are in blue overshoots are in red (there are no over/undershoots).}
\label{fig:GARK_operatorsplit}
\end{figure}
The result in \cref{fig:GARK_operatorsplit} is a range-bounded solution to 2D burgers equation with stochastic transport noise. Demonstrating practical utility for \cref{method: Stochastic Generalised Additive Runge-Kutta}, indicating that one can use the ARK or GARK approach to attain a monotonic solution when the EM map may not be provably monotone.

\section{Conclusion}

This work indicates how one can apply deterministic notions of strong stability preservation in the context of stochastic equations. Firstly, monotone solving strategies are required for the solution of the Euler Maruyama map. 
Secondly, bounded increments are required for the notion of SSP to be understood in the stochastic setting. Thirdly, the deterministic proofs can be directly translated into the stochastic setting upon the appropriate identifications of EM and FE flow maps, allowing a range of SSP RK schemes to be applicable. In the first numerical demonstration in \cref{sec:example 1a burgers} we demonstrated the practical utility of all three aspects of the theory when numerically solving a 2D stochastic Burgers equation with constant transport noise. This numerical work required a stochastic extension of the Local-Lax-Friedrich numerical flux, and the use of a local maximum principle limiter in \cite{woodfield2024higher}. The second example in \cref{sec:Example 1b} demonstrated that some of the sufficient conditions for provably monotonic solutions are not always strictly necessary. 

In the Additive Runge-Kutta or Generalised Additive Runge-Kutta setting, we showed that one can inherit monotonic properties from an FE scheme and a diffusion-only EM scheme, as opposed to inheriting monotonic properties from the entire EM flow map. In \cref{sec:example:1e operator splitting GARK} we demonstrated that this could be practically used to generate bounded solutions to the 2D Burgers equation with non-constant spatially dependent transport noise. In this example, it was not obvious whether one could create a monotone EM scheme due to the lack of an appropriate Reimann solver, indicating practical utility to the strong stability-preserving stochastic extension to ARK and GARK schemes. \newline
In \cref{sec:convergence} we remark that stochastic schemes requiring bounded increments can converge weakly or strongly to the solution and be SSP. We prove mean square (strong) convergence order $1/2$ of SRK, SARK, methods when using the bounded increments in \cite{milstein2002numerical} and we remark how this proof extends to SGARK schemes.

Stochastic SSP methodology can be applied in a wide variety of contexts. Specifically motivating this work are SPDEs describing ocean and climate. In particular, the Stochastic Advection by Lie Transport framework  \cite{holm2015variational} produces SPDEs with nonlinear monotonic properties similar to their deterministic fluid PDE counterparts. In \cref{sec:Example 1c: 2D Advection,Example 1d: Incompressible Euler} it is demonstrated range boundedness and local maximum principles can be numerically attained for the advection equation and 2D incompressible Euler's equation with transport noise using monotonic solving strategies, bounded increments and an SSP timestepper.

The use of SSP methodology could be conjectured to have more broad applications to SDEs and SPDEs, where positivity preservation, contractive behaviour and other types of nonlinear stability are desirable.

\section*{Acknowledgements}
JW is especially grateful to R. Hu, for suggesting reading the work of G. N. Milstein, and M. V. Tretyakov.
Thankful to A. Lobbe, E. Fausti, and W. Pan for discussions regarding the particle filter. Discussions with D. D. Holm, L. Tianshi, H. Weller, C. Cotter, D. Crisan, O. Street, R. Wood.
The corresponding author has been supported during the present work by the European Research Council (ERC) Synergy grant ``Stochastic Transport in Upper Ocean Dynamics" (STUOD) -- DLV-856408.

%%Vancouver style references.
\bibliographystyle{abbrv}
\bibliography{refs.bib}

\appendix

% \begin{remark}[woody]
% Let the data $\lbrace Y_t \rbrace $arise from some  realisation of a SPDE
% \begin{align}
%     d Y^{\theta^*}_t = \beta_{\theta^*}(Y_t,t)dt + \Sigma_{\theta^*}(Y_t,t) dZ_t
% \end{align}
% with unknown $\theta^*$ parameter. For some realization $dZ^*$ of the signal. 
% Let the numerical method approximate
% \begin{align}
%     d X^{\theta}_t = \beta_{\theta}(X_t,t)dt + \Sigma_{\theta}(X_t,t) dZ_t
% \end{align}
% With the additive Runge-Kutta method $(A,b,c)$, $(\tilde{A},\tilde{b},\tilde{c})$, 
% \begin{align}
% X^{n+1} = \Phi(\Delta t, X^{n},\beta_{\theta},\Sigma_{\theta}, dW_t)  = \Phi^{[1]}(\Delta t, X^{n},\beta_{\theta}) + \Phi^{[2]}(\Delta t, X^{n+1},\Sigma_{\theta}, dW_t) 
% \end{align}
% We seek to minimise the following problem where one has access to the data 
% \begin{align}
%     J = \sum_{k=1}^{n}J_{k} := 
% \end{align}
% \appendix
\section{Appendix}
\subsection{Big O notation}
We understand the big-$\mathcal{O}(\Delta t)$-notation in the limit $\Delta t \rightarrow 0$, more specifically, we say $f(\Delta t) = \mathcal{O}(g(\Delta t))$ if there exists a constant $L>0$ such that $||f(\Delta t)||_2 \leq L ||g(\Delta t)||_2$ for all $\Delta t \in [0,\tau]$ sufficiently small.
\subsection{Multivariate Taylor's theorem}\label{sec:taylors-theorem}

% this is the more general case. 
% Multivariate Taylor's theorem - Let $f: \mathbf{R}^n \rightarrow \mathbf{R}$ be a $k$-times continuously differentiable function at the point $\b q \in \mathbf{R}^n$. Let $f: \mathbf{R}^n \rightarrow \mathbf{R}$ be $k+1$ times continuously differentiable in a closed compact ball around $\b q$, denoted $B(\b k,r)=\left\{\b k \in \mathbb{R}^n:||\b q -\b k || \leq r\right\}$ for some $r>0$.

% \begin{align}
% f(\b q+\b k)&=\sum_{|\alpha| \leq k} \frac{D^\alpha f(\b q)}{\alpha!}(k)^\alpha+\sum_{|\beta|=k+1} R_\beta(\boldsymbol{x})(k)^\beta \quad \text{Where} \quad R_\beta(\b q+\b k)=\frac{|\beta|}{\beta!} \int_0^1(1-\lambda)^{|\beta|-1} D^\beta f(\b q+\lambda \b k) d \lambda
% \end{align}
% continuity $(k+1)$-th  $B$, one immediately obtains the uniform estimates

% $$
% \left|R_\beta(\boldsymbol{x})\right| \leq \frac{1}{\beta!} \max _{|\alpha|=|\beta|} \max _{\boldsymbol{y} \in B}\left|D^\alpha f(\boldsymbol{y})\right|, \quad \boldsymbol{x} \in B
% $$

Multivariate Taylor's theorem - Let $f^k: \mathbb{R}^n \rightarrow \mathbb{R}$ be a $1$-times continuously differentiable function at the point $\b q \in \mathbb{R}^n$. Let $f^k: \mathbb{R}^n \rightarrow \mathbb{R}$ be $2$ times continuously differentiable in a closed compact ball around $\b q$, denoted $B_{\b q}(r)=\left\{\b y \in \mathbb{R}^n:||\b q -\b y || \leq r\right\}$ for some $r>0$. Then for $\b q+ \b a \in B_{q}(r)$, one has the following Taylors theorem
\begin{align}
f^k(\b q+\b a)&=f^k(\b q) +  D f^k|_{\b q} \b a+ R^k_f (\b q,\b a), \quad \text{where} \quad |R^k_{f}(\b q, \b a)| \leq L||\b a||_{2}^2.
\end{align}

Then in the vector-valued case, where $\b f$ is made of components $f^k$, one can write down the following component-wise Multivariate Taylor theorem of the following form, with an error estimate on the remainder 
\begin{align}
\b f(\b q+\b a) &= \b f(\b q) +  D \b f|_{\b q}\b a+ \b R_{\b f}(\b q,\b a), \quad\text{where}\quad ||\b R_{\b f}(\b q,\b a)||_2 \leq L||\b a||_2^2.
\end{align}
Similarly one can write this for the vector $\b g_p$, the $p$-th component of $G$.
\begin{align}
\b g_p(\b q+\b a) &= \b g_p(\b q) +  D \b g_{p} |_{\b q}(\b a)+ \b R_{\b g}(\b q,\b a), \quad\text{where}\quad ||\b R_{\b g}(\b q,\b a)||_2 \leq L||\b a||_2^2.
\end{align}
A matrix valued version can be written for $G$ in $\mathbb{R}^{d\times P}$, 
\begin{align}
G(\b q+\b a) &= G(\b q) +  D \b G|_{\b q}(\b a)+ \b R_{G}(\b q,\b a), \quad\text{where}\quad  ||\b R_{G}(\b q,\b a)||_2 \leq L||\b a||_2^2.
\end{align}

% \begin{remark}[Linear Stability] The Euler-Maruyama scheme is often said to be linearly stable when the solution of $q^{n+1}_{i} = q^{n}_{i} + a_{ij}q^{n}_{j}\Delta t + b_{pij}q_{j}^n \Delta W^p$ (summation over repeated index), satisfies $\lim_{n\rightarrow \infty}\mathbb{E}[(\b q^n)^2] = 0$. This can be calculated by taking the expectation through the product, and using properties of the normal distribution. To arrive at a notion of CFL condition. This notion of stability occurs in a limit, and in expectation. 
% \end{remark}

% \begin{definition}[Convergence order $p$ towards an ergodic Itô process with respect to the Ergodic criterion \cite{kloeden1992stochastic}]
% Let $\b q_n$ denote the numerical approximation of the stochastic process $\b q$ at time $t_n$, after $n$ steps of equal step-size $\Delta t$. We say that the numerical method is converging with respect to the ergodic criterion with order $p>0$ (towards an ergodic Itô process) if $\forall h \in C^{\infty}_P(\mathbb{R}^d,\mathbb{R})$, $\exists \delta>0, C_{f}>0$, s.t 
% \begin{align}
% |\frac{1}{nt}\sum_{n=0}^{nt-1} h(\b q_n)- \int_{\mathbb{R}^{d}}h(q)d \mu(q)|\leq C_f\Delta t^{p}, \quad \forall \Delta t \in (0,\delta).
% \end{align}
% \end{definition}

\subsection{Inequalities}\label{sec:inequalities}

We tabulate some additional estimates on the moments of the bounded normal increments in \cite{milstein2002numerical,milstein2004stochastic}. All odd moments are zero $\mathbb{E}[(\Delta \widetilde{Z})^{2m+1}] = 0$, $\forall m = \lbrace 1,...,M \rbrace$.  
The moment estimate in \cite{milstein2002numerical}, can be attained by computing 
\begin{align}
\mathbb{E}[(\Delta Z - \Delta \widetilde{Z})^2] = \frac{2}{\sqrt{2\pi}}\int_{A}^\infty (x-A)^2 e^{-x^2/2}dx = \frac{2}{\sqrt{2\pi}}\int_{0}^\infty y^2 e^{-y^2/2}e^{-Ay}dye^{-A^2/2},
\end{align}
Then since $y^2 e^{-y^2/2},e^{-Ay}\geq 0$, one can take the absolute value into the integral and apply Hölder's inequality with the $L^\infty,L^1$ norms on $(0,\infty)$. This allows the following upper-bound
\begin{align}
\mathbb{E}[(\Delta Z - \Delta \widetilde{Z})^2] = e^{-A^2/2}\frac{2}{\sqrt{2\pi}}\int_{0}^\infty y^2 e^{-y^2/2}e^{-Ay}dy\leq e^{-A^2/2}\mathbb{E}[\Delta Z^2]||e^{-Ay}||_{L^{\infty}(0,\infty)}= e^{-A^2/2}.
\end{align}

Similarly, we have that higher-order moments are similarly bounded,
\begin{align}
\mathbb{E}[(\Delta Z - \Delta \widetilde{Z})^{2m}] = \frac{2}{\sqrt{2\pi}}\int_{A}^\infty (x-A)^{2m} e^{-x^2/2}dx = \frac{2}{\sqrt{2\pi}}\int_{0}^\infty y^{2m} e^{-y^2/2}e^{-Ay}dye^{-A^2/2} \leq \mathbb{E}[\Delta Z^{2m}]e^{-A^2/2}.
\end{align}
Taking $A = \sqrt{2k|\ln(\Delta t)|}$, as in \cref{eq:bounded increments}, gives
odd moments $\mathbb{E}[(\Delta Z - \Delta \widetilde{Z})^{2m+1}] = 0$ and even moments $\mathbb{E}[(\Delta Z - \Delta \widetilde{Z})^{2m}] \leq \frac{(2m)!\Delta t^{k}}{2^m m!}$. Using the Law of Total Expectation, to condition the expectation over ($|x|>A$, $x<A$, $x<-A$), one can compute
\begin{align}
    \mathbb{E}[(\Delta Z^2 - \Delta \widetilde{Z}^2)(\Delta Z - \Delta \widetilde{Z})] &= \frac{1}{\sqrt{2\pi}} 
    \int_{a}^{\infty}(x^2-a^2)(x-a)e^{-x^2/2}dx +\frac{1}{\sqrt{2\pi}} 
    \int_{-\infty}^{-a}(x^2-a^2)(x+a)e^{-x^2/2}dx \\
    &= 0.
\end{align}

\begin{method}[Burgers Godunov solver]\label{method:godunov-method}
If one has a 1d edge discontinuity, where $q_l,q_r$ denote the values to the left and right of the discontinuity.  Burgers equation has a convex flux function $f(q) = q^2/2$. The solutions are shock solutions if $q_l>q_r$, or a rarefaction fan solution if $q_l<q_r$. This allows one to write down the numerical flux at each cell boundary as 
\begin{align}
F_{i+1/2} = 
\begin{cases} 
F_{shock}, \quad (q_l > q_r),\\
F_{rarefraction}, \quad (q_l < q_r).
\end{cases}
\end{align}
Where the shock and rarefaction solutions are respectively given by  
\begin{align}
F_{shock} =
\left\lbrace
\begin{aligned}
 f(q_l),\quad & s>0\\
 f(q_r),\quad & s<0
\end{aligned}
\right\rbrace
,
\quad F_{rarefaction} = 
\left\lbrace
\begin{aligned}
f(q_l), \quad &(q_l>0)\\
0,   \quad &(q_l<0)\\
f(q_r), \quad &(q_r<0)\\
0,   \quad &(q_r>0)
\end{aligned}
\right\rbrace,
\end{align}
here the shock speed $s$ is specified by the Rankine-Hugoniot condition
$$ s= \frac{1/2 q_l^2 - 1/2 q_r^2}{q_l - q_r} = \frac{q_l+q_r}{2}.$$
\end{method}

\end{document}